\documentclass[11pt]{article}
\usepackage[latin1]{inputenc}
\usepackage{amsmath,amsthm,amssymb}
\usepackage{amsfonts}
\usepackage{amsmath,amsthm,amssymb,amscd}
\usepackage{latexsym}
\usepackage{color}
\usepackage{graphicx}
\usepackage{mathrsfs}
\usepackage{cite}

\makeatletter \@addtoreset{equation}{section} \makeatother

\newtheorem{theorem}{Theorem}[section]

\newtheorem{lemma}{Lemma}[section]
\newtheorem{remark}{Remark}[section]

\newtheorem{corollary}[theorem]{Corollary}
\newcommand{\R}{\mathbb{R}}

\begin{document}

\title
{\bf Normalized solutions for Schr\"{o}dinger system with quadratic and cubic interactions}

\date{}

\maketitle
\begin{center}
\author{\bf {Xiao Luo}}
\footnote{Email addresses: luoxiao@hfut.edu.cn (Luo).}
\author{\bf{Juncheng Wei}}
\footnote{Email addresses: jcwei@math.ubc.ca (Wei).}
\author{\bf{Xiaolong Yang}}
\footnote{Email addresses: yangxiaolong@mails.ccnu.edu.cn (Yang).}
\author{\bf{Maoding Zhen}}
\footnote{Email addresses: maodingzhen@163.com (Zhen).}
\end{center}
\begin{center}
\footnotesize {1 School of Mathematics, Hefei University of Technology, Hefei, 230009, P. R. China}\\
\footnotesize {2 Department of Mathematics, University of British Columbia, Vancouver, B.C.,
V6T 1Z2, Canada}\\
\footnotesize {3 School of Mathematics and Statistics, Central China Normal University, Wuhan, 430079, P. R. China}\\
\footnotesize {4 School of Mathematics, Hefei University of Technology, Hefei, 230009, P. R. China}\\
\end{center}

\begin{abstract}
{In this paper, we give a complete study on the existence and non-existence of solutions to the following mixed coupled nonlinear Schr\"{o}dinger system
\begin{equation*}
\begin{cases}
-\Delta u+\lambda_{1}u=\beta uv+\mu_{1}u^{3}+\rho v^{2}u  & \text{in} \ \ \mathbb{R}^{N},\\
-\Delta v+\lambda_{2}v= \frac{\beta}{2}u^{2}+\mu_{2}v^{3}+\rho u^{2}v& \text{in} \ \ \mathbb{R}^{N},
\end{cases}
\end{equation*}
under the normalized mass conditions $\int_{\mathbb{R}^{N}}u^{2}dx=b^{2}_{1}$  and $\int_{\mathbb{R}^{N}}v^{2}dx=b^{2}_{2}$.
Here $b_1, b_2>0$ are prescribed constants, $N\geq 1$, $\mu_{1}, \mu_{2}, \rho>0$, $\beta\in \mathbb{R}$ and the frequencies $\lambda_{1},\lambda_{2}$ are unknown and will appear as Lagrange multipliers. In the one dimension case, the energy functional is bounded from below on the product of $L^2$-spheres, normalized ground states exist and are obtained as global minimizers.  When $N=2$, the energy functional is not always bounded on the product of $L^2$-spheres. We give a classification of the existence and nonexistence of global minimizers. Then under suitable conditions on $b_1$ and $b_2$, we prove the existence of normalized solutions.  When $N=3$, the energy functional is always unbounded on the product of $L^2$-spheres. We show that under suitable conditions on $b_1$ and $b_2$, at least two normalized solutions exist, one is a ground state and the other is an excited state. Furthermore, by refining the upper bound of the ground state energy, we provide a precise mass collapse behavior of the ground state and a precise limit behavior of the excited state as $\beta\rightarrow 0$. Finally, we deal with the high dimensional cases $N\geq 4$. Several non-existence results are obtained if $\beta<0$.  When $N=4$, $\beta>0$, the system is a mass-energy double critical problem, we obtain the existence of a normalized ground state and  its synchronized mass collapse behavior. Comparing with the well studied homogeneous case $\beta=0$, our main results indicate that the quadratic interaction term not only enriches the set of solutions to the above Schr\"{o}dinger system but also leads to a stabilization of the related evolution system.
}\medskip

\medskip

\end{abstract}

\section{Introduction and main results}

In this paper, we look for solutions to the following coupled Schr\"{o}dinger system
\begin{equation}\label{23}
\begin{cases}

-\Delta u+\lambda_{1}u=\beta uv+\mu_{1}u^{3}+\rho v^{2}u  & \text{in} \ \ \mathbb{R}^{N},\\

-\Delta v+\lambda_{2}v= \frac{\beta}{2}u^{2}+\mu_{2}v^{3}+\rho u^{2}v& \text{in} \ \ \mathbb{R}^{N},
\end{cases}
\end{equation}
satisfying the additional constraints
\begin{equation}\label{24}
\int_{\mathbb{R}^{N}}u^{2}dx=b^{2}_{1}\ \text{and} \ \int_{\mathbb{R}^{N}}v^{2}dx=b^{2}_{2}.
\end{equation}
Here, $b_{1}, b_{2}>0$ are prescribed constants, $N\geq 1$, $\mu_{1}, \mu_{2}, \rho>0$, $\beta\in \mathbb{R}$ and the frequencies $\lambda_{1},\lambda_{2} $ are unknown and will appear as Lagrange multipliers.

Problem \eqref{23}-\eqref{24} arises from the research of stationary exponentially localized bright solitary waves
for the following two-wave mixing system
\begin{equation}\label{H}
\begin{cases}

i\frac{\partial \Phi_{1}}{\partial z}+\Delta \Phi_{1}+\beta \Phi^{\ast}_{1}\Phi_{2}+s(|\Phi_{1}|^{2}+\rho|\Phi_{2}|^{2})=0  ,\\

i\frac{\partial \Phi_{2}}{\partial z}+\Delta \Phi_{2}-\gamma\Phi_{2}+\frac{\beta}{2} \Phi^{2}_{1}+s(\eta|\Phi_{2}|^{2}+\rho|\Phi_{1}|^{2})\Phi_{2}=0,
\end{cases}
\end{equation}
which describes the dynamics of beam propagation in lossless bulk $\chi^{(2)}$ media inhering cubic nonlinearity, under conditions for second-harmonic generation type-I. $z$ is the propagation distance coordinate, the superscript $\ast$ denotes the complex conjugate function and $\beta$ is a constant. The slowly varying complex envelope functions of the fundamental wave $\Phi_{1}=\Phi_{1}(x,z)$ and of the second harmonic $\Phi_{2}=\Phi_{2}(x,z)$ are assumed to propagate with a constant polarization, $\overrightarrow{e}_{1},\overrightarrow{e}_{2}$ along the $z$ axis. The electric field $\overrightarrow{E}=\overrightarrow{E}(\overrightarrow{R},Z,M)$ is given by
\begin{align}\label{H_{1}}
\overrightarrow{E}=\overrightarrow{E}(\overrightarrow{R},Z,M)=E_{0}(\Phi_{1}e^{i\theta_{1}}\overrightarrow{e}_{1}+2\Phi_{2}e^{i2\theta_{1}}\overrightarrow{e}_2),
\end{align}
where $\overrightarrow{R}=r_{0}x,\ Z=z_{0}z,\ \theta_{1}=k_{1}Z-\omega_{1}M$, $\omega_{1}$ is fundamental frequency and $(\Phi_{1},\  \Phi_{2})$ satisfies \eqref{H}. The real normalization parameters $E_{0},\ z_{0}$ and $ R_{0}$ are given by
\begin{align}\label{H_{2}}
E_{0}=\frac{4\widetilde{\chi}^{(2)}_{1}}{3|\widetilde{\chi}^{(3)}_{1k}|},\ \ z_{0}=2k_{1}r^{2}_{0}\ \text{and}\  r^{2}_{0}=\frac{3|\widetilde{\chi}^{(3)}_{1k}|}{16\mu_{0}\omega^{2}_{1}(\widetilde{\chi}^{(2)}_{1})^{2}},
\end{align}
where $\mu_{0}$ is the vacuum permeability and $k_{p}$ is the wave number at the frequency $\omega_{p}$. The real parameters $\gamma,s,\eta$ and $\rho$ are given by
\begin{align}\label{H_{3}}
\gamma=2z_{0}(2k_{1}-k_{2}), \ s=sign(\widetilde{\chi}^{(3)}_{1k}),\ \eta=16\frac{\widetilde{\chi}^{(3)}_{2k}}{\widetilde{\chi}^{(3)}_{1k}},\ \  \rho=8\frac{\widetilde{\chi}^{(3)}_{1c}}{\widetilde{\chi}^{(3)}_{1k}},
\end{align}
where $2k_{1}-k_{2}\ll k_{1}$ is the phase-mismatch parameter, $\widetilde{\chi}^{j}_{p}=\widetilde{\chi}^{j}(\omega_{p})$ denote the Fourier components at frequency $\omega_{p}$ of the $j$th order susceptibility tensor and  the scalar $\chi^{3}_{p,q}$ is the vectorial Fourier transform of $\chi^{3}$ $(p,q=1,2)$. Thus, $\widetilde{\chi}^{(2)}_{1}=\widetilde{\chi}^{(2)}_2$ represents the quadratic nonlinearity, $\widetilde{\chi}^{(3)}_{pk}$ and $\widetilde{\chi}^{(3)}_{1c}=\widetilde{\chi}^{(3)}_{2c}$ are the parts
of the cubic nonlinearity responsible for self-phase and cross-phase modulation, respectively. For more details about physical meaning of system \eqref{H}, one can refer to the papers \cite{B1,B2,B3,B4}.

Indeed, if we substitute
\begin{align}\label{H_{4}}
\Phi_{1}(x,z)=u(x)e^{i\lambda_{0}z}\ \text{ and }\ \Phi_{2}(x,z)=v(x)e^{i2\lambda_{0}z}
\end{align}
into \eqref{H}, then $(u,v)$ solves the following stationary system
\begin{equation}\label{H_{5}}
\begin{cases}

-\Delta u+\lambda_{1}u=\beta uv+s(u^{2}+\rho v^{2}) u & \text{in} \ \ \mathbb{R}^{N},\\

-\Delta v+\lambda_{2}v= \frac{\beta}{2}u^{2}+s(\eta v^{2}+\rho u^{2})v& \text{in} \ \ \mathbb{R}^{N},
\end{cases}
\end{equation}
with $\lambda_1=\lambda_0$, $\lambda_2=4\lambda_0+\gamma$ and $s=\pm1$. In this paper we consider the focusing case, i.e., $s=1$ in \eqref{H_{5}}. After rescaling and renaming the parameters, we obtain system \eqref{23}.

Motivated by the fact that the $L^{2}$-norm is a preserved quantity of the evolution, (see \cite{B3}), we are interested in searching solutions to \eqref{23} with prescribed $L^2$-norm--the so-called normalized solutions to (\ref{23}). It is standard  that the solutions of \eqref{23}-\eqref{24} can be obtained as critical points of the energy functional
$$
J_{\beta}(u,v)=\frac{1}{2}\int_{\mathbb{R}^{N}}(|\nabla u|^{2}+|\nabla v|^{2})dx-\frac{1}{4}\int_{\mathbb{R}^{N}}(\mu_{1}u^{4}+\mu_{2}v^{4}+2\rho u^{2}v^{2})dx-\frac{\beta}{2}\int_{\mathbb{R}^{N}}u^{2}vdx,
$$
on the constraint space $\mathrm{T}_{b_{1}}\times \mathrm{T}_{b_{2}} $, where for $b\in\mathbb{R}$ we define
$$\mathrm{T}_{b}:=\big\{u\in H^{1}(\mathbb{R}^{N}):\int_{\mathbb{R}^{N}}u^{2}=b^{2}\big\}.$$

In the last ten years, the study of normalized solutions for Schr\"{o}dinger equations or systems has received lots of attention. However almost all the results deal with the cubic interactions.
When $\beta=0,\ N=1,\ \rho>0$, N. Nguyen and Z. Wang in \cite{nw1} proved the existence of normalized solutions to problem \eqref{23} by  minimizing the corresponding energy functional constrained on the product of the $L^{2}$-sphere and using concentration-compactness arguments, they also studied the stability properties of these solutions. Since the corresponding constrained functional is unbounded both from above and from below on the $L^{2}$-sphere, the approach used in \cite{nw1} does not work for problem \eqref{23} with $N=3$. When $\beta=0, \ N=3$, T. Bartsch, L. Jeanjean and N. Soave \cite{BJJN}  proved that for arbitrary masses $b_{i}$ and positive parameter $\mu_{i}$, there exists $\rho_{2}>\rho_{1}>0$ such that for both $0<\rho<\rho_{1}$ and $\rho>\rho_{2}$, system \eqref{23}-\eqref{24} has a positive radial solution.
T. Bartsch and N. Soave in \cite{TBNS} proved the existence of at least one normalized solution to problem \eqref{23} in the case of $\rho<0$ by a new approach based on the introduction of a natural constraint associated to problem \eqref{23}, and in \cite{TBNSS}, they proved the existence of infinitely many solutions of problem \eqref{23} with $\mu_{1}=\mu_{2}>0$ and $\rho \leq -\mu_{1}$ by using a suitable minimax argument. Later, by using bifurcation theory and the continuation method, T. Bartsch, X. Zhong and W. Zou \cite{BZZ20} obtained the existence of normalized solutions for any given $b_{1},b_{2}>0$ for $\rho$ in a large range.  They also got a result about the nonexistence of positive solutions which shows that their existence theorem is almost optimal. By using standard Ljusternik-Schnirelmann theory,  when $\beta=0$, the authors in \cite{nttv,htst} considered problem \eqref{23} in bounded domains under the case $\mu_{1}$, $\mu_{2}$, $\rho<0$, they proved the existence of infinitely many normalized solutions and occurrence of phase-separation as $\rho\rightarrow -\infty$. In \cite{bnhtgv}, B. Noris et al. studied problem \eqref{23} in bounded domains of $\R^{N}$, or the problem with trapping potentials in the whole space $\R^{N}$ (the presence of a trapping potential makes the two problems essentially equivalent) with $N\leq 3$. In both cases, they proved the existence of positive solutions with small masses $b_{1}$ and $b_{2}$, and the orbital stability of the associated solitary waves. When $\beta=0, \ N=2$, Guo et al. in \cite{glwz,glwz1,GZZ} considered the existence, non-existence, uniqueness and asymptotic behavior of solutions to problem \eqref{23}-\eqref{24} with certain type of trapping potentials.

Since competing quadratic and cubic nonlinearities is a general physical phenomenon, it is important to know the effect of such a competition on normalized solutions. In \cite{WX20}  the authors gave a  first study to problem (\ref{23}) with both quadratic and cubic interactions (without the mass constraints). In the present paper, we give a complete study on  the existence of normalized solutions in the less studied case $\beta\neq0.$ First, we consider the one dimensional case, in which the energy functional is bounded from blow on the product of $L^2$-spheres $\mathrm{T}_{b_{1}}\times \mathrm{T}_{b_{2}} $(see Lemma \ref{LA43}), so we define
\begin{equation}\label{eq1.4}
m_{\beta}(b_{1},b_{2})=\inf_{(u,v)\in \mathrm{T}_{b_{1}}\times \mathrm{T}_{b_{2}}}J_{\beta}(u,v)
\end{equation}
and then we prove the following result
\begin{theorem}\label{Th4}
If $N=1$, $\mu_{1}, \mu_{2}, \rho>0$, then for every $\beta>0$, $-\infty<m_{\beta}(b_{1},b_{2})<0$ is achieved. In addition any minimizing sequence for \eqref{eq1.4} is, up to translation, strongly convergent in $H^{1}(\mathbb{R})\times H^{1}(\mathbb{R}) $ to a solution of \eqref{23}-\eqref{24}.
\end{theorem}

Next, we turn to $N=2.$ In this case, the energy functional is not always bounded on $\mathrm{T}_{b_{1}}\times \mathrm{T}_{b_{2}} $, relating to the values $b_{1}$ and $b_{2}$. Let $Q$ be the unique positive radial solution to
\begin{equation}
\label{Qdef}
-\Delta Q+ Q= Q^3,\  Q\in H^1 (\R^2).
\end{equation}

 The main results in dimension $N=2$ can be stated as follows:
\begin{theorem}\label{Th3}
Let $N=2$ and $\mu_{1}, \mu_{2}, \rho>0$:

(i) If \ $\max\{(\mu_{1}+\rho)b^{2}_{1}, (\mu_{2}+\rho)b^{2}_{2}\}<\| Q\|_{L^{2}(\mathbb{R}^{2})}^{2}$, then

 a) for every $\beta<0$, problem \eqref{23}-\eqref{24} has no positive solutions,

 b) for every $\beta>0$, when $0<b_{2}\leq \| Q\|_{L^{2}(\mathbb{R}^{2})}$, $-\infty<m_{\beta}(b_{1},b_{2})<0$ is achieved by $(u,v)$ which is a positive solution  of \eqref{23}-\eqref{24}.

(ii) If $\frac{\mu_{1}b^{4}_{1}+\mu_{2}b^{4}_{2}+2\rho b^{2}_{1}b^{2}_{2}}{(b^{2}_{1}+b^{2}_{2})}>\| Q\|_{L^{2}(\mathbb{R}^{2})}^{2}$,  then $m_{\beta}(b_{1},b_{2})=-\infty$ for any $\beta \in \mathbb{R}$.
\end{theorem}
In order to prove Theorem \ref{Th3}, in section \ref{sec7}, we introduce the following minimization problem
\[A=A(\mu_1,\mu_1,\rho,b_1,b_2):=\inf_{(u,v)\in \mathrm{T}_{b_{1}}\times \mathrm{T}_{b_{2}}}\frac{\int_{\mathbb{R}^{2}}(|\nabla u|^{2}+|\nabla v|^{2})dx}{\int_{\mathbb{R}^{2}}\left(\mu_{1}u^{4}+\mu_{2}v^{4}+2\rho u^{2}v^{2}\right)dx}\]
and prove that  \begin{align*}
\frac{\| Q\|_{L^{2}(\mathbb{R}^{2})}^{2}}{2\max\{(\mu_{1}+\rho)b^{2}_{1}, (\mu_{2}+\rho)b^{2}_{2}\}}\leq A\leq\frac{\frac{1}{2}(b^{2}_{1}+b^{2}_{2})\| Q\|_{L^{2}(\mathbb{R}^{2})}^{2}}{\mu_{1}b^{4}_{1}+\mu_{2}b^{4}_{2}+2\rho b^{2}_{1}b^{2}_{2}}.
 \end{align*}
We shall prove that if  $A>\frac{1}{2}$, then  $J_{\beta}(u,v)$ is coercive on $\mathrm{T}_{b_{1}}\times \mathrm{T}_{b_{2}}$ and $m_{\beta}(b_{1},b_{2})>-\infty$ is achieved, so there exists at least one normalized ground state for \eqref{23}.

\begin{remark}
When $N=2$, if $A<\frac{1}{2}$, then $m_{\beta}(b_{1},b_{2})=-\infty.$ Thus the minimization problem constrained on $\mathrm{T}_{b_{1}}\times \mathrm{T}_{b_{2}}$ does not work. Since $\int_{\mathbb{R}^{2}}|\nabla  u|^{2}+|\nabla  v|^{2}$ and $\int_{\mathbb{R}^{2}}\mu_{1}u^{4}+\mu_{2} v^{4}+2\rho  u^{2} v^{2}$ behave at the same way under $L^{2}$ preserving scaling of $(u,v)$, $J_{\beta}(\iota u(\iota x),\iota v(\iota x))$ may strictly increasing with respect to $\iota\in \mathbb{R}$, so the usual methods developed on the Pohozaev-Nehari constraint can not be applied directly here, see \cite{BJJN,TBNS,Soave,Soave1}. Inspired by the recent works \cite{YY1,LIZOU2},  we try to construct a submanifold of $\mathrm{T}_{b_{1}}\times \mathrm{T}_{b_{2}}$ as following
 \small{\begin{align*}
N_{b_{1},b_{2}}:=\Big\{(u,v)\in \mathrm{T}_{b_{1}}\times \mathrm{T}_{b_{2}}\mid P_{b_{1},b_{2}}(u,v)=0 \text{ and } \int_{\mathbb{R}^{2}}(|\nabla u|^{2}+|\nabla v|^{2})dx<\frac{1}{2}\int_{\mathbb{R}^{2}}\left(\mu_{1}u^{4}+\mu_{2}v^{4}+2\rho u^{2}v^{2}\right)dx\Big\},
\end{align*}}
on which $J(u,v)$ may admits a minimizer, where $P_{b_{1},b_{2}}(u,v)=0$ is the related Pohozaev-Nehari identity. Indeed, for any $(u,v)$ on the set $N_{b_{1},b_{2}}$, $J_{\beta}(\iota u(\iota x),\iota v(\iota x))$ has a unique maximum point ${\iota}_{u,v}$ and is strictly increasing in $(-\infty,{\iota}_{u,v})$ and decreasing in $({\iota}_{u,v},+\infty)$. Therefore, we expect a constrained variation can be used to obtain a normalized solution to \eqref{23}. However, due to the uncertainty of sign of the quadratic interaction term and inhomogeneity of the coupling term, it seems difficult to prove the compactness of the minimizing sequences (or Palais-Smale sequences) developed in $N_{b_{1},b_{2}}$. We believe that when $A<\frac{1}{2}$, the existence of normalized solutions to \eqref{23} in dimension two is an expected interesting result.
\end{remark}

%
Now, we deal with the three dimensional case. In this case, $m_{\beta}(b_{1},b_{2})=-\infty$ for any $b_{1},b_{2}>0. $ Indeed, the leading term is $L^{2}$ supercritical and Sobolev subcritical, the energy function $J_{\beta}(u,v)$ is unbounded both from above and from below on $\mathrm{T}_{b_{1}}\times \mathrm{T}_{b_{2}} $. In order to search for two normalized solutions, we use the ideas introduced by N. Soave \cite{Soave1, Soave} to study the corresponding fiber maps $\Psi_{u,v}(t)$ (see \eqref{pp}), which has the same Mountain pass structure as the original functional. The benefits of the fiber map are that the critical point of $\Psi_{u,v}(t)$ allow to project a function on $\mathcal{P}_{b_{1},b_{2}}$ (see \eqref{int4}).  The monotonicity and convexity of $\Psi_{u,v}(t)$ has a strongly affect the structure of $\mathcal{P}_{b_{1},b_{2}}$ and then intimately related to the minimax structure of $J_{\beta}(u,v)|_{\mathrm{T}_{b_{1}}\times \mathrm{T}_{b_{2}}}$. To show our main results, we first introduce the Gagliardo-Nirenberg-Sobolev inequality
\begin{equation}\label{4}
\|u\|_{L^{p}(\mathbb{R}^{N})}\leq C_{N,p}\|\nabla u\|_{L^{2}(\mathbb{R}^{N})}^{\gamma_{p}}\| u\|_{L^{2}(\mathbb{R}^{N})}^{1-\gamma_{p}}\ \ \text{for all} \ \ u\in H^{1}(\mathbb{R}^{N}),
\end{equation}
where
\begin{align*}
\gamma_{p}=\frac{N(p-2)}{2p},
\end{align*}
and we denote by $C_{N,p}$ the best constant in the the Gagliardo-Nirenberg-Sobolev inequality $H^{1}(\mathbb{R}^{N})\hookrightarrow L^{p}(\mathbb{R}^{N}), \ 2<p<2^{\ast}=\frac{2N}{N-2}(N\ge 3).$
Our main results are the following:
\begin{theorem}\label{Th1}
When $N=3$, $\mu_{1}, \mu_{2}, \rho, \beta>0$ and $$\beta\left(2b^{\frac{3}{2}}_{1}+b^{\frac{3}{2}}_{2}\right)C^{3}_{3,3}C^{2}_{3,4}\sqrt{\mu_{1}b_{1}+\mu_{2}b_{2}+\rho b^{\frac{1}{2}}_{1}b^{\frac{1}{2}}_{2}}<\frac{2\sqrt{6}}{3},$$ then \eqref{23}-\eqref{24} has at least two positive normalized solutions, one is a ground state $(\widehat{u}_{\beta},\widehat{v}_{\beta})$, the other is an excited state $(\widetilde{u}_{\beta},\widetilde{v}_{\beta})$. Moreover, $J_{\beta}(\widehat{u}_{\beta},\widehat{v}_{\beta})\rightarrow0^{+}$, $\int_{\mathbb{R}^{3}}(|\nabla \widehat{u}_{\beta}|^{2}+|\nabla \widehat{v}_{\beta}|^{2})dx\rightarrow 0$ and there exists $\rho_{2}>0 $ such  that when $\rho>\rho_{2}$, then $(\widetilde{u}_{\beta},\widetilde{v}_{\beta})\rightarrow(\widetilde{u}_{0},\widetilde{v}_{0})$ strongly in $H^{1}(\mathbb{R}^{3})\times H^{1}(\mathbb{R}^{3})$ as $\beta\rightarrow 0$, where $(\widetilde{u}_{0},\widetilde{v}_{0})$ is a normalized ground state of \eqref{23} with $\beta=0$.
\end{theorem}
\begin{remark}
Theorem \ref{Th1} gives the existence of two normalized solutions for \eqref{23}. The first one is a local minimizer for which we establish the compactness of minimizing sequence. The second solution is obtained through a constrained linking. Theorem \ref{Th1}  also shows the limit behavior of the solutions as $\beta\rightarrow 0 $. The first solution will disappear and the second solution  will converge to the normalized solution of system \eqref{23} with $\beta=0$, which has been studied by T. Bartsch, L. Jeanjean and N. Soave  in \cite{BJJN}.  When $\beta=0, \ N=3$, in \cite{BJJN} the authors proved that for arbitrary masses $b_{i}$ and parameter $\mu_{i}$, there exists $\rho_{2}>\rho_{1}>0$ depending on the masses such that for both $0<\rho<\rho_{1}$ and $\rho>\rho_{2}$ system \eqref{23}--\eqref{24} has a positive radial solution. Recently,
 T. Bartsch, X. Zhong and W. Zou in \cite{BZZ20} obtained the existence of normalized solutions for any given $b_{1}, b_{2}>0$ and $\rho>0$ in a large range, which is independent of the masses. They also have a result about the nonexistence of positive solutions which shows that their existence theorem is almost optimal. Therefore our results indicate that the quadratic interaction term not only enriches the set of solutions to the above Schr\"{o}dinger system but also expands the permissible range of $\rho$.
\end{remark}

Furthermore, we give a mass collapse behavior of the ground states obtained in Theorem \ref{Th1}.
From Theorem 1.2 of \cite{ZZS15}, we know that $(u_{0},v_{0})=(\sqrt{2}\beta^{-1} w, \beta^{-1}w)$ is the unique positive solution of
\begin{equation}\label{eqd1}
\begin{cases}

-\Delta u+u= \beta uv  & \text{in} \ \ \mathbb{R}^{3},\\

-\Delta v+v=\frac{\beta}{2} u^{2}& \text{in} \ \ \mathbb{R}^{3},
\end{cases}
\end{equation}
 where $w$ is the unique positive solution of
\begin{equation}\label{eqd2}
-\Delta u +u =u^{2}, \ u \in H^{1}(\mathbb{R}^{3}).
\end{equation}

\begin{theorem}\label{Th2}
Assume that the assumptions in Theorem 1.3 hold, and $(u_{b_{1},b_{2}},v_{b_{1},b_{2}})$ is a ground state for \eqref{23}-\eqref{24}. Up to a subsequence, we have 
$$\big(L^{-1}_{1}u_{b_{1},b_{2}}(\theta^{-1}_{1}(x)), L^{-1}_{2}v_{b_{1},b_{2}}(\theta^{-1}_{2}(x))\big)\rightarrow (\bar{u}, \bar{v})$$
in \ $H^{1}(\mathbb{R}^{3})\times H^{1}(\mathbb{R}^{3})$  as $b_1,b_2\to 0$ and $b_1\backsim b_2$, where $(\bar{u}, \bar{v}) $ satisfies
\begin{equation*}
\begin{cases}

-\Delta u+\lambda^{\ast}_{1}u= \beta uv  & \text{in} \ \ \mathbb{R}^{3},\\

-\Delta v+\lambda^{\ast}_{2}v= \frac{\beta}{2}u^{2}& \text{in} \ \ \mathbb{R}^{3},
\end{cases}
\end{equation*}
for some $\lambda^{\ast}_{1},\ \lambda^{\ast}_{2}>0$, where $\theta_{1}=\frac{2\beta^{2}b_{1}^{\frac{6}{5}}b_{2}^{\frac{4}{5}}}{16^{\frac{2}{5}}\|w\|^{2}_{L^{2}(\mathbb{R}^{3})}},\ \theta_{2}=\frac{\beta^{2}b_{1}^{\frac{8}{5}}b_{2}^{\frac{2}{5}}}{16^{\frac{1}{5}}\|w\|^{2}_{L^{2}(\mathbb{R}^{3})}},\ L_{1}=\frac{2\beta^{4}b_{1}^{\frac{14}{5}}b_{2}^{\frac{6}{5}}}{16^{\frac{3}{5}}\|w\|^{4}_{L^{2}(\mathbb{R}^{3})}}$, $L_{2}=\frac{4\beta^{4}b_{1}^{\frac{12}{5}}b_{2}^{\frac{8}{5}}}{16^{\frac{4}{5}}\|w\|^{4}_{L^{2}(\mathbb{R}^{3})}}$.
Moreover, if $ \lambda^{\ast}_{1}= \lambda^{\ast}_{2}$, then $(\bar{u}, \bar{v})=(\sqrt{2}\beta^{-1} w,\beta^{-1} w)$. 
\end{theorem}
Denote the set of ground states to \eqref{23}-\eqref{24} by $\mathcal{M}_{b_{1},b_{2}}$. If $(u,v)$ solves \eqref{23} with some $\lambda_{1},\lambda_{2}\in \mathbb{R}$, then $\Phi(z,x)=e^{-i\lambda_{1}z}u(x),\ \Psi(z,x)=e^{-i\lambda_{2}z}v(x)$ satisfies the time-dependent system
\begin{equation}\label{L}
\begin{cases}
i\partial_{z} \Phi-\Delta \Phi=\beta \Phi \Psi+\mu_{1}\Phi^{3} +\rho \Psi^{2}\Phi& \text{in} \ \ \mathbb{R}^{3},\\

i\partial_{z} \Psi-\Delta \Psi= \frac{\beta}{2}\Phi^{2}+\mu_{2}\Psi^{3}+\rho \Phi^{2}\Psi& \text{in} \ \ \mathbb{R}^{3},
\end{cases}
\end{equation}
where $(z,x)\in \mathbb{R}\times \mathbb{R}^{3},\ i=\sqrt{-1}$. From \cite{BBR98}, we know  the local well-posedness of solutions to \eqref{L} holds for $\beta>0$, then we can consider the stability of $\mathcal{M}_{b_{1},b_{2}}$. The set $\mathcal{M}_{b_{1},b_{2}}$ is said to be stable under the Cauchy flow of \eqref{L} if for any $\epsilon>0$, there exists $\delta>0$ such that for any $(\Phi_{0},\Psi_{0})\in H^{1}(\mathbb{R}^{3})\times H^{1}(\mathbb{R}^{3})$ satisfying $$dist_{H^{1}(\mathbb{R}^{3})\times H^{1}(\mathbb{R}^{3})}((\Phi_{0},\Psi_{0}),\mathcal{M}_{b_{1},b_{2}})<\delta,$$ then the solution $((\Phi(z,\cdot),\Psi(z,\cdot))$ of \eqref{L} with $((\Phi(0,\cdot),\Psi(0,\cdot))=(\Phi_{0},\Psi_{0})$ satisfies $$\sup dist_{H^{1}(\mathbb{R}^{3})\times H^{1}(\mathbb{R}^{3})}((\Phi(z,\cdot),\Psi(z,\cdot)),\mathcal{M}_{b_{1},b_{2}})<\epsilon,$$ where $Z$ is the maximal existence time for $((\Phi(z,\cdot),\Psi(z,\cdot))$. We have

\begin{theorem}\label{Th5}
When $\mu_{1}, \mu_{2}, \rho, \beta>0$ and $$\beta\left(2b^{\frac{3}{2}}_{1}+b^{\frac{3}{2}}_{2}\right)C^{3}_{3,3}C^{2}_{3,4}\sqrt{\mu_{1}b_{1}+\mu_{2}b_{2}+\rho b^{\frac{1}{2}}_{1}b^{\frac{1}{2}}_{2}}<\frac{2\sqrt{6}}{3},$$ then the set $\mathcal{M}_{b_{1},b_{2}}$  is compact, up to translation, and it is stable.
\end{theorem}
\begin{remark}
Theorem \ref{Th5} indicates that the small quadratic interaction term leads to a stabilizations of standing waves corresponding to \eqref{L}. Indeed, when $\beta=0$, T. Bartsch, L. Jeanjean and N. Soave in \cite{BJJN} showed that the associated solitary wave is orbitally unstable by blowing up in finite time. Therefore, by creating a gap in the ground state energy level of the system (from positive to negative), the quadratic coupling term not only makes the ground state solutions stable, but also changes the structure of the energy functional and enriches the solution set.
\end{remark}

Now, we deal with the four dimensional case. In this case, $m_{\beta}(b_{1},b_{2})=-\infty$ for any $b_{1},b_{2}>0. $ Indeed, the leading term is $L^{2}$-critical and Sobolev critical, the energy function $J_{\beta}(u,v)$ is unbounded both from above and from below on $\mathrm{T}_{b_{1}}\times \mathrm{T}_{b_{2}}$.
Denote $\mathcal{S}=\inf_{D^{1,2}(\mathbb{R}^{4})\backslash \{0\}}\frac{\|\nabla u\|^{2}_{L^{2}(\R^4)}}{\|u\|^{2}_{L^{4}(\R^4)}}.$
From \cite{GT96}, we know that $\mathcal{S}$ is attained by the Aubin-Talanti bubbles
\begin{align}\label{a5}
U_{\epsilon}(x):=\frac{2\sqrt{2}\epsilon}{\epsilon^{2}+|x|^{2}},\ \epsilon>0,\ x\in \mathbb{R}^{4}.
\end{align}
Then $U_{\epsilon}$ satisfies $-\Delta u=u^{3}$ and $\int_{\mathbb{R}^{4}}|\nabla U_{\epsilon}|^{2}dx=\int_{\mathbb{R}^{4}}| U_{\epsilon}|^{4}dx=\mathcal{S}^{2}$. On the other hand, if $0<\rho<\min\{\mu_{1},\mu_{2}\}$ or $\rho>\max\{\mu_{1},\mu_{2}\}$(see Lemma \ref{lem2.5}),  $\left(\sqrt{\frac{\rho-\mu_{2}}{\rho^{2}-\mu_{1}\mu_{2}}}U_{\epsilon},\sqrt{\frac{\rho-\mu_{1}}{\rho^{2}-\mu_{1}\mu_{2}}}U_{\epsilon} \right)$ is the least energy solutions to the following elliptic system:
\begin{equation}\label{a6}
\begin{cases}
-\Delta u=\mu_{1}u^{3}+\rho v^{2}u  & \text{in} \ \ \mathbb{R}^{4},\\
-\Delta v= \mu_{2}v^{3}+\rho u^{2}v& \text{in} \ \ \mathbb{R}^{4},\\
u,v\in D^{1,2}(\mathbb{R}^{4}).
\end{cases}
\end{equation}
Define
\begin{equation}\label{a7}
\mathcal{S}_{\mu_{1},\mu_{2},\rho}=\inf_{(u,v)\in[D^{1,2}(\mathbb{R}^{4})]^{2}\backslash \{(0,0)\}}\frac{\int_{\mathbb{R}^{4}}(|\nabla u|^{2}+|\nabla v|^{2})dx}{\left(\int_{\mathbb{R}^{4}}\left(\mu_{1}u^{4}+\mu_{2}v^{4}+2\rho u^{2}v^{2}\right)dx\right)^{\frac{1}{2}}}.
\end{equation}
The main results on this aspect can be stated as follows:

If $\beta<0$, we have the following non-existence results.
\begin{theorem}\label{th1.4}
Let $\mu_i,b_i,\rho>0(i=1,2)$ and $\beta<0$.
\begin{enumerate}
\item
If $N=4$, then problem \eqref{23}-\eqref{24} has no positive solution $(u,v)\in H^1(\mathbb{R}^{4})\times H^1(\mathbb{R}^{4})$.
\item
If $N=5$, problem \eqref{23}-\eqref{24} has no positive solution $(u,v)\in H^1(\mathbb{R}^{5})\times H^1(\mathbb{R}^{5})$ satisfying the additional assumption that $u\in L^p(\R^5)$ for some $p\in (0,\frac{5}{3}]$.
\item
Moreover,
if $N\ge4$, problem \eqref{23}-\eqref{24} has no non-trivial radial solution.
\end{enumerate}
\end{theorem}

Next, we consider the case $\beta>0$. 

\begin{theorem}\label{th1.5}
When $N=4$, let $\mu_i,b_i>0(i=1,2)$, $\beta>0$, and $\rho\in \big(0,\min\{\mu_1,\mu_2\}\big)\cup\big(\max\{\mu_1,\mu_2\},\infty\big)$, then
the following conclusions hold.
\begin{enumerate}
\item
If $0<\beta b_1<\!\frac{3}{2|C_{4,3}|^{3}}$ and $0<\beta b_2<\!\frac{3}{|C_{4,3}|^{3}}$,
then \eqref{23}-\eqref{24} has a positive ground state solution $(u_{b_1,b_{2}},v_{b_{1},b_2})\in \mathrm{T}_{b_{1}}\times \mathrm{T}_{b_{2}} $.

\item
Moreover, there exists $\sigma_1>0$ such that
\begin{equation*}
\big(\sigma_1 u_{b_1,b_{2}}(\sigma_1 x),\sigma_1 v_{b_1,b_{2}}(\sigma_1 x)\big)
\to \Big(\sqrt{\frac{\rho-\mu_2}{\rho^2-\mu_1\mu_2}}U_{\varepsilon_0},\sqrt{\frac{\rho-\mu_1}{\rho^2-\mu_1\mu_2}}U_{\varepsilon_0}\Big)
\end{equation*}
in $D^{1,2}(\R^4)\times D^{1,2}(\R^4)$, for some $\varepsilon_0>0$ as $(b_1,b_2)\to (0,0)$, up to a subsequence.
\end{enumerate}
\end{theorem}
\begin{remark}
If $N=4$, in problem \eqref{23}, $\beta uv$ and $\frac{\beta}{2}u^2$ can be regarded as mass-critical terms. Moreover, the terms $u^3, v^3$ and the coupled terms $u^2v, v^2u$ are mass super-critical and energy critical.
Theorem \ref{th1.5} indicate that problem \eqref{23} with mass critical lower order perturbation term possesses at least one normalized ground state solution, whose two components both converge to the Aubin-Talanti bubble in related Sobolev space by making appropriate scaling, as the masses of two components vanish.
More recently, in \cite{LYZ} jointly with W. Zou, the first and third authors in this present paper considered the equation
\begin{equation}\label{eqA0.1}
\begin{cases}
-\Delta u+\lambda_1u=\alpha_1|u|^{p-2}u+\mu_1u^3+\beta v^2u\quad&\hbox{in}~\R^4,\\
-\Delta v+\lambda_2v=\alpha_2|v|^{p-2}v+\mu_2v^3+\beta u^2v\quad&\hbox{in}~\R^4,\\
\end{cases}
\end{equation}
under the mass constraint $(u,v)\in \mathrm{T}_{b_{1}}\times \mathrm{T}_{b_{2}} $, where $\mu_1,\mu_2,\beta>0$, $\alpha_1,\alpha_2\in \R$, $p\!\in\!(2,4)$ and $\lambda_1,\lambda_2\!\in\!\R$ appear as Lagrange multipliers. We must point out that
compared with \cite{LYZ}, the problem \eqref{23} in $\R^4$ is more delicate due to the uncertainty of sign of the quadratic interaction term and inhomogeneity of the coupling term. Moreover, since system \eqref{23} with mixed couplings is asymmetric, the permissible range of $b_1,b_2$ obtained in Theorem \ref{th1.5} may not be optimal, so it is difficult to prove the non-existence of normalized solutions to \eqref{23} for larger $b_1$ and $b_2$.
\end{remark}

Finally, we give outline of the proofs.
In the one dimension case, the energy functional is bounded from below on the product of $L^2$-spheres, the constrained minimization method developed by L. Jeanjean \cite{GL18} can be used to obtained a normalized ground state, which is obtained by establishing the compactness of the minimizing sequences. When $N=3$, the energy functional is always unbounded on the product of $L^2$-spheres. We use the ideas introduced by N. Soave \cite{Soave1, Soave} to study the related fiber maps $\Psi_{u,v}(t)$(see \eqref{pp}). It is easy to see that the critical point of $\Psi_{u,v}(t)$ allow to project a function on $\mathcal{P}_{b_{1},b_{2}}$(see \eqref{int4}) and the monotonicity and convexity properties of $\Psi_{u,v}(t)$ has a strongly affect the structure of $\mathcal{P}_{b_{1},b_{2}}$ and then intimately related to the minimax structure of $J_{\beta}(u,v)|_{\mathrm{T}_{b_{1}}\times \mathrm{T}_{b_{2}}}$.

However, for the system \eqref{23} we study, due to the uncertainty of the sign of the term $\beta\int_{\mathbb{R}^{N}}u^{2}vdx$ in the corresponding energy functional, the above method cannot be used directly, and we need to introduce additional constrains on the previous Pohozaev manifold. On the new manifold with additional constrains $\mathcal{P}_{b_{1},b_{2}}$(see \eqref{k1}-\eqref{k3}), under suitable condition on $b_1$ and $b_2$,  we can prove that $J_{\beta}(u,v)|_{\mathrm{T}_{b_{1}}\times \mathrm{T}_{b_{2}}}$ admits a convex-concave geometry and the new manifold  $\mathcal{P}_{b_{1},b_{2}}$ is non-empty. To show that $\mathcal{P}_{b_{1},b_{2}}$ is a natural constraint, we use some ideas introduced by F. Clarke in \cite{FHC76} and already used by J. Mederski and J. Schino in \cite{MS21} to deal with minimization problems whose constraints are given by inequalities. Therefore, $J_{\beta}(u,v)|_{\mathrm{T}_{b_{1}}\times \mathrm{T}_{b_{2}}}$ has a local minimizer and a mountain pass critical point. By establish the compactness of minimizing sequence, we obtain a solution which is a local minimizer. The second solution is obtained through a constrained mountain pass. It is worth pointing out that after proving the strongly convergence of minimizing sequences, we need to verify that the limit function $(u,v)$ is still in the allowed set $\mathcal{P}_{b_{1},b_{2}}$. Indeed, if $\beta\int_{\mathbb{R}^{N}}u^{2}vdx\leq0,$ then the local minimum point will disappear, thus we get a contradiction.

When $N=3$, obviously, we obtain that $J_{\beta}(u,v)|_{\mathrm{T}_{b_{1}}\times \mathrm{T}_{b_{2}}}$ has no local minimizer as $\beta\rightarrow 0^+$, i.e. the ground state energy converges to $0$ as $\beta \to 0^+$. The mountain pass solutions $(\widetilde{u}_{\beta},\widetilde{v}_{\beta})$ obtained in Theorem \ref{Th1} depends on $\beta$, we shall analyze the convergence properties of $(\widetilde{u}_{\beta},\widetilde{v}_{\beta})$ as $\beta\to 0^+$. We first deduce that for the function $\beta\in [0,\infty)\mapsto m_{\beta}(b_{1},b_{2})\in\R$ is monotone non-increasing, and $\{(\widetilde{u}_{\beta},\widetilde{v}_{\beta})\}$ is bounded in $H^{1}(\mathbb{R}^{3})\times H^{1}(\mathbb{R}^{3})$.
Then there exists a subsequence, such that $(\widetilde{u}_{\beta},\widetilde{v}_{\beta})$ converge to $(\widetilde{u}_{0},\widetilde{v}_{0})$ strongly in $H^{1}(\mathbb{R}^{3})\times H^{1}(\mathbb{R}^{3})$ as $\beta\to 0^+$, where $(\widetilde{u}_{0},\widetilde{v}_{0})$ is a ground state solution of \eqref{23}-\eqref{24} with $\beta=0$.

In the proof of Theorem \ref{Th2}, the main ingredient is the refined upper bound of $m^{+}(b_1,b_2)$ (see Lemma \ref{LEM}) i.e.
\begin{equation*}
m^{+}_{\beta}(b_{1},b_{2})<-\frac{1}{6}\Big[\frac{4\beta^{4}b_{1}^{\frac{22}{5}}b_{2}^{\frac{8}{5}}}{16^{\frac{4}{5}}\|w\|^{6}_{L^{2}(\mathbb{R}^{3})}}+\frac{\beta^{4}b_{1}^{\frac{16}{5}}b_{2}^{\frac{14}{5}}}{16^{\frac{2}{5}}\|w\|^{6}_{L^{2}(\mathbb{R}^{3})}}\Big]\|\nabla w\|^{2}_{L^{2}(\mathbb{R}^{3})}.
\end{equation*}
This refinement needs to keep the testing functions staying in the admissible set $A_{R_{0}}=\big\{(u,v)\in \mathrm{T}_{b_{1}}\times \mathrm{T}_{b_{2}}:\|\nabla  u\|^{2}_{L^2(\R^3)}+\|\nabla  v\|^{2}_{L^2(\R^3)}< R^2_{0}\big\}$.
We overcome this difficulty by utilising the properties of the unique positive radial ground state solution of \eqref{eqd1}. By accurate estimation and careful analysis, we give a precise description of the asymptotic behavior of solutions as the mass $b_{1},b_{2}$ goes to zero.

To prove item 1 of Theorem  \ref{th1.5}, we follow the approach of \cite{WW21}. We can obtain a bounded minimizing sequence by using the Pohozaev constraint approach(see\cite{Soave,Soave1}). However, it is very difficult to prove the compactness of a minimizing sequence at positive energy levels. Motivated by \cite{WW21}, we drive a better energy estimate on the associated mountain pass energy level (see Lemma \ref{lem3.4} ), i.e $$0<m_{\beta}(b_1,b_2)<\frac{k_1+k_2}{4}\mathcal{S}^2.$$ This is enough to guarantee the compactness of minimizing sequences at the energy level $m_{\beta}(b_1,b_2).$ To prove item 2 of Theorem \ref{th1.5}, we first prove the following energy splitting asymptotic property of the solution
\begin{equation*}
||\nabla u_{b_1,b_{2}}||^2_2+\|\nabla v_{b_1,b_{2}}\|^2_2\to (k_1+k_2)\mathcal{S}^2,\ \ \mu_1\|u_{b_1,b_{2}}\|^{4}_{4}+\mu_2\|v_{b_1,b_{2}}\|^{4}_{4}+2\rho\|u_{b_1,b_{2}}v_{b_1,b_{2}}\|^2_2\to (k_1+k_2)\mathcal{S}^2,
\end{equation*}
as $(b_1,b_2)\to (0,0)$. Then we obtain that $(u_{b_1,b_{2}},v_{b_1,b_{2}})$ is a special minimizing sequence of the minimizing problem \eqref{a7}. We follow some ideas from Lemma 3.5 of \cite{LL}, and by the uniqueness of ground state solution of \eqref{a7}, we obtain the precisely asymptotic behavior of $(u_{b_1,b_{2}},v_{b_1,b_{2}})$  as $(b_{1},b_{2})\rightarrow(0,0)$.


\vskip2mm

Throughout the paper, we shall write $a\backsim b$ if $C_{1}a\leq b\leq C_{2}a$ where $C_{i}, i=1,2$ are constants. $H^{1}_{r}$ denotes the subspace of functions in $H^{1}$ which are radially symmetric with respect to 0, and $\mathrm{T}_{b_{i,r}}=\mathrm{T}_{b_{i}}\cap H^{1}_{r}, i=1,2 $. The rest of this paper is organized as follows.   In section \ref{sec7}, we prove Theorem \ref{Th4}. In section \ref{sec15}, we prove Theorem\ref{Th3}. In section \ref{sec4}, we prove Theorem  \ref{Th1}. In section \ref{sec16}, we prove Theorems \ref{th1.4} and \ref{th1.5}.
\section{Proof of Theorem \ref{Th4}}\label{sec7}
To prove Theorem \ref{Th4}, we use the ideas introduced in \cite{GL16,GL18}. First we recall  the rearrangement results of M. Shibata \cite{Shibata} as presented in \cite{NI14,GL16}. Let $u$ be a Borel measurable function on $\mathbb{R}^{N}$. It is said to vanish at infinity if the level set $|\{x\in \mathbb{R}^{N}: u(x)>t\}|<+\infty$ for every $t>0$. Here $|A|$ stands for the N-dimensional Lebesgue measure of a Lebesgue measurable set $A\subset \mathbb{R}^{N}$. Considering two Borel mesurable functions $u,v$ which vanish at infinity in $\mathbb{R}^{N}$, for $t>0,$ we define $A^{\star}(u,v:t):=\{x\in \mathbb{R}^{N}: |x|<r \}$, where $r>0$ is chosen so that $$B(0,r)=|\{x\in \mathbb{R}^{N}: |u(x)|>t \}|+|\{x\in \mathbb{R}^{N}: |v(x)|>t \}|,$$ and $\{u,v\}^{\star}$ by $$\{u,v\}^{\star}(x):=\int^{\infty}_{0}\chi_{A^{\star}(u,v:t)}(x)dt,$$ where $\chi_{A}(x)$ is a characteristic function of the set $A\subset \mathbb{R}^{N}.$
\begin{lemma}[\cite{NI14} Lemma A.1]\label{LM}
(i) The function $\{u,v\}^{\star}$ is radially symmetric, non-increasing and lower semi-continuous. More, for each $t>0$ there holds $\{x\in \mathbb{R}^{N}:\{u,v\}^{\star}>t\}=A^{\star}(u,v:t)$.

(ii) Let $\Phi:[0,\infty)\rightarrow[0,\infty)$ be non-decreasing lower semi-continuous, continuous at 0 and $\Phi(0)=0.$ Then $\{\Phi(u), \Phi(v)\}^{\star}=\Phi(\{u,v\}^{\star})$.

(iii) $\|\{u,v\}^{\star}\|^{p}_{L^{p}(\mathbb{R}^{N})}=\|u\|^{p}_{L^{p}(\mathbb{R}^{N})}+\|v\|^{p}_{L^{p}(\mathbb{R}^{N})}$ for $1\leq p<\infty.$

(iv) If $u,v\in H^{1}(\mathbb{R}^{N})$, then $\{u,v\}^{\star}\in H^{1}(\mathbb{R}^{N}) $ and $$\|\nabla \{u,v\}^{\star}\|^{2}_{L^{2}(\mathbb{R}^{N})}\leq \|\nabla u\|^{2}_{L^{2}(\mathbb{R}^{N})}+\|\nabla v\|^{2}_{L^{2}(\mathbb{R}^{N})}.$$ In addition, if  $u,v\in (H^{1}(\mathbb{R}^{N})\cap C^{1}(\mathbb{R}^{N}))\setminus\{0\}$ are radially symmetric, positive and non-increasing, then $$\int_{\mathbb{R}^{N}}|\nabla \{u,v\}^{\star}|^{2}dx<\int_{\mathbb{R}^{N}}|\nabla u|^{2}dx+\int_{\mathbb{R}^{N}}|\nabla v|^{2}dx.$$

(v) Let $u_{1},u_{2},v_{1},v_{2}\geq 0$ be Borel measurable functions which vanish at infinity, then $$\int_{\mathbb{R}^{N}}(u_{1}u_{2}+v_{1}v_{2})dx\leq \int_{\mathbb{R}^{N}}\{u_{1},v_{1}\}^{\star}\{u_{2},v_{2}\}^{\star}dx.$$
\end{lemma}
The solution of \eqref{23}--\eqref{24} can be found as a critical point of the following  energy functional
$$
J_{\beta}(u,v)=\frac{1}{2}\int_{\mathbb{R}^{N}}(|\nabla u|^{2}+|\nabla v|^{2})dx-\frac{1}{4}\int_{\mathbb{R}^{N}}(\mu_{1}u^{4}+\mu_{2}v^{4}+2\rho u^{2}v^{2})dx-\frac{\beta}{2}\int_{\mathbb{R}^{N}}u^{2}vdx.
$$

From  Gagliardo-Nirenberg-Sobolev inequality \eqref{4}, we have
\begin{align}\label{LA11}
\frac{\mu_{1}}{4}\int_{\mathbb{R}^{N}}u^{4}dx\leq\frac{\mu_{1}}{4} C^{4}_{N,4}b^{4-N}_{1}\|\nabla u\|^{N}_{L^{2}(\mathbb{R}^{N})},\ \ \frac{\mu_{2}}{4}\int_{\mathbb{R}^{N}}v^{4}dx\leq\frac{\mu_{2}}{4} C^{4}_{N,4}b^{4-N}_{2}\|\nabla v\|^{N}_{L^{2}(\mathbb{R}^{N})},
\end{align}
\begin{align}\label{LA12}
\int_{\mathbb{R}^{N}}u^{2}v^{2}dx&\leq\left(\int_{\mathbb{R}^{N}}u^{4}dx\right)^{\frac{1}{2}}\left(\int_{\mathbb{R}^{N}}v^{4}dx\right)^{\frac{1}{2}}\\\nonumber
&\leq C^{4}_{N,4}b^{\frac{4-N}{2}}_{1}b^{\frac{4-N}{2}}_{2}\|\nabla u\|^{\frac{N}{2}}_{L^{2}(\mathbb{R}^{N})}\|\nabla v\|^{\frac{N}{2}}_{L^{2}(\mathbb{R}^{N})}\\\nonumber
&\leq\frac{1}{2}C^{4}_{N,4}b^{\frac{4-N}{2}}_{1}b^{\frac{4-N}{2}}_{2}\left[\|\nabla u\|^{2}_{L^{2}(\mathbb{R}^{N})}+\|\nabla u\|^{2}_{L^{2}(\mathbb{R}^{N})}\right]^{\frac{N}{2}},
\end{align}
\begin{align}\label{LA13}
\beta\int_{\mathbb{R}^{N}}u^{2}vdx&\leq|\beta|\left(\int_{\mathbb{R}^{N}}|u|^{3}dx\right)^{\frac{2}{3}}\left(\int_{\mathbb{R}^{N}}|v|^{3}dx\right)^{\frac{1}{3}}\\\nonumber
&\leq |\beta|\left[\frac{2}{3}\int_{\mathbb{R}^{N}}|u|^{3}dx+\frac{1}{3}\int_{\mathbb{R}^{N}}|v|^{3}dx\right]\\\nonumber
&\leq|\beta|\left[\frac{2}{3} C^{3}_{N,3}b^{\frac{6-N}{2}}_{1}\|\nabla u\|^{\frac{N}{2}}_{L^{2}(\mathbb{R}^{N})}+\frac{1}{3}C^{3}_{N,3}b^{\frac{6-N}{2}}_{2}\|\nabla v\|^{\frac{N}{2}}_{L^{2}(\mathbb{R}^{N})}\right]\\\nonumber
&\leq|\beta|\left[\left(\frac{2}{3}b^{\frac{6-N}{2}}_{1}+\frac{1}{3}b^{\frac{6-N}{2}}_{2}\right)C^{3}_{N,3}\left[\|\nabla u\|^{2}_{L^{2}(\mathbb{R}^{N})}+\|\nabla v\|^{2}_{L^{2}(\mathbb{R}^{N})}\right]^{\frac{N}{4}}\right].
\end{align}

Next, we show that $m(b_{1},b_{2})<0.$ We now focus on Sobolev subcritical nonlinear Schr\"{o}dinger equation with prescribed $L^{2}$ norm. For fixed $\mu>0,2<p<2+\frac{4}{N}$, we search for $H^{1}$ and $\lambda \in \mathbb{R}$ solving
\begin{equation}\label{KL}
\begin{cases}
-\Delta u+u=\mu |u|^{p-2}u  & \text{in} \ \ \mathbb{R}^{N},\\
\int_{\mathbb{R}^{N}}u^{2}dx=b^{2}.
\end{cases}
\end{equation}
Solutions of \eqref{KL} can be found as critical points of $J_{\mu,p}:H^{1}\rightarrow \mathbb{R}$,
\begin{align}\label{KJ}
J_{\mu,p}=\frac{1}{2}\int_{\mathbb{R}}|\nabla u_{n}|^{2}dx-\frac{\mu}{p}\int_{\mathbb{R}}|u|^{p}dx
\end{align}
constrained on $S_{b}=\{u\in H^{1}|\int_{\mathbb{R}^{N}}u^{2}dx=b^{2}\}$.
\begin{lemma}\label{LA43}
When $N=1$, for any $\rho>0,\beta>0$, we have $$m_{\beta}(b_{1},b_{2})=\inf_{(u,v)\in \mathrm{T}_{b_{1}}\times \mathrm{T}_{b_{2}}}J_{\beta}(u,v)<0.$$
\end{lemma}
\begin{proof}
From \eqref{LA11}, \eqref{LA12} and \eqref{LA13}, we can deduce that
\begin{align*}
J_{\beta}(u,v)&=\frac{1}{2}\int_{\mathbb{R}}(|\nabla u|^{2}+|\nabla v|^{2})dx-\frac{1}{4}\int_{\mathbb{R}}(\mu_{1}u^{4}+\mu_{2}v^{4}+2\rho u^{2}v^{2})dx-\frac{\beta}{2}\int_{\mathbb{R}}u^{2}vdx\\\nonumber
&\geq \frac{1}{2}\int_{\mathbb{R}}(|\nabla u|^{2}+|\nabla v|^{2})dx-\left[\frac{\mu_{1}}{4} C^{4}_{1,4}b^{3}_{1}+\frac{\mu_{2}}{4} C^{4}_{1,4}b^{3}_{2}+\frac{\rho}{4}C^{4}_{1,4}b^{\frac{3}{2}}_{1}b^{\frac{3}{2}}_{2}\right]
\left[\|\nabla u\|^{2}_{L^{2}(\mathbb{R})}+\|\nabla u\|^{2}_{L^{2}(\mathbb{R})}\right]^{\frac{1}{2}}\\
&\quad-\frac{|\beta|}{2}\left[\left(\frac{2}{3}b^{\frac{5}{2}}_{1}+\frac{1}{3}b^{\frac{5}{2}}_{2}\right)C^{3}_{1,3}\left[\|\nabla u\|^{2}_{L^{2}(\mathbb{R})}+\|\nabla v\|^{2}_{L^{2}(\mathbb{R})}\right]^{\frac{1}{4}}\right],
\end{align*}
so $J$ is coercive and in particular $m_{\beta}(b_{1},b_{2})>-\infty.$  Since
\begin{align*}
 J_{\beta}(u,v)&=\frac{1}{2}\int_{\mathbb{R}}(|\nabla u|^{2}+|\nabla v|^{2})dx-\frac{\beta}{2}\int_{\mathbb{R}}u^{2}vdx-\frac{1}{4}\int_{\mathbb{R}}\left(\mu_{1}u^{4}+\mu_{2}v^{4}+2\rho u^{2}v^{2}\right)dx\\
&= J_{\mu_{1},4}(u)+J_{\mu_{2},4}(v)-\frac{1}{2}\int_{\mathbb{R}}\rho u^{2}v^{2}dx-\frac{\beta}{2}\int_{\mathbb{R}}u^{2}vdx,
\end{align*}
where $J_{\mu_{i},4}, i=1,2$ is defined in \eqref{KJ}. Then we choose $u=u_{\mu_{1},p,b_{1}}$ and $v=u_{\mu_{2},4,b_{2}}$, where $u_{\mu_{i},4,b_{i}}(i=1,2)$ is the unique positive solution of \eqref{KL} with $b$ and $\mu$ are replaced by $b_{i}$ and $\mu_{i}$. We have $ J_{\mu_{1},4}(u_{\mu_{i},4,b_{i}})<0,$ (see details in Lemma 2.1 of \cite{CZ21}).
Therefore $$m_{\beta}(b_{1},b_{2})\leq J_{\beta}(u_{\mu_{1},4,b_{1}},u_{\mu_{2},4,b_{2}})<0,$$
so for any $\beta >0$, we obtain $$m_{\beta}(b_{1},b_{2})=\inf_{(u,v)\in \mathrm{T}_{b_{1}}\times \mathrm{T}_{b_{2}}}J_{\beta}(u,v)<0.$$
\end{proof}
\begin{lemma}\label{Lemg}
Let $\{(u_{n},v_{n})\}\subset \mathrm{T}_{b_{1}}\times \mathrm{T}_{b_{2}} $ be a minimizing sequence for $m(b_{1},b_{2})$. Then for $\beta>0$, $\{(|u_{n}|,|v_{n}|)\}$ is also a minimizing sequence.
\begin{proof}
Since $$\int_{\mathbb{R}^{N}}|\nabla |u_{n}||^{2}dx\leq \int_{\mathbb{R}}|\nabla u_{n}|^{2}dx,\ \int_{\mathbb{R}^{N}}|\nabla |v_{n}||^{2}dx\leq \int_{\mathbb{R}}|\nabla v_{n}|^{2}dx,$$$$\int_{\mathbb{R}}u_{n}^{2}v_{n}dx\leq \int_{\mathbb{R}}|u_{n}|^{2}|v_{n}|dx,$$
we have
\begin{align*}
 J_{\beta}(|u_{n}|,|v_{n}|)&=\frac{1}{2}\int_{\mathbb{R}}(|\nabla |u_{n}||^{2}+|\nabla |v_{n}||^{2})dx-\frac{\beta}{2}\int_{\mathbb{R}}|u_{n}|^{2}|v_{n}|dx\\
&\quad-\frac{1}{4}\int_{\mathbb{R}}\left(\mu_{1}|u_{n}|^{4}+\mu_{2}|v_{n}|^{4}+2\rho |u_{n}|^{2}|v_{n}|^{2}\right)dx\\
&\leq \frac{1}{2}\int_{\mathbb{R}}(|\nabla u_{n}|^{2}+|\nabla v_{n}|^{2})dx-\frac{\beta}{2}\int_{\mathbb{R}}u_{n}^{2}v_{n}dx\\
&\quad-\frac{1}{4}\int_{\mathbb{R}}\left(\mu_{1}u_{n}^{4}+\mu_{2}v_{n}^{4}+2\rho u_{n}^{2}v_{n}^{2}\right)dx=J_{\beta}(u_{n},v_{n}).
\end{align*}
Thus, $ J_{\beta}(|u_{n}|,|v_{n}|)\leq J_{\beta}(u_{n},v_{n}).$ Therefore, when $\beta>0$, $\{(|u_{n}|,|v_{n}|)\}$ is also a minimizing sequence.
\end{proof}
\end{lemma}
\begin{lemma}\label{Lema}

(i) If $(d^{n}_{1},d^{n}_{2})$ is such  that $(d^{n}_{1},d^{n}_{2})\rightarrow(d_{1},d_{2})$ as $n\rightarrow +\infty$ with $0\leq d^{n}_{i}\leq b_{i}$ for $i=1,2$, we have $m_{\beta}(d^{n}_{1},d^{n}_{2})\rightarrow m_{\beta}(d_{1},d_{2})$ as $n\rightarrow +\infty$.

(ii) Let $d_{i}\geq 0,\  b_{i}\geq0, i=1,2 $  such that $b^{2}_{1}+d^{2}_{1}=c_{1}^{2},\  b^{2}_{2}+d^{2}_{2}=c_{2}^{2}$, then $m_{\beta}(b_{1},b_{2})+ m_{\beta} (d_{1},d_{2})\geq m_{\beta}(c_{1}, c_{2}).$
\end{lemma}
\begin{proof}
The proof is similar to that of Lemma 3.1 in \cite{GL16}, we omit the details here.
\end{proof}
\begin{lemma}\label{Lem12}
For any $\rho>0,\beta>0$, let $(u_{n},v_{n})\subset H^{1}(\mathbb{R})\times H^{1}(\mathbb{R})$ be a sequence such  that $$J_{\beta}(u_{n},v_{n})\rightarrow m_{\beta}(b_{1},b_{2})\ \text{and}\ \int_{\mathbb{R}}u_{n}^{2}dx= b^{2}_{1}, \ \int_{\mathbb{R}^{N}}v_{n}^{2}dx= b^{2}_{2}, $$
then $\{(u_{n},v_{n})\}$ is relatively compact in $H^{1}(\mathbb{R})\times H^{1}(\mathbb{R})$ up to translations, that is there exists a subsequence $(u_{n_{k}},v_{n_{k}})$, a sequence of points $\{y_{k}\}\subset \mathbb{R}^{N}$ and a function $(\widetilde{u},\widetilde{v})\in \mathrm{T}_{b_{1}}\times \mathrm{T}_{b_{2}}\times H^{1}(\mathbb{R})\times H^{1}(\mathbb{R})$ such that $(u_{n_{k}}(\cdot+y_{k}),v_{n_{k}}(\cdot+y_{k}) )\rightarrow (\widetilde{u},\widetilde{v})$ strongly in $H^{1}(\mathbb{R})\times H^{1}(\mathbb{R}). $
\end{lemma}
\begin{proof}
We use the ideas introduced in \cite{GL16,GL18}. Assume that $(u^{n}_{1},u^{n}_{2})$ is a minimizing sequence associated to the functional $J$ on $\mathrm{T}_{b_{1}}\times \mathrm{T}_{b_{2}}$, from \eqref{Lag} and the coerciveness of functional $J$ on $\mathrm{T}_{b_{1}}\times \mathrm{T}_{b_{2}}$, the sequence $(u^{n}_{1},u^{n}_{2})$ is bounded in $H^{1}(\mathbb{R}^{N})\times H^{1}(\mathbb{R}).$ If $$\sup_{y\in \mathbb{R}}\int_{B(y,R)}(v_{n}^{2}+v_{n}^{2})dx=o(1),\ \text{for some}\ R>0,$$ then by  Vanishing lemma (see Lemma I.1 in \cite{PLLL84}), we have that $(u_{n},v_{n})\rightarrow(0,0)$ in $2<p<2^{\ast}$, contrary  with the fact $m(b_{1},b_{2})<0$, therefore there exists a $\beta_{0}>0$ and a sequence $\{y_{n}\}\subset \mathbb{R}$ such that $$\int_{B(y_{n},R)}(v_{n}^{2}+v_{n}^{2})dx\geq \beta_{0}>0.$$
Since $(u^{n}_{1},u^{n}_{2})$ is bounded in $H^{1}(\mathbb{R})\times H^{1}(\mathbb{R})$, $(u^{n}_{1},u^{n}_{2})$ is weakly convergence in  $H^{1}(\mathbb{R})\times H^{1}(\mathbb{R})$  and local compactness in $L^{2}(\mathbb{R})\times L^{2}(\mathbb{R})$, so $$(u_{n}(x-y_{n}),v_{n}(x-y_{n}))\rightharpoonup (u,v)\neq (0,0)\ \text{in}\  H^{1}(\mathbb{R})\times H^{1}(\mathbb{R}).$$ Let $$w_{n}(x)=u_{n}(x)-u(x+y_{n}),\ \sigma_{n}(x)=v_{n}(x)-v(x+y_{n}),$$ we show that $w_{n}(x)\rightarrow 0 ,\sigma_{n}(x)\rightarrow 0$ in $L^{p}(\mathbb{R})$ for $2<p<2^{\star}$, suppose by contradiction that there exists a sequence $\{z_{n}\}\subset \mathbb{R}^{N}$ such that $$(w_{n}(x-z_{n}), \sigma_{n}(x-z_{n}))\rightharpoonup (w,\sigma)\neq (0,0)\ \text{in}\  H^{1}(\mathbb{R})\times H^{1}(\mathbb{R}).$$ By the Brezis-Lieb lemma and Lemma 2.4 in \cite{GL18}, we have
\begin{align}\label{Lah}
J_{\beta}(u_{n},v_{n})&=J_{\beta}(u_{n}(x-y_{n}),v_{n}(x-y_{n}))\\\nonumber
&=J_{\beta}(w_{n}(x-y_{n}),\sigma_{n}(x-y_{n}))+J_{\beta}(u,v)+o(1)\\\nonumber
&=J_{\beta}(w_{n}(x-z_{n})-w,\sigma_{n}(x-z_{n})-\sigma)+J_{\beta}(w,\sigma)+J_{\beta}(u,v)+o(1),\\\nonumber
\end{align}
\begin{align*}
\|u_{n}(x-y_{n})\|^{2}_{L^{2}(\mathbb{R})}&=\|w_{n}(x-z_{n})\|^{2}_{L^{2}(\mathbb{R})}+\|u\|^{2}_{L^{2}(\mathbb{R})}+o(1)\\\nonumber
&=\|w_{n}(x-z_{n})-w\|^{2}_{L^{2}(\mathbb{R})}+\|w\|^{2}_{L^{2}(\mathbb{R})}+\|u\|^{2}_{L^{2}(\mathbb{R})}+o(1),\\\nonumber
\end{align*}
and
\begin{align*}
\|v_{n}(x-y_{n})\|^{2}_{L^{2}(\mathbb{R})}&=\|\sigma_{n}(x-z_{n})\|^{2}_{L^{2}(\mathbb{R})}+\|v\|^{2}_{L^{2}(\mathbb{R})}+o(1)\\\nonumber
&=\|\sigma_{n}(x-z_{n})-\sigma\|^{2}_{L^{2}(\mathbb{R})}+\|\sigma\|^{2}_{L^{2}(\mathbb{R})}+\|v\|^{2}_{L^{2}(\mathbb{R})}+o(1).\\\nonumber
\end{align*}
So, we obtain
\begin{align}\label{Laj}
\|w_{n}(x-z_{n})-w\|^{2}_{L^{2}(\mathbb{R})}
&=b^{2}_{1}-\|w\|^{2}_{L^{2}(\mathbb{R})}-\|u\|^{2}_{L^{2}(\mathbb{R})}+o(1)\\\nonumber
&:=d^{2}_{1}+o(1),
\end{align}
\begin{align}\label{Lak}
\|\sigma_{n}(x-z_{n})-\sigma\|^{2}_{L^{2}(\mathbb{R})}
&=b^{2}_{2}-\|\sigma\|^{2}_{L^{2}(\mathbb{R})}-\|v\|^{2}_{L^{2}(\mathbb{R})}+o(1)\\\nonumber
&:=d^{2}_{2}+o(1).
\end{align}
By \eqref{Lah}-\eqref{Lak}, (i) of Lemma \ref{Lema} and $J_{\beta}(u_{n},v_{n})\rightarrow m_{\beta}(b_{1},b_{2})$, we obtain
\begin{align}\label{LA35}
m_{\beta}(b_{1},b_{2})\geq m_{\beta}(d_{1},d_{2})+J_{\beta}(w,\sigma)+J_{\beta}(u,v).
\end{align}
If $J_{\beta}(w,\sigma)>m_{\beta}(\|w\|_{L^{2}(\mathbb{R})},\|\sigma\|_{L^{2}(\mathbb{R})} )$ or $J_{\beta}(u,v)>m_{\beta}(\|u\|_{L^{2}(\mathbb{R})},\|v\|_{L^{2}(\mathbb{R})} )$, then by \eqref{Laj}-\eqref{LA35} and (ii) of Lemma \ref{Lema}, we have
\begin{align*}
m_{\beta}(b_{1},b_{2})&> m_{\beta}(d_{1},d_{2})+m_{\beta}(\|w\|_{L^{2}(\mathbb{R})},\|\sigma\|_{L^{2}(\mathbb{R})} )+m_{\beta}(\|u\|_{L^{2}(\mathbb{R})},\|v\|_{L^{2}(\mathbb{R})} )\\
&> m_{\beta}(d_{1},d_{2})+m_{\beta}(\sqrt{b^{2}_{1}-d^{2}_{1}},\sqrt{b^{2}_{2}-d^{2}_{2}} )\geq m_{\beta}(b_{1},b_{2})
\end{align*}
which is impossible. So, $$J_{\beta}(w,\sigma)=m_{\beta}(\|w\|_{L^{2}(\mathbb{R})},\|\sigma\|_{L^{2}(\mathbb{R})} ),\ J_{\beta}(u,v)=m_{\beta}(\|u\|_{L^{2}(\mathbb{R})},\|v\|_{L^{2}(\mathbb{R})} ).$$
Let $\widetilde{w}, \widetilde{\sigma}, \widetilde{u},\widetilde{v}$ be the Schwarz symmetric-decreasing rearrangement of $w,\sigma,u,v$. From Section 3.3 in \cite{Lieb01}, we know that $\widetilde{w}, \widetilde{\sigma}, \widetilde{u},\widetilde{v}$ are nonnegative.  Since $$\|\widetilde{w}\|^{2}_{L^{2}(\mathbb{R})}=\|w\|^{2}_{L^{2}(\mathbb{R})},\ \|\widetilde{\sigma}\|^{2}_{L^{2}(\mathbb{R})}=\|\sigma\|^{2}_{L^{2}(\mathbb{R})},$$
$$\|\widetilde{u}\|^{2}_{L^{2}(\mathbb{R})}=\|u\|^{2}_{L^{2}(\mathbb{R})},\ \|\widetilde{v}\|^{2}_{L^{2}(\mathbb{R})}=\|v\|^{2}_{L^{2}(\mathbb{R})},$$
$$J_{\beta}(\widetilde{w},\widetilde{\sigma})\leq J_{\beta}(w,\sigma),\ J_{\beta}(\widetilde{u},\widetilde{v})\leq J_{\beta}(u,v).$$
By the standard argument as \cite{Lieb01}, we can deduce that $$J_{\beta}(\widetilde{w},\widetilde{\sigma})=m_{\beta}(\|w\|_{L^{2}(\mathbb{R})},\|\sigma\|_{L^{2}(\mathbb{R})} ),\ J_{\beta}(\widetilde{u},\widetilde{v})=m_{\beta}(\|u\|_{L^{2}(\mathbb{R})},\|v\|_{L^{2}(\mathbb{R})} ).$$
Therefore, $(\widetilde{w},\widetilde{\sigma}),\ (\widetilde{u},\widetilde{v}) $ are solutions of system \eqref{23}, by the standard regularity results, we can get $\widetilde{w},\widetilde{\sigma},\widetilde{u},\widetilde{v}\in C^{2}(\mathbb{R}^{2})$. Without restriction, we may assume $w\neq 0$. We divide into two cases.

{\bf Case 1}: $w\neq0$ and $u\neq0. $

From (ii), (iv), (v) of Lemma \ref{LM}, we have $$\int_{\mathbb{R}}|\nabla \{\widetilde{w},\widetilde{u}\}^{\star}|^{2}dx<\int_{\mathbb{R}}|\nabla \widetilde{w}|^{2}dx+\int_{\mathbb{R}}|\nabla \widetilde{u}|^{2}dx\leq\int_{\mathbb{R}}|\nabla w|^{2}dx+\int_{\mathbb{R}}|\nabla u|^{2}dx,$$
\begin{align*}
\int_{\mathbb{R}}| \{\widetilde{w},\widetilde{u}\}^{\star}|^{2}| \{\widetilde{\sigma},\widetilde{v}\}^{\star}|^{2}dx&=\int_{\mathbb{R}}\{|\widetilde{w}|^{2},|\widetilde{u}|^{2}\}^{\star}\{|\widetilde{\sigma}|^{2},|\widetilde{v}|^{2}\}^{\star}dx\\
&\geq \int_{\mathbb{R}}|\widetilde{w}|^{2}|\widetilde{\sigma}|^{2}+|\widetilde{u}|^{2}|\widetilde{v}|^{2}dx\\
&=\int_{\mathbb{R}}\widetilde{(|w|^{2})}\widetilde{(|\sigma|^{2})}+\widetilde{(|u|^{2})}\widetilde{(|v|^{2})}dx\\
&\geq \int_{\mathbb{R}}|w|^{2}|\sigma|^{2}+|u|^{2}|v|^{2}dx,
\end{align*}
\begin{align*}
\int_{\mathbb{R}}| \{\widetilde{w},\widetilde{u}\}^{\star}|^{2} \{\widetilde{\sigma},\widetilde{v}\}^{\star}dx&=\int_{\mathbb{R}}\{|\widetilde{w}|^{2},|\widetilde{u}|^{2}\}^{\star}\{\widetilde{\sigma},\widetilde{v}\}^{\star}dx\\
&\geq \int_{\mathbb{R}}|\widetilde{w}|^{2}\widetilde{\sigma}+|\widetilde{u}|^{2}\widetilde{v}dx\\
&=\int_{\mathbb{R}}\widetilde{(|w|^{2})}\widetilde{(|\sigma|)}+\widetilde{(|u|^{2})}\widetilde{(|v|)}dx\\
&\geq \int_{\mathbb{R}}|w|^{2}|\sigma|+|u|^{2}|v|dx,
\end{align*}
so
\begin{align}\label{LA36}
J_{\beta}(w,\sigma)+J_{\beta}(u,v)>J_{\beta}(\{\widetilde{w},\widetilde{u}\}^{\star},\{\widetilde{\sigma},\widetilde{v}\}^{\star}).
\end{align}
By (iii) of Lemma \ref{LM}, we get
\begin{align}\label{LA37}
\int_{\mathbb{R}}|\{\widetilde{w},\widetilde{u}\}^{\star}|^{2}dx=\int_{\mathbb{R}}(|\widetilde{w}|^{2}+|\widetilde{u}|^{2})dx=\int_{\mathbb{R}}(|w|^{2}+|u|^{2})dx,
\end{align}
\begin{align}\label{LA38}
\int_{\mathbb{R}}|\{\widetilde{\sigma},\widetilde{v}\}^{\star}|^{2}dx=\int_{\mathbb{R}}(|\widetilde{\sigma}|^{2}+|\widetilde{v}|^{2})dx=\int_{\mathbb{R}}(|\sigma|^{2}+|v|^{2})dx.
\end{align}
By \eqref{Laj}-\eqref{LA38},  (iii) of Lemma \ref{LM} and (ii) of Lemma \ref{Lema}, we obtain
\begin{align*}
m_{\beta}(b_{1},b_{2})> m_{\beta}(d_{1},d_{2})+m_{\beta}(\sqrt{b^{2}_{1}-d^{2}_{1}},\sqrt{b^{2}_{2}-d^{2}_{2}} )\geq m_{\beta}(b_{1},b_{2}),
\end{align*}
a contradiction.

{\bf Case 2}: $w \neq0,\ u=0$ and $v \neq 0.$

If $\sigma\neq 0$, we only need to  replaced  $w,u$ by  $\sigma,v$ in Case 1, by the same argument as Case 1, we can get a contradiction.   Thus, we suppose that $\sigma= 0$. By (ii)-(v) of Lemma \ref{LM}, we have
\begin{align}\label{LA39}
J_{\beta}(\{\widetilde{w},0\}^{\star},\{\widetilde{v},0\}^{\star})&
\leq \frac{1}{2}\int_{\mathbb{R}}(|\nabla \widetilde{w}|^{2}+|\nabla \widetilde{v}|^{2})dx-\frac{\mu_{2}}{4}\int_{\mathbb{R}^{N}}|\widetilde{w}|^{4}dx-\frac{\mu_{1}}{4}\int_{\mathbb{R}}|\widetilde{v}|^{4}dx\\\nonumber
&\quad-\frac{\rho}{2}\int_{\mathbb{R}}|\widetilde{w}|^{2}|\widetilde{v}|^{2}dx-\frac{\beta}{2}\int_{\mathbb{R}}|\widetilde{w}|^{2}\widetilde{v}dx\\\nonumber
&\leq \frac{1}{2}\int_{\mathbb{R}}(|\nabla \widetilde{w}|^{2}+|\nabla \widetilde{v}|^{2})dx-\frac{\mu_{2}}{4}\int_{\mathbb{R}}|\widetilde{w}|^{4}dx-\frac{\mu_{1}}{4}\int_{\mathbb{R}}|\widetilde{v}|^{4}dx\\\nonumber
&=J_{\beta}(\widetilde{w},0)+J_{\beta}(0,\widetilde{v})\leq J_{\beta}(w,0)+J_{\beta}(0,v).
\end{align}
By (iii) of Lemma \ref{LM}, we get
\begin{align}\label{LA40}
\int_{\mathbb{R}}|\{\widetilde{w},0\}^{\star}|^{2}dx=\int_{\mathbb{R}}|\widetilde{w}|^{2}dx=\int_{\mathbb{R}}|w|^{2}dx,
\end{align}
\begin{align}\label{LA41}
\int_{\mathbb{R}}|\{\widetilde{v},0\}^{\star}|^{2}dx=\int_{\mathbb{R}}|\widetilde{v}|^{2}dx=\int_{\mathbb{R}}|v|^{2}dx.
\end{align}
By \eqref{Laj}-\eqref{Lak}, \eqref{LA35},\eqref{LA39},\eqref{LA40}\eqref{LA41},  (iii) of Lemma \ref{LM} and (ii) of Lemma \ref{Lema}, we obtain
\begin{align*}
m_{\beta}(b_{1},b_{2})> m_{\beta}(d_{1},d_{2})+m_{\beta}(\sqrt{b^{2}_{1}-d^{2}_{1}},\sqrt{b^{2}_{2}-d^{2}_{2}} )\geq m_{\beta}(b_{1},b_{2}),
\end{align*}
a contradiction. The contradictions obtained in Case 1 and Case 2 indicate that $w_{n}(x)=u_{n}(x)-u(x+y_{n})\rightarrow 0,\ \sigma_{n}(x)=v_{n}(x)-v(x+y_{n})\rightarrow 0$ in $L^{p}(\mathbb{R})$ for $2<p<2^{\ast}$.
\end{proof}
\begin{proof}[{\bf Proof of Theorem \ref{Th4}}]
Let $\{w_{n},\sigma_{n}\}$ be any minimizing sequence for the functional $J$ on $\mathrm{T}_{b_{1}}\times \mathrm{T}_{b_{2}}$. By Lemma \ref{Lem12}, we know that there exists $\{y_{n}\}\subset \mathbb{R}$ such that $(w_{n},\sigma_{n})\rightharpoonup (w,\sigma)$ in $H^{1}( \mathbb{R})\times H^{1}( \mathbb{R})$ and $(w_{n},\sigma_{n})\rightarrow (w, \sigma)$ in $L^{p}(\mathbb{R})\times L^{p}(\mathbb{R}) $ for $2<p<2^{\ast}$. Hence, by the weakly lower semi-continuity of the norm, we have
\begin{align}\label{AK}
J_{\beta}(w,\sigma)\leq m_{\beta}(b_{1},b_{2})<0.
\end{align}
 To show the compactness of $\{w_{n},\sigma_{n}\}$ in $H^{1}( \mathbb{R})\times H^{1}( \mathbb{R})$, it  suffices to prove that $(w,\sigma)\in \mathrm{T}_{b_{1}}\times \mathrm{T}_{b_{2}}$. If $\int_{\mathbb{R}}|w|^{2}dx=b^{2}_{1}$ and  $\ \int_{\mathbb{R}^{N}}|\sigma|^{2}dx=b^{2}_{2}$, we are done.  Assume by contradiction that there exists $\overline{b}_{1}<b_{1}$ and $\overline{b}_{2}<b_{2} $ such that $\int_{\mathbb{R}}|w|^{2}dx=\overline{b}^{2}_{1}<b^{2}_{1},$ or $\ \int_{\mathbb{R}}|\sigma|^{2}dx=\overline{b}^{2}_{2}<b^{2}_{2}$. Then, by the definition of $m_{\beta}(b_{1},b_{2})$,  we have $$m_{\beta}(\overline{b}_{1},\overline{b}_{2})\leq J_{\beta}(w,\sigma).$$

{\bf Case 1} If $\int_{\mathbb{R}}|w|^{2}dx=\overline{b}^{2}_{1}<b^{2}_{1}$ and $\ \int_{\mathbb{R}}|\sigma|^{2}dx=\overline{b}^{2}_{2}<b^{2}_{2}$.
 By   Lemma \ref{LA43} and (ii) of Lemma \ref{Lema}, we have $$J_{\beta}(w,\sigma)\leq m_{\beta}(b_{1},b_{2})\leq m_{\beta}(\overline{b}_{1},\overline{b}_{2})+m_{\beta}\big(\sqrt{b^{2}_{1}-\overline{b}^{2}_{1}},\sqrt{b^{2}_{2}-\overline{b}^{2}_{2}} \big)<m_{\beta}(\overline{b}_{1},\overline{b}_{2})\leq J_{\beta}(w,\sigma),$$
a contradiction.

{\bf Case 2} If $\int_{\mathbb{R}}|w|^{2}dx=b^{2}_{1}$ and $\ \int_{\mathbb{R}}|\sigma|^{2}dx=\overline{b}^{2}_{2}<b^{2}_{2}$.
By   Lemma \ref{LA43} and (ii) of Lemma \ref{Lema} and $m(0,\sqrt{b^{2}_{2}-\overline{b}^{2}_{2}} )<0,$ we have $$J_{\beta}(w,\sigma)\leq m_{\beta}(b_{1},b_{2})\leq m_{\beta}(b_{1},\overline{b}_{2})+m_{\beta}\big(0,\sqrt{b^{2}_{2}-\overline{b}^{2}_{2}} \big)<m_{\beta}(b_{1},\overline{b}_{2})\leq J_{\beta}(w,\sigma),$$
a contradiction.

{\bf Case 3} If $\int_{\mathbb{R}}|w|^{2}dx=\overline{b}^{2}_{1}<b^{2}_{1}$ and $\ \int_{\mathbb{R}}|\sigma|^{2}dx=b^{2}_{2}$.
By   Lemma \ref{LA43} and (ii) of Lemma \ref{Lema} and $m(\sqrt{b^{2}_{1}-\overline{b}^{2}_{1}},0 )<0,$ we have $$J_{\beta}(w,\sigma)\leq m_{\beta}(b_{1},b_{2})\leq m_{\beta}(\overline{b}_{1},b_{2})+m_{\beta}\big(\sqrt{b^{2}_{1}-\overline{b}^{2}_{1}},0 \big)<m_{\beta}(\overline{b}_{1},b_{2})\leq J_{\beta}(w,\sigma),$$
a contradiction.

Therefore, $$(w,\sigma)\in \mathrm{T}_{b_{1}}\times \mathrm{T}_{b_{2}}.$$
\end{proof}

\section{Proof of Theorem \ref{Th3}}\label{sec15}
To show our results, we first define following constant $A$,
\begin{align}\label{A}
 A=\inf_{(u,v)\in \mathrm{T}_{b_{1}}\times \mathrm{T}_{b_{2}}}\frac{\int_{\mathbb{R}^{2}}(|\nabla u|^{2}+|\nabla v|^{2})dx}{\int_{\mathbb{R}^{2}}\left(\mu_{1}u^{4}+\mu_{2}v^{4}+2\rho u^{2}v^{2}\right)dx}.
 \end{align}
 From  Gagliardo-Nirenberg-Sobolev inequality in \cite{MIW},
\begin{equation*}
\|u\|^{4}_{L^{4}(\mathbb{R}^{2})}\leq \frac{2}{\|Q(x)\|^{2}_{L^{2}}}\|\nabla u\|_{L^{2}(\mathbb{R}^{2})}^{2}\| u\|_{L^{2}(\mathbb{R}^{2})}^{2}\ \ \text{for all} \ \ u\in H^{1}(\mathbb{R}^{2}),
\end{equation*}
 where $Q(x)$ is the unique positive solution of $$-\Delta u +u =u^{3}, \ u \in H^{1}(\mathbb{R}^{2}),$$
 and the identity is achieved at $u(x)=Q(|x|)$. It is easy to see that $Q(|x|)$ satisfies that $$\|\nabla Q\|_{L^{2}(\mathbb{R}^{2})}^{2}=\| Q\|_{L^{2}(\mathbb{R}^{2})}^{2}=\frac{1}{2}\|Q\|^{4}_{L^{4}(\mathbb{R}^{2})}.$$
From above Gagliardo-Nirenberg-Sobolev inequality, we have
$$\int_{\mathbb{R}^{2}}(\mu_{1}u^{4}+\mu_{2}v^{4}+2\rho u^{2}v^{2})dx\leq \frac{2}{\|Q(x)\|^{2}_{L^{2}}}\left[\mu_{1} b^{2}_{1}+\mu_{2} b^{2}_{2}+2\rho b_{1}b_{2}\right]\int_{\mathbb{R}^{2}}(|\nabla u|^{2}+|\nabla v|^{2})dx,$$
 so
 \begin{align}
 A\geq \frac{1}{C}>0.
 \end{align}
 Next, we prove that
 \begin{align}\label{LEG}
\frac{\| Q\|_{L^{2}(\mathbb{R}^{2})}^{2}}{2\max\{(\mu_{1}+\rho)b^{2}_{1}, (\mu_{2}+\rho)b^{2}_{2}\}}\leq A\leq\frac{\frac{1}{2}(b^{2}_{1}+b^{2}_{2})\| Q\|_{L^{2}(\mathbb{R}^{2})}^{2}}{\mu_{1}b^{4}_{1}+\mu_{2}b^{4}_{2}+2\rho b^{2}_{1}b^{2}_{2}}.
 \end{align}
 It is easy to see that  $(\frac{b_{1}Q(x)}{\|Q(x)\|_{L^{2}}},\frac{b_{2}Q(x)}{\|Q(x)\|_{L^{2}}}) \in \mathrm{T}_{b_{1}}\times \mathrm{T}_{b_{2}}. $
 Thus,
 \begin{align*}
 A\leq\frac{b^{2}_{1}+b^{2}_{2}}{\mu_{1}b^{4}_{1}+\mu_{2}b^{4}_{2}+2\rho b^{2}_{1}b^{2}_{2}}\frac{\|Q(x)\|^{2}_{L^{2}}\int_{\mathbb{R}^{2}}|\nabla Q|^{2}dx}{\int_{\mathbb{R}^{2}}Q^{4}dx}=\frac{\frac{1}{2}(b^{2}_{1}+b^{2}_{2})\| Q\|_{L^{2}(\mathbb{R}^{2})}^{2}}{\mu_{1}b^{4}_{1}+\mu_{2}b^{4}_{2}+2\rho b^{2}_{1}b^{2}_{2}}.
 \end{align*}
  On the other hand, for any $(u,v)\in H^{1}(\mathbb{R}^{2})\times H^{1}(\mathbb{R}^{2}) $ and $\int_{\mathbb{R}^{2}}u^{2}dx=b^{2}_{1},\ \ \int_{\mathbb{R}^{2}}v^{2}dx=b^{2}_{2}.$ Then from Gagliardo-Nirenberg-Sobolev inequality, we have
\begin{align*}
\frac{\| Q\|_{L^{2}(\mathbb{R}^{2})}^{2}}{2\max\{(\mu_{1}+\rho)b^{2}_{1}, (\mu_{2}+\rho)b^{2}_{2}\}}&\leq\frac{\int_{\mathbb{R}^{2}}(|\nabla u|^{2}+|\nabla v|^{2})dx}{\int_{\mathbb{R}^{2}}(\mu_{1}+\rho )u^{4}dx+\int_{\mathbb{R}^{2}}(\mu_{2}+\rho )v^{4}dx}\\
&\leq\frac{\int_{\mathbb{R}^{2}}(|\nabla u|^{2}+|\nabla v|^{2})dx}{\int_{\mathbb{R}^{2}}\left(\mu_{1}u^{4}+\mu_{2}v^{4}+2\rho u^{2}v^{2}\right)dx}.\\
\end{align*}
From \eqref{LA11}, \eqref{LA12} and the definition of $A$, we have
\begin{align}
J_{0}(u,v)&=\frac{1}{2}\int_{\mathbb{R}^{2}}(|\nabla u|^{2}+|\nabla v|^{2})dx-\frac{1}{4}\int_{\mathbb{R}^{2}}(\mu_{1}u^{4}+\mu_{2}v^{4}+2\rho u^{2}v^{2})dx\\\nonumber
&\geq\left[\frac{1}{2}-\frac{1}{4A}\right]\int_{\mathbb{R}^{2}}(|\nabla u|^{2}+|\nabla v|^{2})dx,
\end{align}
so $J_{0}(u,v)$ is bounded from below on $\mathrm{T}_{b_{1}}\times \mathrm{T}_{b_{2}} $ provided that
\begin{align}\label{AS}
A>\frac{1}{2}.
\end{align}
Let $u\in \mathrm{T}_{b}$  and $t\star u=e^{\frac{Nt}{2}}u(e^{t}x)$, then $t\star u\in \mathrm{T}_{b}$, we define
 \begin{equation}\label{k}
 t\star(u,v)=(t\star u,t\star v).
  \end{equation}
Under the condition of $\max\{(\mu_{1}+\rho)b^{2}_{1}, (\mu_{2}+\rho)b^{2}_{2}\}<\| Q\|_{L^{2}(\mathbb{R}^{2})}^{2}$, we have $$A>\frac{1}{2},$$ so $$\frac{1}{2}\int_{\mathbb{R}^{2}}(|\nabla u|^{2}+|\nabla v|^{2})dx-\frac{1}{4}\int_{\mathbb{R}^{2}}\left(\mu_{1}u^{4}+\mu_{2}v^{4}+2\rho u^{2}v^{2}\right)dx>0.$$ Therefore
\begin{align*}
\frac{1}{2}\beta\int_{\mathbb{R}^{2}}u^{2}vdx=\frac{1}{2}\int_{\mathbb{R}^{2}}(|\nabla u|^{2}+|\nabla v|^{2})dx-\frac{1}{4}\int_{\mathbb{R}^{2}}\left(\mu_{1}u^{4}+\mu_{2}v^{4}+2\rho u^{2}v^{2}\right)dx>0.
\end{align*}
Since  \begin{align*}
\Psi_{u,v}(t)=J_{\beta}(t\star (u,v))&=\frac{e^{2t}}{2}\bigg[\int_{\mathbb{R}^{2}}(|\nabla u|^{2}+|\nabla v|^{2})dx-\frac{1}{2}\int_{\mathbb{R}^{2}}\left(\mu_{1}u^{4}+\mu_{2}v^{4}+2\rho u^{2}v^{2}\right)dx\bigg]\\\nonumber
&-\frac{\beta}{2}e^{t}\int_{\mathbb{R}^{2}}u^{2}vdx.
\end{align*}




$$
J_{\beta}(u,v)=\frac{1}{2}\int_{\mathbb{R}^{N}}(|\nabla u|^{2}+|\nabla v|^{2})dx-\frac{1}{4}\int_{\mathbb{R}^{N}}(\mu_{1}u^{4}+\mu_{2}v^{4}+2\rho u^{2}v^{2})dx-\frac{\beta}{2}\int_{\mathbb{R}^{N}}u^{2}vdx.
$$
\begin{lemma}\label{LL}
Under the condition of $\max\{(\mu_{1}+\rho)b^{2}_{1}, (\mu_{2}+\rho)b^{2}_{2}\}<\| Q\|_{L^{2}(\mathbb{R}^{2})}^{2}$, we have
$$-\infty <m_{\beta}(b_{1},b_{2})=\inf_{(u,v)\in \mathrm{T}_{b_{1}}\times \mathrm{T}_{b_{2}}}J_{\beta}(u,v)<0.$$
\end{lemma}
\begin{proof}

If $ \frac{\| Q\|_{L^{2}(\mathbb{R}^{2})}^{2}}{\max\{(\mu_{1}+\rho)b^{2}_{1}, (\mu_{2}+\rho)b^{2}_{2}\}}> 1$ and $\beta>0$, we have $A>\frac{1}{2}$, thus
\begin{align}\label{Lag}
J_{\beta}(u,v)&=\frac{1}{2}\int_{\mathbb{R}^{2}}(|\nabla u|^{2}+|\nabla v|^{2})dx-\frac{1}{4}\int_{\mathbb{R}^{2}}(\mu_{1}u^{4}+\mu_{2}v^{4}+2\rho u^{2}v^{2})dx
-\frac{\beta}{2}\int_{\mathbb{R}^{2}}u^{2}vdx\\\nonumber
&\geq\left[\frac{1}{2}-\frac{1}{4A}\right]\int_{\mathbb{R}^{2}}(|\nabla u|^{2}+|\nabla v|^{2})dx\\\nonumber
&\quad-\frac{\beta}{2}\left[\left(\frac{2}{3}b^{2}_{1}+\frac{1}{3}b^{2}_{2}\right)C^{3}_{2,3}\left[\|\nabla u\|^{2}_{L^{2}(\mathbb{R}^{2})}+\|\nabla v\|^{2}_{L^{2}(\mathbb{R}^{2})}\right]^{\frac{1}{2}}\right],
\end{align}
so $J_{\beta}(u,v)$ is coercive on $\mathrm{T}_{b_{1}}\times \mathrm{T}_{b_{2}}$ and $$m_{\beta}(b_{1},b_{2})=\inf_{(u,v)\in \mathrm{T}_{b_{1}}\times \mathrm{T}_{b_{2}}}J_{\beta}(u,v)>-\infty.$$
Next, we prove $$m_{\beta}(b_{1},b_{2})=\inf_{(u,v)\in \mathrm{T}_{b_{1}}\times \mathrm{T}_{b_{2}}}J_{\beta}(u,v)<0.$$
It is easy to see that
\begin{align*}
J_{\beta}(t\star (u,v))&=\frac{e^{2t}}{2}\int_{\mathbb{R}^{2}}(|\nabla u|^{2}+|\nabla v|^{2})dx-\frac{\beta}{2}e^{t}\int_{\mathbb{R}^{2}}u^{2}vdx\\\nonumber
&\quad-\frac{e^{2t}}{4}\int_{\mathbb{R}^{2}}\left(\mu_{1}u^{4}+\mu_{2}v^{4}+2\rho u^{2}v^{2}\right)dx,
\end{align*}
when $\beta>0$ we choose $u>0,v>0$ and  $(u,v)\in \mathrm{T}_{b_{1}}\times \mathrm{T}_{b_{2}} $,  it is easy to see $(t\star u, t\star v)\in \mathrm{T}_{b_{1}}\times \mathrm{T}_{b_{2}}$  and $$\lim_{t\rightarrow-\infty}J_{\beta}(t\star (u,v))<0,$$
when $\beta<0$ we choose $u>0,v<0$ and  $(u,v)\in \mathrm{T}_{b_{1}}\times \mathrm{T}_{b_{2}} $,  it is easy to see $(t\star u, t\star v)\in \mathrm{T}_{b_{1}}\times \mathrm{T}_{b_{2}}$  and $$\lim_{t\rightarrow-\infty}J_{\beta}(t\star (u,v))<0.$$
Thus for any $\beta \in \mathbb{R}\setminus \{0\}$, we obtain
 $$m_{\beta}(b_{1},b_{2})=\inf_{(u,v)\in \mathrm{T}_{b_{1}}\times \mathrm{T}_{b_{2}}}J_{\beta}(u,v)<0.$$
\end{proof}
\begin{proof}[{\bf Proof of Theorem \ref{Th3}}]
When $\max\{(\mu_{1}+\rho)b^{2}_{1}, (\mu_{2}+\rho)b^{2}_{2}\}<\| Q\|_{L^{2}(\mathbb{R}^{2})}^{2} $, by the definition of $A$, we have $A>\frac{1}{2}$ . Thus, if $\beta<0$  and there exists a positive solution $(u,v)$ to  \eqref{23}-\eqref{24}, we have
\begin{align*}
\frac{1}{2}\beta\int_{\mathbb{R}^{2}}u^{2}vdx&=\int_{\mathbb{R}^{2}}(|\nabla u|^{2}+|\nabla v|^{2})dx-\frac{1}{2}\int_{\mathbb{R}^{2}}\left(\mu_{1}u^{4}+\mu_{2}v^{4}+2\rho u^{2}v^{2}\right)dx\geq 0,
\end{align*}
a contradiction.

Next, we consider $\beta>0$. When $\max\{(\mu_{1}+\rho)b^{2}_{1}, (\mu_{2}+\rho)b^{2}_{2}\}<\| Q\|_{L^{2}(\mathbb{R}^{2})}^{2}$, from Lemma \ref{LL}, we get
$J_{\beta}(u,v)$ is coercive on $\mathrm{T}_{b_{1}}\times \mathrm{T}_{b_{2}}$ and $$m_{\beta}(b_{1},b_{2})=\inf_{(u,v)\in \mathrm{T}_{b_{1}}\times \mathrm{T}_{b_{2}}}J_{\beta}(u,v)>-\infty.$$  Let $\{(u_{n},v_{n})\}\subset\mathrm{T}_{b_{1}}\times \mathrm{T}_{b_{2}} $  be any minimizing sequence for $m_{\beta}(b_{1},b_{2})$. By Lemma \ref{Lemg}, we know that $\{|u_{n}|,|v_{n}|\}$ is also a minimizing sequence. Then by taking $\{(|u_{n}|,|v_{n}|)\}$ and adapting the Schwarz symmetrization to$\{(|u_{n}|,|v_{n}|)\}$ if necessary, we can obtain a new minimizing sequence (up to a subsequence), such that$\{(u_{n},v_{n})\}$ are all real valued, nonnegative, radially symmetric. Since $\{(u_{n},v_{n})\}$ is bounded in $H_{r}^{1}(\mathbb{R}^{2})\times H_{r}^{1}(\mathbb{R}^{2})$. By the Sobolev embedding theorem, we have $H_{r}^{1}(\R^2)\hookrightarrow\hookrightarrow L^{p}_r(\mathbb{R}^{2})$ for $2<p<+\infty$, thus there exists a $(u,v)\in H_{r}^{1}(\mathbb{R}^{2})\times H_{r}^{1}(\mathbb{R}^{2}) $ such that $(u_{n},v_{n})\rightharpoonup (u,v)$ in $H_{r}^{1}(\mathbb{R}^{2})\times H_{r}^{1}(\mathbb{R}^{2})$, $(u_{n},v_{n})\rightarrow (u,v)$ in $L^{p}_r(\mathbb{R}^{2})\times L^{p}_r(\mathbb{R}^{2})\ \ \text{for }\  2<p<+\infty$ and $(u_{n},v_{n})\rightarrow (u,v)$ a.e in $\mathbb{R}^{2}$. Hence $u,v\geq0$ are radial functions.

 {\bf Step 1: Prove $\lambda_{1}>0$ and $\lambda_{2}>0$}. By Ekeland's variational principle yields in a standard way the existence of a new minimizing sequence, which is also a Palais-Smale sequence for $J_{\beta}$ on $\mathrm{T}_{b_{1}}\times \mathrm{T}_{b_{2}}.$
 So $J_{\beta}'\mid_{\mathrm{T}_{b_{1}}\times \mathrm{T}_{b_{2}}}(u_{n},v_{n})\rightarrow 0$, by the Lagrange multipliers rule, we know that there exists a sequence
$(\lambda_{1,n},\lambda_{2,n})\in \mathbb{R}^{2}$ such that
\begin{equation}\label{LAA1}
\begin{aligned}
&\int_{\mathbb{R}^{2}}(\nabla u_{n}\nabla \varphi+\nabla v_{n}\nabla \psi) dx +\int_{\mathbb{R}^{2}}(\lambda _{1,n}u_{n}\varphi+\lambda_{2,n}v_{n}\psi)dx\\
&-\int_{\mathbb{R}^{2}}(\mu_{1}u_{n}^{3}\varphi+\mu_{2}v_{n}^{3}\psi+\rho v_{n}^{2}u_{n}\varphi+\rho u_{n}^{2}v_{n}\psi)dx\\
 &-\beta\int_{\mathbb{R}^{2}} u_{n}\varphi v_{n}-\frac{\beta}{2}\int_{\mathbb{R}^{2}} u_{n}^{2}\psi=o(1)\|(\psi,\varphi)\|_{H^{1}(\mathbb{R}^{2})\times H^{1}(\mathbb{R}^{2})} \ \  \text{in} \ \ \mathbb{R}^{2},
\end{aligned}
\end{equation}
for every $(\varphi,\psi)\in H^{1}(\mathbb{R}^{2})\times H^{1}(\mathbb{R}^{2})$. We claim both $\lambda_{1,n}$ and $\lambda_{2,n}$ are bounded sequence, and at least one of them is converging, up to a subsequence, to a strictly negative value. Indeed, we can using $(u_{n},0) $ and $(0,v_{n}) $ as text function in \eqref{LAA1}, we have $$\int_{\mathbb{R}^{2}}\lambda _{1,n}u^{2}_{n}dx=-\int_{\mathbb{R}^{2}}|\nabla u_{n}|^{2} dx +
\int_{\mathbb{R}^{2}}(\mu_{1}u_{n}^{4}+\rho v_{n}^{2}u^{2}_{n})dx +\beta\int_{\mathbb{R}^{2}} u^{2}_{n} v_{n}+o(1)\|\varphi\|_{H^{1}(\mathbb{R}^{2})} ,$$
$$ \int_{\mathbb{R}^{2}}\lambda _{2,n}v^{2}_{n}dx=-\int_{\mathbb{R}^{2}}|\nabla v_{n}|^{2} dx
+\int_{\mathbb{R}^{2}}(\mu_{2}v_{n}^{4}+\rho v_{n}^{2}u^{2}_{n})dx +\frac{\beta}{2}\int_{\mathbb{R}^{2}} u^{2}_{n} v_{n}+o(1)\|\psi\|_{H^{1}(\mathbb{R}^{2})} ,$$
so
\begin{equation}\label{LAA4}
\begin{aligned}
\int_{\mathbb{R}^{2}}(\lambda _{1,n}u^{2}_{n}+\lambda _{2,n}v^{2}_{n})dx &=-\int_{\mathbb{R}^{2}}(|\nabla u_{n}|^{2} +|\nabla v_{n}|^{2})dx+\int_{\mathbb{R}^{2}}(\mu_{1}u_{n}^{4}+\mu_{2}v_{n}^{4}+2\rho v_{n}^{2}u^{2}_{n})dx\\
&\quad+\frac{3}{2}\beta\int_{\mathbb{R}^{2}}u_{n}^{2}v_{n}dx.
\end{aligned}
\end{equation}
By   \eqref{LA11}-\eqref{LA13} and  the boundedness of $\{(u_{n},v_{n})\}$, we can deduce that $(\lambda_{1,n},\lambda_{2,n})$ is bounded, hence up to a subsequence $(\lambda_{1,n},\lambda_{2,n})\rightarrow(\lambda_{1},\lambda_{2})\in \mathbb{R}^{2}$, passing to limits in \eqref{LAA1}, we can deduce that $(u,v)$ is a nonnegative solutions of  \eqref{23} .  Therefore
 \begin{align*}
\int_{\mathbb{R}^{2}}(|\nabla u|^{2}+|\nabla v|^{2})dx+\int_{\mathbb{R}^{2}}(\lambda_{1} u^{2}+\lambda_{2} v^{2})dx
=\int_{\mathbb{R}^{2}}\left(\mu_{1}u^{4}+\mu_{2}v^{4}+2\rho u^{2}v^{2}\right)dx+\frac{3}{2}\beta\int_{\mathbb{R}^{2}}u^{2}vdx.
\end{align*}
 On the one hand, we multiply the first equation of \eqref{23} by $x\cdot\nabla u$ and the second equation of \eqref{23} by $x\cdot\nabla v$, integrate  by parts, we can get
\begin{align}\label{LAA2}
\int_{\mathbb{R}^{2}}(\lambda_{1} u^{2}+\lambda_{2} v^{2})dx=\frac{1}{2}\int_{\mathbb{R}^{2}}\left(\mu_{1}u^{4}+\mu_{2}v^{4}+2\rho u^{2}v^{2}\right)dx+\beta\int_{\mathbb{R}^{2}}u^{2}vdx.
\end{align}
Since $P_{b_{1},b_{2}}(u_{n},v_{n})\rightarrow 0 $, which implies that
\begin{equation}\label{LAA3}
\begin{aligned}
\int_{\mathbb{R}^{2}}(|\nabla u_{n}|^{2}+|\nabla v_{n}|^{2})dx
-\frac{1}{2}\int_{\mathbb{R}^{2}}\left(\mu_{1}u_{n}^{4}+\mu_{2}v_{n}^{4}+2\rho u_{n}^{2}v_{n}^{2}\right)dx
-\frac{1}{2}\beta\int_{\mathbb{R}^{2}}u_{n}^{2}v_{n}dx=o_{n}(1).
\end{aligned}
\end{equation}
Together \eqref{LAA4} with \eqref{LAA3}, we can get
\begin{align}\label{LAA5}
\lambda_{1,n} b^{2}_{1}+\lambda_{2,n} b^{2}_{2}
=\frac{1}{2}\int_{\mathbb{R}^{2}}\left(\mu_{1}u_{n}^{4}+\mu_{2}v_{n}^{4}+2\rho u_{n}^{2}v_{n}^{2}\right)dx+\beta\int_{\mathbb{R}^{2}}u_{n}^{2}v_{n}dx,
\end{align}
when $\beta>0$, it is easy to see that  at least one sequence of $(\lambda_{i,n})$ is positive and bounded away from 0.
Let $n\rightarrow+\infty$ in \eqref{LAA5}, we have
\begin{align}\label{LAA51}
\lambda_{1} b^{2}_{1}+\lambda_{2} b^{2}_{2}
=\frac{1}{2}\int_{\mathbb{R}^{2}}\left(\mu_{1}u^{4}+\mu_{2}v^{4}+2\rho u^{2}v^{2}\right)dx+\beta\int_{\mathbb{R}^{2}}u^{2}vdx.
\end{align}
We claim that if $\lambda_{1}>0$(resp.$\lambda_{2}>0)$, then $\lambda_{2}>0$(resp.$\lambda_{1}>0$). Indeed, we know that at least one sequence of $(\lambda_{i})$ is positive and bounded away from 0. If $\lambda_{2}>0$, now we argue by contradiction and assume that $\lambda_{1}\leq0$, then $$-\Delta u=-\lambda_{1}u+\beta uv+\mu_{1}u^{3}+\rho v^{2}u\geq0.$$
Using a Liouville type theorem[\cite{NI14}, Lemma A.2], we can deduce that $u=0$. So, $v$ satisfies that
\begin{equation*}
\begin{cases}

-\Delta v+\lambda_{2}v= \mu_{2}v^{3},& \text{in} \ \ \mathbb{R}^{2},\\
v>0,&\text{in} \ \ \mathbb{R}^{2},\\
\int_{\mathbb{R}^{3}}v^{2}dx=b^{2}_{2},& \text{in} \ \ \mathbb{R}^{2}.\\
\end{cases}
\end{equation*}
Therefore, when $0<b_{2}\leq \| Q\|_{L^{2}(\mathbb{R}^{2})}$ we have
\begin{align*}
m_{\beta}(b_{1},b_{2})&=\lim_{n\rightarrow+\infty}J_{\beta}(u_{n},v_{n})=\lim_{n\rightarrow+\infty}J_{\beta}(u_{n},v_{n}-v)+J(0,v)\\
&=\lim_{n\rightarrow+\infty}\bigg[\frac{1}{2}\int_{\mathbb{R}^{2}}(|\nabla u_{n} |^{2}+|\nabla (v_{n}-v)|^{2})dx-\frac{1}{4}\int_{\mathbb{R}^{2}}\Big(\mu_{1}u_{n}^{4}+\mu_{2}(v_{n}-v)^{4}+2\rho u_{n}^{2}(v_{n}-v)^{2}\Big)dx\\
&\quad-\frac{1}{2}\beta\int_{\mathbb{R}^{2}}u_{n}^{2}(v_{n}-v)dx\bigg]+J(0,v)\geq\lim_{n\rightarrow+\infty}\frac{1}{2}\int_{\mathbb{R}^{2}}(|\nabla (u_{n}-u)|^{2}+|\nabla v_{n}|^{2})dx+m_{\beta}(0,b_{2})\geq0,
\end{align*}
in contradiction with $m_{\beta}(b_{1},b_{2})<0$, where we use the fact that $m_{\beta}(0,b_{2})=0$ for $0<b_{2}\leq \| Q\|_{L^{2}(\mathbb{R}^{2})}$, (see Theorem 1.2 in \cite{LZ20}). Thus, $\lambda_{1}>0,\ \lambda_{2}>0.$

If $\lambda_{1}>0$, now we argue by contradiction and assume that $\lambda_{2}\leq0$, then $$-\Delta v=-\lambda_{2}v \frac{\beta}{2}u^{2}+\mu_{2}v^{3}+\rho u^{2}v\geq0.$$
Using a Liouville type theorem[\cite{NI14}, Lemma A.2], we can deduce that $v=0$. So, by the structure of system \eqref{23}, we get $u=0$,  which is impossible.

{\bf Step 2: Prove the $L^{2}$ convergence}. Indeed, it is easy to see that $\int_{\mathbb{R}^{2}}u^{2}\leq b^{2}_{1},\ \int_{\mathbb{R}^{2}}v^{2}\leq b^{2}_{2}.$  From \eqref{LAA2} and \eqref{LAA5}, we have $$\lambda_{1}\left(b^{2}_{1}-\int_{\mathbb{R}^{2}} u^{2}dx\right)+\lambda_{2}\left(b^{2}_{2}-\int_{\mathbb{R}^{2}} v^{2}dx\right)=0,$$
so
\begin{align*}
\int_{\mathbb{R}^{2}}u^{2}= b^{2}_{1},\ \int_{\mathbb{R}^{2}}v^{2}= b^{2}_{2}.
\end{align*}

{\bf Step 3: Prove the $H^{1}$ convergence}.
 It is easy to see that
\begin{align*}
m_{\beta}(b_{1},b_{2})&=\lim_{n\rightarrow +\infty}J_{\beta}(u_{n},v_{n})\\
&=\lim_{n\rightarrow +\infty}\bigg[\frac{1}{2}\int_{\mathbb{R}^{2}}(|\nabla u_{n}|^{2}+|\nabla v_{n}|^{2})dx-\frac{1}{4}\int_{\mathbb{R}^{2}}\left(\mu_{1}u_{n}^{4}+\mu_{2}v_{n}^{4}+2\rho u_{n}^{2}v_{n}^{2}\right)dx-\frac{1}{2}\beta\int_{\mathbb{R}^{2}}u_{n}^{2}v_{n}dx\bigg].
\end{align*}

{\bf Case 1} If $u=0,v=0$, by compact Sobolev embedding,  we have
\begin{align*}
m_{\beta}(b_{1},b_{2})=\lim_{n\rightarrow +\infty}J_{\beta}(u_{n},v_{n})\geq\lim_{n\rightarrow +\infty}\frac{1}{2}\int_{\mathbb{R}^{2}}(|\nabla u_{n}|^{2}+|\nabla v_{n}|^{2})dx\geq0.
\end{align*}
contradict with Lemma \ref{LL}.

{\bf Case 2} If $u\neq0,v=0$, indeed, if $v=0$, by the structure of system \eqref{23}, we get $u=0$, so Case 2 doesn't happen.

{\bf Case 3} If $u=0,\ v\neq0$, let $$\overline{u}_{n}=u_{n},\ \overline{v}_{n}=v_{n}-v,$$
from  \cite[Lemma 2.4]{GL18}, we get
\begin{equation}\label{Bag3}
\int_{\mathbb{R}^{2}}|\overline{u}_{n}|^{2}|\overline{v}_{n}|^{2}dx=\int_{\mathbb{R}^{2}}|u_{n}|^{2}|v_{n}|^{2}dx-\int_{\mathbb{R}^{2}}|u|^{2}|v|^{2}dx+o(1).
\end{equation}
By Brezis Lieb Lemma in \cite{WM}, we can obtain that
\begin{equation}\label{Bag}
\int_{\mathbb{R}^{2}}|\overline{u}_{n}|^{2}\overline{v}_{n}dx=\int_{\mathbb{R}^{2}}|u_{n}|^{2}v_{n}dx-\int_{\mathbb{R}^{2}}|u|^{2}vdx+o(1).
\end{equation}
So, from \eqref{Bag3} and  \eqref{Bag}, we have
\begin{align*}
J_{\beta}(u_{n},v_{n}) =J_{\beta}(\overline{u}_{n},\overline{v}_{n})+J_{\beta}(u,v)+o(1).
\end{align*}
Therefore
\begin{align*}
m_{\beta}(b_{1},b_{2})&=\lim_{n\rightarrow +\infty}J_{\beta}(u_{n},v_{n})=\lim_{n\rightarrow +\infty}J_{\beta}(\overline{u}_{n},\overline{v}_{n})+J_{\beta}(0,v)\\
&=\lim_{n\rightarrow +\infty}\bigg[\frac{1}{2}\int_{\mathbb{R}^{2}}(|\nabla \overline{u}_{n}|^{2}+|\nabla \overline{v}_{n}|^{2})dx-\frac{1}{4}\int_{\mathbb{R}^{2}}\left(\mu_{1}\overline{u}_{n}^{4}+\mu_{2}\overline{v}_{n}^{4}+2\rho \overline{u}_{n}^{2}\overline{v}_{n}^{2}\right)dx\\
&\quad-\frac{1}{2}\beta\int_{\mathbb{R}^{2}}\overline{u}_{n}^{2}\overline{v}_{n}dx\bigg]+J_{\beta}(0,v)\\
&\geq\lim_{n\rightarrow +\infty}\frac{1}{2}\int_{\mathbb{R}^{2}}(|\nabla \overline{u}_{n}|^{2}+|\nabla \overline{v}_{n}|^{2})dx+J_{\beta}(0,v)\geq m_{\beta}(0,b_{2})=0.
\end{align*}
contradict with Lemma \ref{LL}, where we use the fact that $m_{\beta}(0,b_{2})=0$ for $0<b_{2}\leq \| Q\|_{L^{2}(\mathbb{R}^{2})}$.

{\bf Case 4}  If $u\neq0,v\neq0$, let $$\widetilde{u}_{n}=u_{n}-u,\ \widetilde{v}_{n}=v_{n}-v,$$
from \cite[Lemma 2.1]{PH}, we get
\begin{equation}\label{BAg3}
\int_{\mathbb{R}^{2}}|\widetilde{u}_{n}|^{2}|\widetilde{v}_{n}|^{2}dx=\int_{\mathbb{R}^{2}}|u_{n}|^{2}|v_{n}|^{2}dx-\int_{\mathbb{R}^{2}}|u|^{2}|v|^{2}dx+o(1).
\end{equation}
By Brezis Lieb Lemma in \cite{WM}, we can obtain
\begin{equation}\label{BAg}
\int_{\mathbb{R}^{2}}|\widetilde{u}_{n}|^{2}\widetilde{v}_{n}dx=\int_{\mathbb{R}^{2}}|u_{n}|^{2}v_{n}dx-\int_{\mathbb{R}^{2}}|u|^{2}vdx+o(1).
\end{equation}
So, from \eqref{BAg3} and \eqref{BAg}, we have
\begin{align*}
J_{\beta}(u_{n},v_{n}) =J_{\beta}(\widetilde{u}_{n},\widetilde{v}_{n})+J_{\beta}(u,v)+o(1).
\end{align*}
Therefore
\begin{align*}
m_{\beta}(b_{1},b_{2})&=\lim_{n\rightarrow +\infty}J_{\beta}(u_{n},v_{n})=\lim_{n\rightarrow +\infty}J_{\beta}(\widetilde{u}_{n},\widetilde{v}_{n})+J_{\beta}(u,v)\\
&=\lim_{n\rightarrow +\infty}\bigg[\frac{1}{2}\int_{\mathbb{R}^{2}}(|\nabla \widetilde{u}_{n}|^{2}+|\nabla \widetilde{v}_{n}|^{2})dx\\
&\quad-\frac{1}{4}\int_{\mathbb{R}^{2}}\left(\mu_{1}\widetilde{u}_{n}^{4}+\mu_{2}\widetilde{v}_{n}^{4}+2\rho \widetilde{u}_{n}^{2}\widetilde{v}_{n}^{2}\right)dx
-\frac{1}{2}\beta\int_{\mathbb{R}^{2}}\widetilde{u}_{n}^{2}\widetilde{v}_{n}dx\bigg]+J_{\beta}(u,v)\\
&\geq\lim_{n\rightarrow +\infty}\frac{1}{2}\int_{\mathbb{R}^{2}}(|\nabla \widetilde{u}_{n}|^{2}+|\nabla \widetilde{v}_{n}|^{2})dx+J_{\beta}(u,v)\geq m_{\beta}(b_{1},b_{2}).
\end{align*}
Thus, $J_{\beta}(u,v)=m_{\beta}(b_{1},b_{2})\ \text{and}\ (u_{n},v_{n})\rightarrow (u,v)$ in $H^{1}(\mathbb{R}^{2})\times H^{1}(\mathbb{R}^{2}) .$

Next, we prove the second part of Theorem \ref{Th3}. It is easy to see that
\begin{align*}
J_{0}(t\star (u,v))&=\frac{e^{2t}}{2}\int_{\mathbb{R}^{2}}(|\nabla u|^{2}+|\nabla v|^{2})dx-\frac{e^{2t}}{4}\int_{\mathbb{R}^{2}}\left(\mu_{1}u^{4}+\mu_{2}v^{4}+2\rho u^{2}v^{2}\right)dx.
\end{align*}
If
\begin{align}\label{Lac}
\frac{1}{2}\int_{\mathbb{R}^{2}}(|\nabla u|^{2}+|\nabla v|^{2})dx<\frac{1}{4}\int_{\mathbb{R}^{2}}\left(\mu_{1}u^{4}+\mu_{2}v^{4}+2\rho u^{2}v^{2}\right)dx,
\end{align}
then
$$J_{0}(t\star (u,v))\rightarrow -\infty\  \text{as}\ t \rightarrow +\infty.$$
It is easy to see
 \begin{align}\label{Lad}
J_{\beta}(t\star (u,v))&=\frac{e^{2t}}{2}\int_{\mathbb{R}^{2}}(|\nabla u|^{2}+|\nabla v|^{2})dx-\frac{\beta}{2}e^{t}\int_{\mathbb{R}^{2}}u^{2}vdx-\frac{e^{2t}}{4}\int_{\mathbb{R}^{2}}\left(\mu_{1}u^{4}+\mu_{2}v^{4}+2\rho u^{2}v^{2}\right)dx\\\nonumber
&=e^{2t}J_{0}(u,v)-\frac{\beta}{2}e^{t}\int_{\mathbb{R}^{2}}u^{2}vdx.
\end{align}
If \eqref{Lac} holds, there exists $(u,v)\in \mathrm{T}_{b_{1}}\times \mathrm{T}_{b_{2}}$ such  that $J_{0}(u,v)<0,$ by \eqref{Lad} we can deduce  $$\inf_{(u,v)\in \mathrm{T}_{b_{1}}\times \mathrm{T}_{b_{2}}}J_{\beta}(u,v)=-\infty.$$
\end{proof}

\section{Proof of Theorem \ref{Th1}}\label{sec4}
Define the set
\begin{equation}\label{int4}
\begin{aligned}
\mathcal{P}_{b_{1},b_{2}}(u,v):=\Big\{&(u,v)\in\mathrm{T}_{b_{1}}\times \mathrm{T}_{b_{2}} :P_{b_{1},b_{2}}(u,v)=0,\ \beta\int_{\mathbb{R}^{3}}u^{2}vdx>0\Big\},
\end{aligned}
\end{equation}
where
\begin{align*}
P_{b_{1},b_{2}}(u,v)&=\int_{\mathbb{R}^{3}}(|\nabla u|^{2}+|\nabla v|^{2})dx-\frac{3}{4}\int_{\mathbb{R}^{3}}\left(\mu_{1}u^{4}+\mu_{2}v^{4}+2\rho u^{2}v^{2}\right)dx-\frac{3}{4}\beta\int_{\mathbb{R}^{3}}u^{2}vdx,
\end{align*}
and
\begin{align}\label{pp}
\Psi_{u,v}(t)=J_{\beta}(t\star (u,v))=&\frac{e^{2t}}{2}\int_{\mathbb{R}^{3}}(|\nabla u|^{2}+|\nabla v|^{2})dx-\frac{\beta}{2}\int_{\mathbb{R}^{3}}e^{\frac{3t}{2}}u^{2}vdx\\\nonumber
&-\frac{e^{3t}}{4}\int_{\mathbb{R}^{3}}\left(\mu_{1}u^{4}+\mu_{2}v^{4}+2\rho u^{2}v^{2}\right)dx.
\end{align}
It is easy to check
\begin{align*}
\Psi'_{u,v}(t)=e^{2t}\int_{\mathbb{R}^{3}}(|\nabla u|^{2}+|\nabla v|^{2})dx-\frac{3\beta}{4}\int_{\mathbb{R}^{3}}e^{\frac{3t}{2}}u^{2}vdx-\frac{3e^{3t}}{4}\int_{\mathbb{R}^{3}}\left(\mu_{1}u^{4}+\mu_{2}v^{4}+2\rho u^{2}v^{2}\right)dx,
\end{align*}
and $$\Psi'_{u,v}(t)=\mathcal{P}_{b_{1},b_{2}}(t\star u,t\star v),\ \ \mathcal{P}_{b_{1},b_{2}}(u, v)=\Big\{(u,v)\in \mathrm{T}_{b_{1}}\times \mathrm{T}_{b_{2}}:\Psi'_{u,v}(0)=0,\ \beta\int_{\mathbb{R}^{3}}u^{2}vdx>0 \Big\}.$$
We decompose $ \mathcal{P}_{b_{1},b_{2}}(u, v)$ into three disjoint unions $$\mathcal{P}_{b_{1},b_{2}}(u, v)=\mathcal{P}^{+}_{b_{1},b_{2}}(u, v)\cup \mathcal{P}^{0}_{b_{1},b_{2}}(u, v)\cup \mathcal{P}^{-}_{b_{1},b_{2}}(u, v),$$ where
\begin{align}\label{k1}
\mathcal{P}^{+}_{b_{1},b_{2}}(u, v):=\Big\{(u,v)\in \mathrm{T}_{b_{1}}\times \mathrm{T}_{b_{2}},\  (u,v)\in\mathcal{P}_{b_{1},b_{2}}, \  \beta\int_{\mathbb{R}^{3}}u^{2}vdx>0 :\Psi''_{u,v}(0)>0 \Big\},
\end{align}
\begin{align}\label{k2}
&\mathcal{P}^{0}_{b_{1},b_{2}}(u, v):=\Big\{(u,v)\in \mathrm{T}_{b_{1}}\times \mathrm{T}_{b_{2}},\ (u,v)\in\mathcal{P}_{b_{1},b_{2}},\  \beta\int_{\mathbb{R}^{3}}u^{2}vdx>0 :\Psi''_{u,v}(0)=0 \Big\},
\end{align}
\begin{align}\label{k3}
\mathcal{P}^{-}_{b_{1},b_{2}}(u, v):=\Big\{(u,v)\in \mathrm{T}_{b_{1}}\times \mathrm{T}_{b_{2}},\ (u,v)\in\mathcal{P}_{b_{1},b_{2}},\  \beta\int_{\mathbb{R}^{3}}u^{2}vdx>0 :\Psi''_{u,v}(0)<0 \Big\}.
\end{align}

To prove Theorem \ref{Th1}, we first give some lemmas. The following auxiliary result shows the role of $\mathcal{P}_{b_{1},b_{2}}(u,v)$.
\begin{lemma}\label{Lemp}
If $(u,v)$ is a solution of problem \eqref{23}--\eqref{24} for some $\lambda_{1},\lambda_{2}\in \mathbb{R}$, then $(u,v)\in \mathcal{P}_{b_{1},b_{2}}(u,v)$.
\end{lemma}
\begin{proof}
On the one hand, we multiply the first equation of \eqref{23} by $x\cdot\nabla u$ and the second equation of \eqref{23} by $x\cdot\nabla v$, integrate  by parts, we can get
\begin{align}\label{LA111}
&\frac{1}{2}\int_{\mathbb{R}^{3}}(|\nabla u|^{2}+|\nabla v|^{2})dx+\frac{ 3}{2}\int_{\mathbb{R}^{3}}(\lambda_{1} u^{2}+\lambda_{2} v^{2})dx\\\nonumber
&=\frac{3}{4}\int_{\mathbb{R}^{3}}\left(\mu_{1}u^{4}+\mu_{2}v^{4}+2\rho u^{2}v^{2}\right)dx+\frac{3}{2}\beta\int_{\mathbb{R}^{3}}u^{2}vdx.
\end{align}
On the other hand, we multiply the first equation of \eqref{23} by $ u$ and the second equation of \eqref{23} by $ v$, integrate  by parts, we have
\begin{align}\label{LAg}
\int_{\mathbb{R}^{3}}(|\nabla u|^{2}+|\nabla v|^{2})dx+\int_{\mathbb{R}^{3}}(\lambda_{1} u^{2}+\lambda_{2} v^{2})dx=\int_{\mathbb{R}^{3}}\left(\mu_{1}u^{4}+\mu_{2}v^{4}+2\rho u^{2}v^{2}\right)dx+\frac{3}{2}\beta\int_{\mathbb{R}^{3}}u^{2}vdx.
\end{align}
Together \eqref{LA111} with \eqref{LAg}, we obtain
\begin{align}\label{LA3}
\int_{\mathbb{R}^{3}}(|\nabla u|^{2}+|\nabla v|^{2})dx-\frac{3}{4}\int_{\mathbb{R}^{3}}\left(\mu_{1}u^{4}+\mu_{2}v^{4}+2\rho u^{2}v^{2}\right)dx-\frac{3}{4}\beta\int_{\mathbb{R}^{3}}u^{2}vdx=0.
\end{align}
\end{proof}
To show that the energy functional has a concave-convex structure, by \eqref{LA11}-\eqref{LA13}, we introduce $$\mathcal{D}_{1}=\mu_{1} C^{4}_{3,4}b_{1}, \  \mathcal{D}_{2}=\mu_{2}C^{4}_{4,4}b_{2},\ \mathcal{D}_{3}=C^{4}_{3,4}b^{\frac{1}{2}}_{1}b^{\frac{1}{2}}_{2},\ \ \mathcal{D}_{4}=\big(\frac{2}{3}b^{\frac{3}{2}}_{1}+\frac{1}{3}b^{\frac{3}{2}}_{2}\big)C^{3}_{3,3}.$$ If  $\rho>0$
\begin{align}\label{LA15}
J_{\beta}(u,v)&\geq\frac{1}{2}\int_{\mathbb{R}^{3}}(|\nabla u|^{2}+|\nabla v|^{2})dx-\frac{1}{4}\mathcal{D}_{1}\big(\int_{\mathbb{R}^{3}}|\nabla u|^{2}dx\big)^{\frac{3}{2}}-\frac{1}{4}\mathcal{D}_{2}\big(\int_{\mathbb{R}^{3}}|\nabla v|^{2}dx\big)^{\frac{3}{2}}\\\nonumber
&\quad-\frac{1}{4}\rho\mathcal{D}_{3}\big(\int_{\mathbb{R}^{3}}(|\nabla u|^{2}+|\nabla v|^{2})dx\big)^{\frac{3}{2}}-\frac{1}{2}|\beta|\mathcal{D}_{4}\big(\int_{\mathbb{R}^{3}}(|\nabla u|^{2}+|\nabla v|^{2})dx\big)^{\frac{3}{4}}\\\nonumber
&\geq h\Big(\big(\int_{\mathbb{R}^{3}}(|\nabla u|^{2}+|\nabla v|^{2})dx\big)^{\frac{1}{2}}\Big),
\end{align}
where $h(t):(0,+\infty)\rightarrow \mathbb{R}$ defined by $$h(t)=\frac{1}{2}t^{2}-\frac{1}{4}(\mathcal{D}_{1}+\mathcal{D}_{2}+\rho\mathcal{D}_{3})t^{3}-\frac{1}{2}|\beta|\mathcal{D}_{4}t^{\frac{3}{2}}.$$
\begin{lemma}\label{LAA}
When $|\beta|\left(2b^{\frac{3}{2}}_{1}+b^{\frac{3}{2}}_{2}\right)C^{3}_{3,3}C^{2}_{3,4}\sqrt{\mu_{1}b_{1}+\mu_{2}b_{2}+\rho b^{\frac{1}{2}}_{1}b^{\frac{1}{2}}_{2}}<\frac{2\sqrt{6}}{3}$ holds,  then $h(t)$ has exactly two critical points, one is a local minimum at negative level, the other one is a global maximum at positive level. Further, there exists $0<R_{0}<R_{1}$ such that $h(R_{0})=h(R_{1})=0$, $h(t)>0$ if and only if
$t\in (R_{0},R_{1})$.
\end{lemma}
\begin{proof}
Since
$$h(t)=t^{\frac{3}{2}}\Big[\frac{1}{2}t^{\frac{1}{2}}-\frac{1}{4}C^{4}_{3,4}(\mu_{1} b_{1}+\mu_{2} b_{2}+\rho b^{\frac{1}{2}}_{1}b^{\frac{1}{2}}_{2})t^{\frac{3}{2}}-\frac{1}{2}|\beta|\big(\frac{2}{3}b^{\frac{3}{2}}_{1}+\frac{1}{3}b^{\frac{3}{2}}_{2}\big)C^{3}_{3,3}\Big]$$
it is easy to see that $h(t)>0$ if and only if $\varphi(t)>\frac{1}{2}|\beta|\big(\frac{2}{3}b^{\frac{3}{2}}_{1}+\frac{1}{3}b^{\frac{3}{2}}_{2}\big)C^{3}_{3,3}$, where $\varphi(t)=\frac{1}{2}t^{\frac{1}{2}}-\frac{1}{4}C^{4}_{3,4}(\mu_{1} b_{1}+\mu_{2} b_{2}+\rho b^{\frac{1}{2}}_{1}b^{\frac{1}{2}}_{2})t^{\frac{3}{2}}$, since $$\varphi'(t)=\frac{1}{4}t^{-\frac{1}{2}}-\frac{3}{8}C^{4}_{3,4}(\mu_{1} b_{1}+\mu_{2} b_{2}+\rho b^{\frac{1}{2}}_{1}b^{\frac{1}{2}}_{2})t^{\frac{1}{2}},$$
it is easy to see that $\varphi(t)$ has a unique global maximum point at positive level in $\widetilde{t}=\frac{2}{3C^{4}_{3,4}(\mu_{1} b_{1}+\mu_{2} b_{1}+\rho b^{\frac{1}{2}}_{1}b^{\frac{1}{2}}_{2})}$ and $$\varphi(\widetilde{t})=\frac{\sqrt{6}}{9}\frac{1}{\sqrt{C^{4}_{3,4}(\mu_{1} b_{1}+\mu_{2} b_{2}+\rho b^{\frac{1}{2}}_{1}b^{\frac{1}{2}}_{2})}},$$
 therefore $h(t)$ is positive on an open interval $(R_{0},R_{1})$ if and only if $\varphi(\widetilde{t})>\frac{1}{2}|\beta|\left(\frac{2}{3}b^{\frac{3}{2}}_{1}+\frac{1}{3}b^{\frac{3}{2}}_{2}\right)C^{3}_{3,3}$, so
\begin{align}\label{LA4}
|\beta|\big(2b^{\frac{3}{2}}_{1}+b^{\frac{3}{2}}_{2}\big)C^{3}_{3,3}C^{2}_{3,4}\sqrt{\mu_{1}b_{1}+\mu_{2}b_{2}+\rho b^{\frac{1}{2}}_{1}b^{\frac{1}{2}}_{2}}<\frac{2\sqrt{6}}{3}.
\end{align}
Since
\begin{align*}
h'(t)&=t^{\frac{1}{2}}\Big[t^{\frac{1}{2}}-\frac{3}{4}C^{4}_{3,4}(\mu_{1} b_{1}+\mu_{2} b_{2}+\rho b^{\frac{1}{2}}_{1}b^{\frac{1}{2}}_{2})t^{\frac{3}{2}}-\frac{3}{4}|\beta|\big(\frac{2}{3}b^{\frac{3}{2}}_{1}+\frac{1}{3}b^{\frac{3}{2}}_{2}\big)C^{3}_{3,3}\Big]=t^{\frac{1}{2}}g(t),
\end{align*}
 we can deduce that when
\begin{equation}\label{LA5}
\begin{aligned}
g(\frac{4}{9C^{4}_{3,4}(\mu_{1} b_{1}+\mu_{2} b_{1}+\rho b^{\frac{1}{2}}_{1}b^{\frac{1}{2}}_{2})})
=\frac{4}{9\sqrt{C^{4}_{3,4}(\mu_{1} b_{1}+\mu_{2} b_{1}+\rho b^{\frac{1}{2}}_{1}b^{\frac{1}{2}}_{2})}}-\frac{3}{4}|\beta|\big(\frac{2}{3}b^{\frac{3}{2}}_{1}+\frac{1}{3}b^{\frac{3}{2}}_{2}\big)C^{3}_{3,3}>0,
\end{aligned}
\end{equation}
then $h(t)$ has exactly two critical points, one is a local minimum at negative level, the other one is a global maximum at positive level. It is easy to see that when \eqref{LA4} holds, then  \eqref{LA5} also holds.
\end{proof}
Let
\begin{equation*}
N:=\Big\{(u,v)\in H^1(\mathbb{R}^{4})\times H^1(\mathbb{R}^{4})~ |~ \beta\int_{\R^3}u^2vdx>0\Big\},
\end{equation*}
and for simplicity, we still denote
\begin{equation*}
\mathcal{P}_{b_1,b_2}=\Big\{(u,v)\in (\mathrm{T}_{b_{1}}\times \mathrm{T}_{b_{2}})\cap N ~|~ P_{b_1,b_2}(u,v)=0\Big\}.
\end{equation*}

\begin{lemma}\label{lem4.3}
Under the assumption \eqref{LA4}, then $\mathcal{P}^{0}_{b_{1},b_{2}}=\emptyset$ and $\mathcal{P}_{b_{1},b_{2}}$ is a $C^{1}$ submanifold in $H^{1}(\mathbb{R}^{3})\times H^{1}(\mathbb{R}^{3})$ with codimension 3.
\end{lemma}
\begin{proof}
Assume by contradiction that there exists a $(u,v)\in \mathcal{P}^{0}_{b_{1},b_{2}}$ such that
\begin{align}\label{LA6}
\int_{\mathbb{R}^{3}}(|\nabla u|^{2}+|\nabla v|^{2})dx
=\frac{3}{4}\int_{\mathbb{R}^{3}}(\mu_{1}u^{4}+\mu_{2}v^{4}+2\rho u^{2}v^{2})dx
+\frac{3}{4}\beta\int_{\mathbb{R}^{3}}u^{2}vdx
\end{align}
and
\begin{align}\label{LA7}
2\int_{\mathbb{R}^{3}}(|\nabla u|^{2}+|\nabla v|^{2})dx-\frac{9\beta}{8}\int_{\mathbb{R}^{3}}u^{2}vdx
-\frac{9}{4}\int_{\mathbb{R}^{3}}\left(\mu_{1}u^{4}+\mu_{2}v^{4}+2\rho u^{2}v^{2}\right)dx=0.
\end{align}
From  \eqref{LA6} and \eqref{LA7}, we obtain
\begin{align}\label{LA8}
\beta\int_{\mathbb{R}^{3}}u^{2}vdx=
2\int_{\mathbb{R}^{3}}\left(\mu_{1}u^{4}+\mu_{2}v^{4}+2\rho u^{2}v^{2}\right)dx.
\end{align}
Thus, from \eqref{LA11}-\eqref{LA13}, we have
\begin{align}\label{LA9}
\int_{\mathbb{R}^{3}}(|\nabla u|^{2}+|\nabla v|^{2})dx&=
\frac{9}{4}\int_{\mathbb{R}^{3}}\left(\mu_{1}u^{4}+\mu_{2}v^{4}+2\rho u^{2}v^{2}\right)dx\\\nonumber
&\leq\frac{9\mu_{1}}{4} C^{4}_{3,4}b_{1}\|\nabla u\|^{3}_{L^{2}(\mathbb{R}^{3})}+\frac{9\mu_{2}}{4} C^{4}_{3,4}b_{2}\|\nabla v\|^{3}_{L^{2}(\mathbb{R}^{3})}\\\nonumber
&\quad+\frac{9}{4}\rho C^{4}_{3,4}b^{\frac{1}{2}}_{1}b^{\frac{1}{2}}_{2}\left[\|\nabla u\|^{2}_{L^{2}(\mathbb{R}^{3})}+\|\nabla v\|^{2}_{L^{2}(\mathbb{R}^{3})}\right]^{\frac{3}{2}}\\\nonumber
&\leq \frac{9}{4}C^{4}_{3,4}(\mu_{1} b_{1}+\mu_{2} b_{2}+\rho b^{\frac{1}{2}}_{1}b^{\frac{1}{2}}_{2})\left[\|\nabla u\|^{2}_{L^{2}(\mathbb{R}^{3})}+\|\nabla v\|^{2}_{L^{2}(\mathbb{R}^{3})}\right]^{\frac{3}{2}},
\end{align}
\begin{align}\label{LA10}
\int_{\mathbb{R}^{3}}(|\nabla u|^{2}+|\nabla v|^{2})dx&=
\frac{9}{8}\beta\int_{\mathbb{R}^{3}}u^{2}vdx\\\nonumber
& \leq\frac{9}{8}|\beta|\left[\left(\frac{2}{3}b^{\frac{3}{2}}_{1}+\frac{1}{3}b^{\frac{3}{2}}_{2}\right)C^{3}_{3,3}\left[\|\nabla u\|^{2}_{L^{2}(\mathbb{R}^{3})}+\|\nabla v\|^{2}_{L^{2}(\mathbb{R}^{3})}\right]^{\frac{3}{4}}\right].
\end{align}
So, by \eqref{LA9} and \eqref{LA10}, we get $$\frac{1}{3\sqrt{C^{4}_{3,4}(\mu_{1} b_{1}+\mu_{2} b_{1}+\rho b^{\frac{1}{2}}_{1}b^{\frac{1}{2}}_{2})}}\leq\frac{9}{16}|\beta|\mathcal{D}_{4},$$
therefore
\begin{align}\label{LA14}
|\beta|\left(2b^{\frac{3}{2}}_{1}+b^{\frac{3}{2}}_{2}\right)C^{3}_{3,3}C^{2}_{3,4}\sqrt{\mu_{1}b_{1}+\mu_{2}b_{2}+\rho b^{\frac{1}{2}}_{1}b^{\frac{1}{2}}_{2}}\geq\frac{16}{9},
\end{align}
contract  the assumption \eqref{LA4}, which implies that $\mathcal{P}^{0}_{b_{1},b_{2}}=\emptyset$. Next, we show that $\mathcal{P}_{b_{1},b_{2}}$ is a smooth manifold of codimension 3 on $H^{1}(\mathbb{R}^{3})\times H^{1}(\mathbb{R}^{3})$. It is easy to see that  $\mathcal{P}_{b_{1},b_{2}}$ is defined by $P_{b_{1},b_{2}}(u,v)=0$, $G(u)=0$, $F(u)=0$ and $E(u,v)>0$, where
\begin{equation*}
G(u)=\int_{\mathbb{R}^{3}}u^{2}dx-b^{2}_{1},\ \ F(v)=\int_{\mathbb{R}^{3}}v^{2}dx-b^{2}_{2} \ \  \text{and} \ \ E(u,v)=\beta\int_{\mathbb{R}^{3}}u^2v.
\end{equation*}
Since $P_{b_{1},b_{2}}(u,v)$, $G(u)$, $F(u)$ and $E(u,v)$ are class of $C^{1}$, we only need to check that $$d(P_{b_{1},b_{2}}(u,v),G(u),F(v),E(u,v)): H\rightarrow \mathbb{R}^{3}\ \text{ is surjective},$$
for every $(u,v)\in (G^{-1}(0)\times F^{-1}(0))\cap E^{-1}(0)\cap P^{-1}_{b_1,b_2}(0)$.
If this is not true, $dP_{b_{1},b_{2}}(u,v)$ has to be linearly dependent from $dG(u)$, $dF(v)$ and $dE(u,v)$ i.e. there exist $\nu_{1},\nu_{2},\nu_{3}\in \mathbb{R}$  such that
\begin{equation*}
\begin{cases}

2\int_{\mathbb{R}^{3}}\nabla u\nabla \varphi +\nu_{1}u\varphi+\nu_3 \beta uv\varphi
=3\int_{\mathbb{R}^{3}}(\mu_{1}u^{3}\varphi+\rho v^{2}u\varphi) +\frac{3\beta}{2}\int_{\mathbb{R}^{3}} u\varphi v & \text{in} \ \ \mathbb{R}^{3},\\

2\int_{\mathbb{R}^{3}}\nabla v\nabla \psi +\nu_{2}v\psi+\frac{\nu_3}{2} \beta u^2\psi
=3\int_{\mathbb{R}^{3}}(\mu_{2}v^{3}\psi+\rho u^{2}v\psi) +\frac{3\beta}{4}\int_{\mathbb{R}^{3}} u^{2}\psi & \text{in} \ \ \mathbb{R}^{3},
\end{cases}
\end{equation*}
for every $(\varphi,\psi)\in H^{1}(\mathbb{R}^{3})\times H^{1}(\mathbb{R}^{3})$, so
\begin{equation*}
\begin{cases}

-2 [\Delta u  +\nu_{1}u+\nu_3\beta uv]=3(\mu_{1}u^{3}+\rho v^{2}u) +\frac{3\beta}{2} uv & \text{in} \ \ \mathbb{R}^{3},\\

-2 [\Delta v  +\nu_{2}v+\frac{\nu_3}{2}\beta u^2]=3(\mu_{2}v^{3}+\rho u^{2}v) +\frac{3\beta}{4}u^{2} & \text{in} \ \ \mathbb{R}^{3}.
\end{cases}
\end{equation*}
The  Pohozaev identity for above system is
\begin{align*}
\int_{\mathbb{R}^{3}}(|\nabla u|^{2}+|\nabla v|^{2})dx-\frac{9}{8}\int_{\mathbb{R}^{3}}\left(\mu_{1}u^{4}+\mu_{2}v^{4}+2\rho u^{2}v^{2}\right)dx-\frac{3}{2}\big(\frac{3}{8}-\frac{\nu_3}{2}\big)\beta\int_{\mathbb{R}^{3}}u^{2}vdx=0.
\end{align*}
Then $\int_{\mathbb{R}^{3}}(|\nabla u|^{2}+|\nabla v|^{2})dx=\frac{9}{8}\int_{\mathbb{R}^{3}}\left(\mu_{1}u^{4}+\mu_{2}v^{4}+2\rho u^{2}v^{2}\right)dx$.
Since $(u,v)\in (G^{-1}(0)\times F^{-1}(0))\cap E^{-1}(0)\cap P^{-1}_{b_1,b_2}(0)$, we have
$\int_{\mathbb{R}^{3}}(|\nabla u|^{2}+|\nabla v|^{2})dx=\frac{3}{4}\int_{\mathbb{R}^{3}}\left(\mu_{1}u^{4}+\mu_{2}v^{4}+2\rho u^{2}v^{2}\right)dx$, this is a contradiction. By Proposition A.1 of \cite{MS21}, we get that if $m_\beta:=\inf\limits_{(u,v)\in \mathcal{P}_{b_1,b_2}}J_{\beta}=J_{\beta}(\bar{u},\bar{v})$, then there exist $\lambda_i\in \R(i=4)$, such that
\begin{align*}
&-(1+\lambda_4)\Delta \bar{u}  +\lambda_1\nu_{1}\bar{u}+\lambda_3\beta \bar{u}\bar{v}=(1+\frac{3\lambda_4}{2})(\mu_{1}\bar{u}^{3}+\rho \bar{v}^{2}u) +(1+\frac{3\lambda_4}{4})\beta \bar{u}\bar{v},\\
&-(1+\lambda_4)\Delta \bar{v}  +\lambda_2 \bar{v}+\frac{\lambda_3}{2}\beta \bar{u}^2=(1+\frac{3\lambda_4}{2}\lambda_4)(\mu_{2}\bar{v}^{3}+\rho \bar{u}^{2}\bar{v}) +(1+\frac{3\lambda_4}{4})\frac{\beta}{2}\bar{u}^{2}.
\end{align*}
Since $E(\bar{u},\bar{v})>0$, from Theorem 1 of \cite{FHC76} or the proof of Lemma 2.11 in \cite{MS21}, we have $\lambda_3=0$. Therefore, we obtain that the restricted set $N$ does not change the structure of the manifold $\mathcal{P}_{b_1,b_2}$. Then, we have that $\mathcal{P}_{b_{1},b_{2}}$ is a smooth manifold of codimension 3 on $\mathrm{T}_{b_{1}}\times \mathrm{T}_{b_{2}}$.
\end{proof}

\begin{lemma}\label{Lem4}
When $\beta\int_{\mathbb{R}^{3}}u^{2}vdx>0$. For every $(u,v)\in \mathrm{T}_{b_{1}}\times \mathrm{T}_{b_{2}}$, the function $\Psi_{u,v}(t)$ has exactly two critical points $s_{u,v}<t_{u,v}\in \mathbb{R}$ and two zeros $c_{u,v}<d_{u,v}\in \mathbb{R}$ with $s_{u,v}<c_{u,v}<t_{u,v}<d_{u,v}$. Moreover,
\begin{enumerate}
\item[$(1)$] $s_{u,v}\star (u,v)\in \mathcal{P}^{+}_{b_{1},b_{2}}(u, v)$ and $t_{u,v}\star (u,v)\in \mathcal{P}^{-}_{b_{1},b_{2}}(u, v)$, and if $t\star (u,v)\in \mathcal{P}_{b_{1},b_{2}}(u, v)$, then either $t=s_{u,v}$ or $t=t_{u,v}.$

\item[$(2)$] $(\int_{\mathbb{R}^{3}}(|\nabla (t\star u)|^{2}+|\nabla (t\star v)|^{2})dx)^{\frac{1}{2}}\leq R_{0}$ for every $t\leq c_{u,v},$ and
\begin{align*}
&J_{\beta}(u,v)(s_{u,v}\star (u,v))\\
&=\min\bigg\{J_{\beta}(t\star (u,v):t\in \mathbb{R}\ \text{and}\ (\int_{\mathbb{R}^{3}}(|\nabla (t\star u)|^{2}+|\nabla (t\star v)|^{2})dx)^{\frac{1}{2}}< R_{0} \bigg\}<0,
\end{align*}
where $R_{0}$ is defined in Lemma \ref{LAA}.
\item[$(3)$] We have $$J_{\beta}(t_{u,v}\star (u,v))=\max\{J_{\beta}(t\star (u,v)):t\in \mathbb{R} \}>0$$ and $\Psi_{u,v}(t)$ is strictly decreasing and concave on $(t_{u,v},+\infty)$. In particular, if $t_{u,v}<0$, then $P_{b_{1},b_{2}}(u,v)<0.$

\item[$(4)$] The maps $(u,v)\in \mathrm{T}_{b_{1}}\times \mathrm{T}_{b_{2}}: s_{u,v} \in \mathbb{R} $ and $(u,v)\in \mathrm{T}_{b_{1}}\times \mathrm{T}_{b_{2}}: t_{u,v} \in \mathbb{R} $ are of class $C^{1}$.
\end{enumerate}
\end{lemma}
\begin{proof}
Let $(u,v)\in \mathrm{T}_{b_{1}}\times \mathrm{T}_{b_{2}}$, since $t\star (u,v)\in \mathcal{P}_{b_{1},b_{2}}(u, v) $  if and only if $(\Psi_{u,v})'(t)=0$. Thus, we first show that $\Psi_{u,v}(t)$ has at least two critical points. From \eqref{LA15}, we have
\begin{align*}
\Psi_{u,v}(t)&=J_{\beta}(t\star (u,v))\geq h\Big(\big[\int_{\mathbb{R}^{3}}(|\nabla (t\star u)|^{2}+|\nabla (t\star v)|^{2})dx\big]^{\frac{1}{2}}\Big)\\
&=h\Big(e^{t}\big[\int_{\mathbb{R}^{3}}(|\nabla  u|^{2}+|\nabla  v|^{2})dx\big])^{\frac{1}{2}}\Big).
\end{align*}
Thus, the $C^{2}$ function $\Psi_{u,v}(t)$ is positive on $$\Big(\ln\frac{R_{0}}{\big[\int_{\mathbb{R}^{3}}(|\nabla  u|^{2}+|\nabla  v|^{2})dx\big]^{\frac{1}{2}}}, \ln\frac{R_{1}}{\big[\int_{\mathbb{R}^{3}}(|\nabla  u|^{2}+|\nabla  v|^{2})dx\big]^{\frac{1}{2}}}\Big),$$ $\Psi_{u,v}(+\infty)=-\infty$  and $\Psi_{u,v}(-\infty)=0^{-}$.

Indeed, $(\Psi_{u,v})'(t)=0$ implies that  $g(t)=\frac{3\beta}{4}\int_{\mathbb{R}^{3}}u^{2}vdx$ where $$ g(t)=e^{\frac{t}{2}}\int_{\mathbb{R}^{3}}(|\nabla u|^{2}+|\nabla v|^{2})dx-\frac{3e^{\frac{3t}{2}}}{4}\int_{\mathbb{R}^{3}}\left(\mu_{1}u^{4}+\mu_{2}v^{4}+2\rho u^{2}v^{2}\right)dx.$$ It is easy to see that $g(t)$ has a unique maximum point, thus the above equation has at most two solutions.
If $\frac{3\beta}{4}\int_{\mathbb{R}^{3}}u^{2}vdx>0$. So, $\Psi_{u,v}(-\infty)=0^{-}$, $\Psi_{u,v}(+\infty)=-\infty$.
 It is easy to see that $\Psi_{u,v}(t)$ has a local minimum point $s_{u,v}$ at negative level in $(0,\ln\frac{R_{0}}{\left(\int_{\mathbb{R}^{3}}(|\nabla  u|^{2}+|\nabla  v|^{2})dx\right)^{\frac{1}{2}}})$ and has a global maximum point $t_{u,v}$  at positive level in $\Big(\ln\frac{R_{0}}{\big[\int_{\mathbb{R}^{3}}(|\nabla  u|^{2}+|\nabla  v|^{2})dx\big]^{\frac{1}{2}}}, \ln\frac{R_{1}}{\big[\int_{\mathbb{R}^{3}}(|\nabla  u|^{2}+|\nabla  v|^{2})dx\big]^{\frac{1}{2}}}\Big).$ Next, we show that $\Psi_{u,v}(t)$ has no other critical points.  From $(u,v)\in \mathrm{T}_{b_{1}}\times \mathrm{T}_{b_{2}},$ $t\in\mathbb{R}$ is a critical point of $\Psi_{u,v}(t)$ if and only if $t\star (u,v)\in \mathcal{P}_{b_{1},b_{2}} ,$ we have $s_{u,v}\star (u,v),\ \  t_{u,v}\star (u,v)\in \mathcal{P}_{b_{1}\times b_{2}} $ and $t\star (u,v) \in \mathcal{P}_{b_{1}\times b_{2}}$ if and only if $t=s_{u,v}$ or $t=t_{u,v}$. Since $s_{u,v}$ is a local minimum point of $\Psi_{u,v}(t)$, we know that $(\Psi_{u,v})''(s_{u,v})\geq0$, from $\mathcal{P}^{0}_{b_{1},b_{2}}=\emptyset$, we know that $(\Psi_{u,v})''(s_{u,v})\neq0$, thus $(\Psi_{u,v})''(s_{u,v})>0$, therefore $s_{u,v}\star (u,v)\in \mathcal{P}^{+}_{b_{1},b_{2}}.$ Similarly, we have $t_{u,v}\star (u,v)\in \mathcal{P}^{-}_{b_{1},b_{2}}.$ By the monotonicity and the behavior at infinity of $\Psi_{u,v}(t)$, we know that $\Psi_{u,v}(t)$ has exactly two zeros $c_{u,v}<d_{u,v}$ with $s_{u,v}<c_{u,v}<t_{u,v}<d_{u,v}$ and $\Psi_{u,v}(t)$ has exactly two inflection points, in particular, $\Psi_{u,v}(t)$ is concave on $[t_{u,v},+\infty)$ and hence if $t_{u,v}<0$, then $P_{b_{1},b_{2}}(u,v)=(\Psi_{u,v})'(0)<0.$ Finally, we prove that $(u,v)\in \mathrm{T}_{b_{1}}\times \mathrm{T}_{b_{2}}\mapsto  s_{u,v} \in \mathbb{R} $ and $(u,v)\in \mathrm{T}_{b_{1}}\times \mathrm{T}_{b_{2}}\mapsto  t_{u,v} \in \mathbb{R} $ are of class $C^{1}$. Indeed, we can apply the implicit function theorem on the $C^{1}$ function $\Phi(t,u,v)=(\Psi_{u,v})'(t)$, then $\Phi(s_{u,v},u,v)=(\Psi_{u,v})'(s_{u,v})=0,  \partial_{s}\Phi(s_{u,v},u,v)=(\Psi_{u,v})''(s_{u,v})<0$, thus $(u,v)\in \mathrm{T}_{b_{1}}\times \mathrm{T}_{b_{2}}\mapsto  s_{u,v} \in \mathbb{R} $  is class of $C^{1}$. Similarly, we can prove that $(u,v)\in \mathrm{T}_{b_{1}}\times \mathrm{T}_{b_{2}}\mapsto t_{u,v} \in \mathbb{R} $  is class of $C^{1}$.
 \end{proof}

\subsection{Existence of a local minimizer}

For $k>0$,  set
$$A_{k}=\Big\{(u,v)\in \mathrm{T}_{b_{1}}\times \mathrm{T}_{b_{2}}:\big(\int_{\mathbb{R}^{3}}(|\nabla  u|^{2}+|\nabla  v|^{2})dx\big)^{\frac{1}{2}}<k\Big\},$$ and
$$m^{+}_{\beta}(b_{1},b_{2})=\inf_{(u,v)\in A_{R_{0}}}J_{\beta}(u,v).$$
From Lemma \ref{Lem4}, we have following corollary
\begin{corollary}\label{Lem6}
The set $\mathcal{P}^{+}_{b_{1},b_{2}}$ is contained in $$A_{R_{0}}=\Big\{(u,v)\in \mathrm{T}_{b_{1}}\times \mathrm{T}_{b_{2}}: \big(\int_{\mathbb{R}^{3}}(|\nabla  u|^{2}+|\nabla  v|^{2})dx\big)^{\frac{1}{2}}< R_{0}\Big\}$$ and $$\sup_{\mathcal{P}^{+}_{b_{1},b_{2}}}J_{\beta}(u,v)\leq0\leq \inf_{\mathcal{P}^{-}_{b_{1},b_{2}}}J_{\beta}(u,v).$$
\end{corollary}
\begin{lemma}\label{Lem7}
We have $m^{+}_{\beta}(b_{1},b_{2})\in (-\infty,0)$ that $$m^{+}_{\beta}(b_{1},b_{2})=\inf_{\mathcal{P}_{b_{1},b_{2}}}J_{\beta}(u,v)=\inf_{\mathcal{P}^{+}_{b_{1},b_{2}}}J_{\beta}(u,v)\ \text{and}\ m^{+}_{\beta}(b_{1},b_{2})< \inf_{\overline{A_{R_{0}}}\setminus A_{R_{0}-\rho}}J_{\beta}(u,v).$$
\end{lemma}
\begin{proof}
For $(u,v)\in A_{R_{0}}$, we have
$$J_{\beta}(u,v)\geq h\Big(\big(\int_{\mathbb{R}^{3}}(|\nabla  u|^{2}+|\nabla  v|^{2})dx\big)^{\frac{1}{2}}\Big)\geq \min_{t\in [0,R_{0}]}h(t)>-\infty.$$
Therefore, $m^{+}_{\beta}(b_{1},b_{2})>-\infty.$ Moreover, there exists $(u,v)\in \mathrm{T}_{b_{1}}\times \mathrm{T}_{b_{2}}$, such that $$\Big(\int_{\mathbb{R}^{3}}(|\nabla  (t\star u)|^{2}+|\nabla  (t\star v)|^{2})dx\Big)^{\frac{1}{2}}< R_{0}$$ and $J_{\beta}(t\star (u,v))<0$ for $t\ll -1$. Hence $m^{+}_{\beta}(b_{1},b_{2})<0.$ Since $\mathcal{P}^{+}_{b_{1},b_{2}}\subset A_{R_{0}}$, we know that $m^{+}_{\beta}(b_{1},b_{2})\leq \inf_{\mathcal{P}^{+}_{b_{1},b_{2}}}J_{\beta}(u,v) .$ On the other hand, if $(u,v)\in A_{R_{0}} $, then $s_{u,v}\star (u,v) \in \mathcal{P}^{+}_{b_{1},b_{2}}\subset A_{R_{0}}$ and
\begin{align*}
J_{\beta}(u,v)(s_{u,v}\star (u,v))=\min\Big\{J_{\beta}&(t\star (u,v)):t\in \mathbb{R}\ \text{and}\ \big(\int_{\mathbb{R}^{3}}(|\nabla  (t\star u)|^{2}+|\nabla  (t\star v)|^{2})dx\big)^{\frac{1}{2}}< R_{0} \Big\}\\
\leq J_{\beta}(u,v),
\end{align*}
where $R_{0}$ is defined in Lemma \ref{LAA}, so $\inf_{\mathcal{P}^{+}_{b_{1},b_{2}}}J_{\beta}(u,v)\leq m^{+}_{\beta}(b_{1},b_{2})$. Since $J_{\beta}(u,v)>0$ on $\mathcal{P}^{-}_{b_{1},b_{2}}$, we know that $\inf_{\mathcal{P}^{+}_{b_{1},b_{2}}}J_{\beta}(u,v)=\inf_{\mathcal{P}_{b_{1},b_{2}}}J_{\beta}(u,v).$ Finally, by the continuity of $h$, there exists $\rho>0$ such that $h(t)\geq\frac{m^{+}_{\beta}(b_{1},b_{2})}{2}$ if $t\in [R_{0}-\rho,R_{0}]$.
Therefore,
$$J_{\beta}(u,v)\geq h\Big(\big(\int_{\mathbb{R}^{3}}(|\nabla  u|^{2}+|\nabla  v|^{2})dx\big)^{\frac{1}{2}}\Big)\geq\frac{m^{+}_{\beta}(b_{1},b_{2})}{2}>m^{+}_{\beta}(b_{1},b_{2})$$ for every $(u,v)\in \mathrm{T}_{b_{1}}\times \mathrm{T}_{b_{2}}$ with $R_{0}-\rho\leq\left(\int_{\mathbb{R}^{3}}(|\nabla  u|^{2}+|\nabla  v|^{2})dx\right)^{\frac{1}{2}}\leq R_{0}$. This completes the proof.
\end{proof}
\begin{lemma}\label{LeM}
Under the assumption \eqref{LA4} holds,  we have $$m^{+}(b_{1},b_{2})<\min\{m^{+}(b_{1},0),m^{+}(0,b_{2})\}.$$
\end{lemma}
\begin{proof}
From \cite{BJJN}, we know that $m^{+}(0,b_{2})$ can be achieved by $v^{\ast}\in \mathrm{T}_{b_{2}}$ and $v^{\ast}$ is radially symmetric and decreasing. We choose aproper test function $u\in\mathrm{T}_{b_{1}}$ such that $(t\star u,v^{\ast})\in \mathrm{T}_{b_{1}}\times \mathrm{T}_{b_{2}}$. From Lemma \ref{Lem4}, we obtain
\begin{align}\label{p4}
h(t)<h_{1}(t)=\frac{1}{2}t^{2}-\frac{1}{4}\mu_{2} C^{4}_{3,4}b_{1}t^{3}-\frac{1}{2}|\beta|\frac{2}{3}b^{\frac{3}{2}}_{2}C^{3}_{3,3}t^{\frac{3}{2}}.
\end{align}
By direct calculations, there exists $0<t^{\ast}<R_{0}$ such that $<h_{1}(t^{\ast})=0$. From \cite{Soave}, we have $$m^{+}(0,b_{2})=\inf_{v\in \mathcal{P}^{+}_{0,b_{2}}}J_{\beta}(0,v)=\inf_{v\in \mathrm{T}_{b_{2}}\cap B(t^{\ast})}J_{\beta}(0,v).$$ Therefore, from the analysis in Lemma \ref{LAA}, we have $$\|\nabla v^{\ast}\|_{L^{2}}\leq t^{\ast}<R_{0}<\widetilde{t}=\frac{2}{3C^{4}_{3,4}(\mu_{1} b_{1}+\mu_{2} b_{1}+\rho b^{\frac{1}{2}}_{1}b^{\frac{1}{2}}_{2})}. $$ Since $h(R_{0})=h(R_{1})=0$ and the monotonicity of $h(t)$, we deduce that $(t\star u,v^{\ast})\in \mathrm{T}_{b_{1}}\times \mathrm{T}_{b_{2}}\cap A_{R_{0}}$ for $t\ll -1$, therefore,
\begin{align*}
m^{+}(b_{1},b_{2})&=\inf_{(u,v)\in \mathrm{T}_{b_{1}}\times \mathrm{T}_{b_{2}}\cap A_{R_{0}}}J(u,v)\leq J(t\star u,v^{\ast})\\
&=\frac{1}{2}\int_{\mathbb{R}^{3}}|\nabla v^{\ast}|^{2}dx-\frac{1}{4}\mu_{2}\int_{\mathbb{R}^{N}}|v^{\ast}|^{4}dx-\frac{\beta}{2}\int_{\mathbb{R}^{3}}|t\star u|^{2}(v^{\ast})dx\\
&\quad +\frac{e^{2t}}{2}\int_{\mathbb{R}^{3}}|\nabla u|^{2}dx-\frac{e^{3t}}{4}\int_{\mathbb{R}^{N}}\mu_{2}u^{4}dx-\frac{1}{2}\int_{\mathbb{R}^{N}}\rho |t\star u|^{2}(v^{\ast})^{2}dx\\
&< J(0,v^{\ast})=m^{+}(0,b_{2}).
\end{align*}
Similarly, we have $$m^{+}(b_{1},b_{2})<m^{+}(b_{1},0).$$ Hence, the proof is completed.
\end{proof}
\begin{lemma}\label{Lem5}
 Let $\{(u_{n},v_{n})\}\subset H^{1}(\mathbb{R}^{3})\times H^{1}(\mathbb{R}^{3})$ be a minimizing sequence for $J_{\beta}(u,v)|_{\mathrm{T}_{b_{1}}\times \mathrm{T}_{b_{2}}}$ at level $m^{+}_{\beta}(b_{1},b_{2})$. Then $\{(u_{n},v_{n})\}$ is bounded in $H^{1}(\mathbb{R}^{3})\times H^{1}(\mathbb{R}^{3})$.
\end{lemma}
\begin{proof}
Since $P_{b_{1},b_{2}}(u_{n},v_{n})\rightarrow 0 $, we have
\begin{align*}
P_{b_{1},b_{2}}(u_{n},v_{n})=\int_{\mathbb{R}^{3}}(|\nabla u_{n}|^{2}+|\nabla v_{n}|^{2})dx
-\frac{3}{4}\int_{\mathbb{R}^{3}}\left(\mu_{1}u_{n}^{4}+\mu_{2}v_{n}^{4}+2\rho u_{n}^{2}v_{n}^{2}\right)dx-\frac{3}{4}\beta\int_{\mathbb{R}^{3}}u_{n}^{2}v_{n}dx=o_{n}(1).
\end{align*}
Thus, from  \eqref{LA13}, we have
\begin{align*}
J_{\beta}(u_{n},v_{n})&=\frac{1}{6}\int_{\mathbb{R}^{3}}(|\nabla u_{n}|^{2}+|\nabla v_{n}|^{2})dx-\frac{\beta}{4}\int_{\mathbb{R}^{3}}u_{n}^{2}v_{n}dx+o_{n}(1)\\
&\geq\frac{1}{6}\int_{\mathbb{R}^{3}}(|\nabla u_{n}|^{2}+|\nabla v_{n}|^{2})dx\\
&\quad-\frac{|\beta|}{4}\Big[\big(\frac{2}{3}b^{\frac{3}{2}}_{1}+\frac{1}{3}b^{\frac{3}{2}}_{2}\big)C^{3}_{3,3}\big[\|\nabla u_{n}\|^{2}_{L^{2}(\mathbb{R}^{3})}+\|\nabla v_{n}\|^{2}_{L^{2}(\mathbb{R}^{3})}\big]^{\frac{3}{4}}\Big].
\end{align*}
Since $\{(u_{n},v_{n})\}$ is  a minimizer sequence for $J_{\beta}(u,v)|_{\mathrm{T}_{b_{1}}\times \mathrm{T}_{b_{2}}}$ at level $m^{+}_{\beta}(b_{1},b_{2})$, we have $J_{\beta}(u_{n},v_{n})\leq m^{+}_{\beta}(b_{1},b_{2})+1$ for $n$ large. Hence
\begin{align*}
&\frac{1}{6}\int_{\mathbb{R}^{3}}(|\nabla u_{n}|^{2}+|\nabla v_{n}|^{2})dx\\
&\leq \frac{|\beta|}{4}\Big[\big(\frac{2}{3}b^{\frac{3}{2}}_{1}+\frac{1}{3}b^{\frac{3}{2}}_{2}\big)C^{3}_{3,3}\big[\|\nabla u_{n}\|^{2}_{L^{2}(\mathbb{R}^{3})}+\|\nabla v_{n}\|^{2}_{L^{2}(\mathbb{R}^{3})}\big]^{\frac{3}{4}}\Big]+m^{+}_{\beta}(b_{1},b_{2})+2 ,
 \end{align*}
 so $\{(u_{n},v_{n})\}$ is bounded in $H^{1}(\mathbb{R}^{3})\times H^{1}(\mathbb{R}^{3})$.
This completes the proof.
\end{proof}

\begin{lemma}\label{Lem55}
 Let $\{(u_{n},v_{n})\}\subset  \mathrm{T}_{b_{1,r}}\times \mathrm{T}_{b_{2,r}}$ be a nonnegative minimizing sequence for $J_{\beta}(u,v)|_{\mathrm{T}_{b_{1}}\times \mathrm{T}_{b_{2}}}$ at level $m^{+}_{\beta}(b_{1},b_{2})$ with additional properties $P_{b_{1},b_{2}}(u_{n},v_{n}) \rightarrow 0$  and $u^{-}_{n},v^{-}_{n}\rightarrow 0$ a.e. in $\mathbb{R}^{3}$, then up to a subsequence $(u_{n},v_{n})\rightarrow(u,v)$  in $H^{1}(\mathbb{R}^{3})\times H^{1}(\mathbb{R}^{3})$, where $(u,v)$ is a positive solution of \eqref{23} for some $\lambda_{1},\lambda_{2}>0.$
\end{lemma}
\begin{proof}
 From Lemma \ref{Lem5}, we known that $\{(u_{n},v_{n})\}$ is bounded in $H_{r}^{1}(\mathbb{R}^{3})\times H_{r}^{1}(\mathbb{R}^{3})$, thus, by the Sobolev embedding theorem, we have $H_{r}^{1}(\R^3)\hookrightarrow\hookrightarrow L^{p}_r(\mathbb{R}^{3})$ for $2<p<6$, thus there exists a $(u,v)\in H_{r}^{1}(\mathbb{R}^{3})\times H_{r}^{1}(\mathbb{R}^{3}) $ such that $(u_{n},v_{n})\rightharpoonup (u,v)$ in $H_{r}^{1}(\mathbb{R}^{3})\times H_{r}^{1}(\mathbb{R}^{3})$, $(u_{n},v_{n})\rightarrow (u,v)$ in $L^{p}_r(\mathbb{R}^{3})\times L^{p}_r(\mathbb{R}^{3})\ \ \text{for }\  2<p<6$ and $(u_{n},v_{n})\rightarrow (u,v)$ a.e in $\mathbb{R}^{3}$. Hence $u,v\geq0$ are radial functions. Since $J'\mid_{\mathrm{T}_{b_{1}}\times \mathrm{T}_{b_{2}}}(u_{n},v_{n})\rightarrow 0$, by the Lagrange multipliers rule, we know that there exists a sequence
$(\lambda_{1,n},\lambda_{2,n})\in \mathbb{R}^{2}$ such that
\begin{equation}\label{LAg1}
\begin{aligned}
&\int_{\mathbb{R}^{3}}(\nabla u_{n}\nabla \varphi+\nabla v_{n}\nabla \psi) dx +\int_{\mathbb{R}^{3}}(\lambda _{1,n}u_{n}\varphi+\lambda_{2,n}v_{n}\psi)dx\\
&-\int_{\mathbb{R}^{3}}(\mu_{1}u_{n}^{3}\varphi+\mu_{2}v_{n}^{3}\psi+\rho v_{n}^{2}u_{n}\varphi+\rho u_{n}^{2}v_{n}\psi)dx\\
 &-\beta\int_{\mathbb{R}^{3}} u_{n}\varphi v_{n}-\frac{\beta}{2}\int_{\mathbb{R}^{3}} u_{n}^{2}\psi=o(1)\|(\psi,\varphi)\|_{H^{1}(\mathbb{R}^{3})\times H^{1}(\mathbb{R}^{3})} \ \  \text{in} \ \ \mathbb{R}^{3},
\end{aligned}
\end{equation}
for every $(\varphi,\psi)\in H^{1}(\mathbb{R}^{3})\times H^{1}(\mathbb{R}^{3})$. We claim both $\lambda_{1,n}$ and $\lambda_{2,n}$ are bounded sequence, and at least one of them is converging, up to a subsequence, to a strictly negative value. Indeed, we can using $(u_{n},0) $ and $(0,v_{n}) $ as text function in \eqref{LAg1}, we have $$\int_{\mathbb{R}^{3}}\lambda _{1,n}u^{2}_{n}dx=-\int_{\mathbb{R}^{3}}|\nabla u_{n}|^{2} dx +
\int_{\mathbb{R}^{3}}(\mu_{1}u_{n}^{4}+\rho v_{n}^{2}u^{2}_{n})dx +\beta\int_{\mathbb{R}^{3}} u^{2}_{n} v_{n}+o(1)\|\varphi\|_{H^{1}(\mathbb{R}^{3})} ,$$
$$ \int_{\mathbb{R}^{3}}\lambda _{2,n}v^{2}_{n}dx=-\int_{\mathbb{R}^{3}}|\nabla v_{n}|^{2} dx
+\int_{\mathbb{R}^{3}}(\mu_{2}v_{n}^{4}+\rho v_{n}^{2}u^{2}_{n})dx +\frac{\beta}{2}\int_{\mathbb{R}^{3}} u^{2}_{n} v_{n}+o(1)\|\psi\|_{H^{1}(\mathbb{R}^{3})} ,$$
so
\begin{equation}\label{LA199}
\begin{aligned}
\int_{\mathbb{R}^{3}}(\lambda _{1,n}u^{2}_{n}+\lambda _{2,n}v^{2}_{n})dx &=-\int_{\mathbb{R}^{3}}(|\nabla u_{n}|^{2} +|\nabla v_{n}|^{2})dx+\int_{\mathbb{R}^{3}}(\mu_{1}u_{n}^{4}+\mu_{2}v_{n}^{4}+2\rho v_{n}^{2}u^{2}_{n})dx\\
&\quad+\frac{3}{2}\beta\int_{\mathbb{R}^{3}}u_{n}^{2}v_{n}dx.
\end{aligned}
\end{equation}
By   \eqref{LA11}-\eqref{LA13} and  the boundedness of $\{(u_{n},v_{n})\}$, we can deduce that $(\lambda_{1,n},\lambda_{2,n})$ is bounded, hence up to a subsequence $(\lambda_{1,n},\lambda_{2,n})\rightarrow(\lambda_{1},\lambda_{2})\in \mathbb{R}^{2}$, passing to limits in \eqref{LAg1}, we can deduce that $(u,v)$ is a nonnegative solutions of  \eqref{23} .  Therefore
 \begin{align*}
\int_{\mathbb{R}^{3}}(|\nabla u|^{2}+|\nabla v|^{2})dx+\int_{\mathbb{R}^{3}}(\lambda_{1} u^{2}+\lambda_{2} v^{2})dx
=\int_{\mathbb{R}^{3}}\left(\mu_{1}u^{4}+\mu_{2}v^{4}+2\rho u^{2}v^{2}\right)dx+\frac{3}{2}\beta\int_{\mathbb{R}^{3}}u^{2}vdx.
\end{align*}
 From \eqref{LA11},  we can get
\begin{align*}
\frac{1}{2}\int_{\mathbb{R}^{3}}(|\nabla u|^{2}+|\nabla v|^{2})dx+\frac{ 3}{2}\int_{\mathbb{R}^{3}}(\lambda_{1} u^{2}+\lambda_{2} v^{2})dx
=\frac{3}{4}\int_{\mathbb{R}^{3}}\left(\mu_{1}u^{4}+\mu_{2}v^{4}+2\rho u^{2}v^{2}\right)dx+\frac{3}{2}\beta\int_{\mathbb{R}^{3}}u^{2}vdx.
\end{align*}
Thus we obtain
\begin{align}\label{LA188}
\int_{\mathbb{R}^{3}}(\lambda_{1} u^{2}+\lambda_{2} v^{2})dx
=\frac{1}{4}\int_{\mathbb{R}^{3}}\left(\mu_{1}u^{4}+\mu_{2}v^{4}+2\rho u^{2}v^{2}\right)dx+\frac{3}{4}\beta\int_{\mathbb{R}^{3}}u^{2}vdx.
\end{align}
From $P_{b_{1},b_{2}}(u_{n},v_{n})\rightarrow 0 $, we have
\begin{equation}\label{LA200}
\begin{aligned}
\int_{\mathbb{R}^{3}}(|\nabla u_{n}|^{2}+|\nabla v_{n}|^{2})dx
-\frac{3}{4}\int_{\mathbb{R}^{3}}\left(\mu_{1}u_{n}^{4}+\mu_{2}v_{n}^{4}+2\rho u_{n}^{2}v_{n}^{2}\right)dx
-\frac{3}{4}\beta\int_{\mathbb{R}^{3}}u_{n}^{2}v_{n}dx=o_{n}(1).
\end{aligned}
\end{equation}
Together \eqref{LA199} with \eqref{LA200}, we can get
\begin{align}\label{LAg2}
\lambda_{1,n} b^{2}_{1}+\lambda_{2,n} b^{2}_{2}
=\frac{1}{4}\int_{\mathbb{R}^{3}}\left(\mu_{1}u_{n}^{4}+\mu_{2}v_{n}^{4}+2\rho u_{n}^{2}v_{n}^{2}\right)dx+\frac{3}{4}\beta\int_{\mathbb{R}^{3}}u_{n}^{2}v_{n}dx.
\end{align}
When $\beta>0$, it is easy to see that  at least one sequence of $(\lambda_{i,n})$ is positive and bounded away from 0.
Let $n\rightarrow+\infty$ in \eqref{LAg2}, we have
\begin{align}\label{LA222}
\lambda_{1} b^{2}_{1}+\lambda_{2} b^{2}_{2}
=\frac{1}{4}\int_{\mathbb{R}^{3}}\left(\mu_{1}u^{4}+\mu_{2}v^{4}+2\rho u^{2}v^{2}\right)dx+\frac{3}{4}\beta\int_{\mathbb{R}^{3}}u^{2}vdx.
\end{align}
We claim that if $\lambda_{1}>0$(resp.$\lambda_{2}>0)$, then $\lambda_{2}>0$(resp.$\lambda_{1}>0$). Indeed, we know that at least one sequence of $(\lambda_{i})$ is positive and bounded away from 0. If $\lambda_{2}>0$, now we argue by contradiction and assume that $\lambda_{1}\leq0$, then $$-\Delta u=-\lambda_{1}u+\beta uv+\mu_{1}u^{3}+\rho v^{2}u\geq0.$$
Using a Liouville type theorem[\cite{NI14}, Lemma A.2], we can deduce that $u=0$. So, $v$ satisfies that
\begin{equation*}
\begin{cases}

-\Delta v+\lambda_{2}v= \mu_{2}v^{3},& \text{in} \ \ \mathbb{R}^{3},\\
v>0,&\text{in} \ \ \mathbb{R}^{3},\\
\int_{\mathbb{R}^{3}}v^{2}dx=b^{2}_{2},& \text{in} \ \ \mathbb{R}^{3}.\\
\end{cases}
\end{equation*}
Therefore
\begin{align*}
m^{+}_{\beta}(b_{1},b_{2})=\lim_{n\rightarrow+\infty}J_{\beta}(u_{n},v_{n})&=\lim_{n\rightarrow+\infty}\left[\frac{1}{8}\int_{\mathbb{R}^{3}}(\mu_{1}u_{n}^{4}+\mu_{2}v_{n}^{4}+2\rho u_{n}^{2}v_{n}^{2})dx-\frac{\beta}{8}\int_{\mathbb{R}^{3}}u_{n}^{2}v_{n}dx\right]\\
&=\frac{\mu_{1}}{8}\int_{\mathbb{R}^{3}}v^{4}dx=m^{+}_{\beta}(0,b_{2}),
\end{align*}
in contradiction with Lemma \ref{LeM}. Thus,   $\lambda_{1}>0,\ \lambda_{2}>0.$

If $\lambda_{1}>0$, now we argue by contradiction and assume that $\lambda_{2}\leq0$, then $$-\Delta v=-\lambda_{2}v \frac{\beta}{2}u^{2}+\mu_{2}v^{3}+\rho u^{2}v\geq0.$$
Using a Liouville type theorem[\cite{NI14}, Lemma A.2], we can deduce that $v=0$. So, by the structure of system \eqref{23}, we get $u=0$,  which is impossible.

It is easy to see that
\begin{align*}
m^{+}_{\beta}(b_{1},b_{2})=\lim_{n\rightarrow +\infty}J_{\beta}(u_{n},v_{n})=\lim_{n\rightarrow +\infty}\left[\frac{1}{6}\int_{\mathbb{R}^{3}}(|\nabla u_{n}|^{2}+|\nabla v_{n}|^{2})dx-\frac{\beta}{4}\int_{\mathbb{R}^{3}}u_{n}^{2}v_{n}dx\right].
\end{align*}
{\bf Case 1} If $u=0,v=0$, by compact Sobolev embedding,
\begin{align*}
m^{+}_{\beta}(b_{1},b_{2})=\lim_{n\rightarrow +\infty}J_{\beta}(u_{n},v_{n})\geq\lim_{n\rightarrow +\infty}\frac{1}{6}\int_{\mathbb{R}^{3}}(|\nabla u_{n}|^{2}+|\nabla v_{n}|^{2})dx\geq0.
\end{align*}
However, from Lemma \ref{Lem4}, we know that  $m^{+}_{\beta}(b_{1},b_{2})<0$, which is  a contraction.

{\bf Case 2} If $u\neq0,v=0$, indeed, if $v=0$, by the structure of system \eqref{23}, we get $u=0$, so Case 2 doesn't happen.

{\bf Case 3} If $u=0,\ v\neq0$, we have
\begin{align*}
m^{+}_{\beta}(b_{1},b_{2})=\lim_{n\rightarrow +\infty}J_{\beta}(u_{n},v_{n})\geq\lim_{n\rightarrow +\infty}\frac{1}{6}\int_{\mathbb{R}^{3}}(|\nabla u_{n}|^{2}+|\nabla v_{n}|^{2})dx\geq0,
\end{align*}
contradicts with $m^{+}_{\beta}(b_{1},b_{2})<0$.

{\bf Case 4}  if $u\neq0,v\neq0$, let $$\widetilde{u}_{n}=u_{n}-u,\ \widetilde{v}_{n}=v_{n}-v,$$
from  \cite[Lemma 2.4]{GL18}, we get
\begin{equation}\label{bag3}
\int_{\mathbb{R}^{3}}|\widetilde{u}_{n}|^{2}|\widetilde{v}_{n}|^{2}dx=\int_{\mathbb{R}^{3}}|u_{n}|^{2}|v_{n}|^{2}dx-\int_{\mathbb{R}^{3}}|u|^{2}|v|^{2}dx+o(1).
\end{equation}
By Brezis Lieb Lemma in \cite{WM}, we can obtain that
\begin{equation}\label{bag}
\int_{\mathbb{R}^{3}}|\widetilde{u}_{n}|^{2}\widetilde{v}_{n}dx=\int_{\mathbb{R}^{3}}|u_{n}|^{2}v_{n}dx-\int_{\mathbb{R}^{3}}|u|^{2}vdx+o(1).
\end{equation}
So, from \eqref{bag3} and \eqref{bag}, we have
\begin{align*}
0&=P_{b_{1},b_{2}}(u_{n},v_{n})+o(1) =P_{b_{1},b_{2}}(\widetilde{u}_{n},\widetilde{v}_{n})+P_{b_{1},b_{2}}(u,v)+o(1)\\
&=\int_{\mathbb{R}^{3}}(|\nabla \widetilde{u}_{n}|^{2}+|\nabla \widetilde{v}_{n}|^{2})dx
-\frac{3}{4}\int_{\mathbb{R}^{3}}\left(\mu_{1}\widetilde{u}_{n}^{4}+\mu_{2}\widetilde{v}_{n}^{4}+2\rho \widetilde{u}_{n}^{2}\widetilde{v}_{n}^{2}\right)dx-\frac{3}{4}\beta\int_{\mathbb{R}^{3}}\widetilde{u}_{n}^{2}\widetilde{v}_{n}dx+o(1).
\end{align*}
It is easy to see that $\int_{\mathbb{R}^{3}}u^{2}\leq b^{2}_{1},\ \int_{\mathbb{R}^{3}}v^{2}\leq b^{2}_{2}.$  From \eqref{LA188} and \eqref{LA222}, we have $$\lambda_{1}\left(b^{2}_{1}-\int_{\mathbb{R}^{3}} u^{2}dx\right)+\lambda_{2}\left(b^{2}_{2}-\int_{\mathbb{R}^{3}} v^{2}dx\right)=0,$$
so
\begin{align}\label{LA31}
\int_{\mathbb{R}^{3}}u^{2}= b^{2}_{1},\ \int_{\mathbb{R}^{3}}v^{2}= b^{2}_{2}.
\end{align}
Therefore, we know that
\begin{align*}
m^{+}_{\beta}(b_{1},b_{2})&=\lim_{n\rightarrow +\infty}J_{\beta}(u_{n},v_{n})=\lim_{n\rightarrow +\infty}J_{\beta}(\widetilde{u}_{n},\widetilde{v}_{n})+J_{\beta}(u,v)\\
&=\lim_{n\rightarrow +\infty}\left[\frac{1}{6}\int_{\mathbb{R}^{3}}(|\nabla \widetilde{u}_{n}|^{2}+|\nabla \widetilde{v}_{n}|^{2})dx-\frac{\beta}{4}\int_{\mathbb{R}^{3}}\widetilde{u}_{n}^{2}\widetilde{v}_{n}dx\right]+J_{\beta}(u,v)\\
& \geq\lim_{n\rightarrow +\infty}\bigg[\frac{1}{6}\int_{\mathbb{R}^{3}}(|\nabla \widetilde{u}_{n}|^{2}+|\nabla \widetilde{v}_{n}|^{2})dx\\
&\quad-\frac{|\beta|}{4}\left(\int_{\mathbb{R}^{3}}|\widetilde{u}_{n}|^{3}dx\right)^{\frac{2}{3}}\left(\int_{\mathbb{R}^{3}}|\widetilde{v}_{n}|^{3}dx\right)^{\frac{1}{3}}\bigg]+J_{\beta}(u,v)\\
&\geq\lim_{n\rightarrow +\infty}\frac{1}{6}\int_{\mathbb{R}^{3}}(|\nabla \widetilde{u}_{n}|^{2}+|\nabla \widetilde{v}_{n}|^{2})dx+J_{\beta}(u,v)\geq m^{+}_{\beta}(b_{1},b_{2}).
\end{align*}
Thus, $J_{\beta}(u,v)=m^{+}_{\beta}(b_{1},b_{2})\ \text{and}\ (u_{n},v_{n})\rightarrow (u,v)$ in $H^{1}(\mathbb{R}^{3})\times H^{1}(\mathbb{R}^{3}) .$
\end{proof}
\subsection{Existence of a second critical point of mountain pass type}
Next, we prove the existence of second critical point of mountain pass type for $J_{\beta}(u,v)\mid_{\mathrm{T}_{b_{1}}\times \mathrm{T}_{b_{2}}}$.
\begin{lemma}\label{Lem8}
Suppose that $J_{\beta}(u,v)<m^{+}_{\beta}(b_{1},b_{2})$. Then the value $t_{u,v}$ defined by Lemma \ref{Lem4} is negative.
\end{lemma}
\begin{proof}
From Lemma \ref{Lem4}, we know that the function $\Psi_{u,v}(t)$ has exactly two critical points $s_{u,v}<t_{u,v}\in \mathbb{R}$ and two zeros $c_{u,v}<d_{u,v}\in \mathbb{R}$ with $s_{u,v}<c_{u,v}<t_{u,v}<d_{u,v}$. If $d_{u,v}\leq0$, then $t_{u,v}<0.$ Assume by contradiction that $d_{u,v}>0$, if $0\in(c_{u,v}, t_{u,v})$, then $J_{\beta}(u,v)=\Psi_{u,v}(0)>0$ contract with $J_{\beta}(u,v)<m^{+}_{\beta}(b_{1},b_{2})<0$. Therefore, $c_{u,v}>0.$ By Lemma \ref{Lem4}-(2), we know that
\begin{align*}
&m^{+}_{\beta}(b_{1},b_{2})>J_{\beta}(u,v)=\Psi_{u,v}(0)\geq \inf_{t\in (-\infty,c_{u,v}]}\Psi_{u,v}(t)\geq J_{\beta}(u,v)(s_{u,v}\star (u,v))\\
&=\min\left\{J_{\beta}(t\star (u,v):t\in \mathbb{R}\ \text{and}\ (\int_{\mathbb{R}^{3}}(|\nabla (t\star u)|^{2}+|\nabla (t\star v)|^{2})dx)^{\frac{1}{2}}< R_{0} \right\}\\
&=J_{\beta}(s_{u,v}\star (u,v)\geq m^{+}_{\beta}(b_{1},b_{2}),
\end{align*}
which is a contradiction.
\end{proof}

\begin{lemma}\label{Lem9}
We have $$m^{-}_{\beta}(b_{1},b_{2}):=\inf_{(u,v)\in \mathcal{P}^{-}_{b_{1},b_{2}} }J_{\beta}(u,v)>0.$$
\end{lemma}
\begin{proof}
From Lemma \ref{LAA}, there exists $0<R_{0}<R_{1}$ such that $h(R_{0})=h(R_{1})=0$, $h(t)>0$ if and only if $t\in (R_{0},R_{1})$. Let $t_{max}$ be the  maximum point of the function $h(t)$ at the positive level $(R_{0},R_{1})$. For every $(u,v)\in \mathcal{P}^{-}_{b_{1},b_{2}}$, there exists $\tau_{u,v}\in \mathbb{R}$ such that $$\big(\int_{\mathbb{R}^{3}}(|\nabla (\tau_{u,v}\star u)|^{2}+|\nabla (\tau_{u,v}\star v)|^{2})dx\big)^{\frac{1}{2}}=t_{max},$$
and from Lemma \ref{Lem4} we know that $t_{u,v}\star (u,v)\in \mathcal{P}^{-}_{b_{1},b_{2}}$ then $t=t_{u,v},$  which implies that 0 is the unique strict maximum of $\Psi_{u,v}(t)$. Therefore,
\begin{align*}
J_{\beta}(u,v)&=\Psi_{u,v}(0)\geq \Psi_{u,v}(\tau_{u,v})=J_{\beta}(\tau_{u,v}\star (u,v))\\
&\geq h\Big(\big(\int_{\mathbb{R}^{3}}(|\nabla (\tau_{u,v}\star u)|^{2}+|\nabla (\tau_{u,v}\star v)|^{2})dx\big)^{\frac{1}{2}}\Big)=h(t_{max})>0,
\end{align*}
so for any $(u,v)\in \mathcal{P}^{-}_{b_{1},b_{2}} $, we have $$m^{-}_{\beta}(b_{1},b_{2}):=\inf_{(u,v)\in \mathcal{P}^{-}_{b_{1},b_{2}} }J_{\beta}(u,v)>0.$$
\end{proof}
Next, we use the idea introduced by Li and Zou \cite{LIZOU} to give the energy estimate.
We observe that $w_{\lambda,\mu}=\Big(\frac{\|w\|^{2}_{L^{2}(\mathbb{R}^{3})}}{\mu b^{2}}\Big)^{2}\frac{\|w\|_{L^{2}(\mathbb{R}^{3})}^{2}}{\mu^{\frac{3}{2}}b^{2}}w\Big(\big(\frac{\|w\|^{2}_{L^{2}(\mathbb{R}^{3})}}{\mu b^{2}}\big)x\Big)$ is a solution of the problem
\begin{equation}\label{intg}
\begin{cases}

-\Delta u+ \lambda u=\mu u^{3} & \text{in} \ \ \mathbb{R}^{3},\\
u>0 \ \ \ & \text{in}\ \ \mathbb{R}^{3},\\
 \int_{\mathbb{R}^{3}}u^{2}dx=b^{2},
\end{cases}
\end{equation}
where $w$ is the unique positive solution of $-\Delta u+  u= u^{3} \ \text{in} \ \ \mathbb{R}^{3}$.
\begin{lemma}\label{Lem1}
For any $b_{1},b_{2}>0,$ $\rho>0,\beta>0$, we have $$m^{-}_{\beta}(b_{1},b_{2})<\min\{m^{-}_{\beta}(0,b_{2}),m^{-}_{\beta}(b_{1},0)\}.$$
\end{lemma}
\begin{proof}
Let $(v,\lambda_{0})\in \mathrm{T}_{b_{2}}\times \mathbb{R}^{+}$ be the unique solution of equation \eqref{intg} with parameters $\mu_{2}$ and $b_{2}$. Let $u(x)=c\frac{\varphi(x)}{|x|^{m}}, \ m< \frac{1}{2}$ $$\varphi(x)\in C^{\infty}_{0}(B_{2}(0)),\  0\leq\varphi(x)\leq1,\ \varphi (x)=1 \ \text{in}\ B_{1}(0).$$
Then $u\in H^{1}(\mathbb{R}^{3})$  and $u\in \mathrm{T}_{b_{1}} $ for some suitable c. Therefore, $(t\star u, v)\in \mathrm{T}_{b_{1}}\times \mathrm{T}_{b_{2}}$ for any $t\in \mathbb{R}$.  So, $$\int_{\mathbb{R}^{3}}|t\star u|^{2}v^{2}dx=C_{0}e^{(3-2m)t}\int_{\mathbb{R}^{3}}\frac{\varphi^{2}(e^{t}(x))}{|x|^{2m}}v^{2}(x)dx,$$
$$\int_{\mathbb{R}^{3}}|t\star u|^{2}vdx=C_{0}e^{(3-2m)t}\int_{\mathbb{R}^{3}}\frac{\varphi^{2}(e^{t}(x))}{|x|^{2m}}v(x)dx.$$
From the decay properties of \cite{LINI} that $v$ decays exponentially$$v(x)=O(|x|^{-\frac{1}{2}}e^{-\lambda^{\frac{1}{2}}_{0}|x|})\ \text{and}\  |u(x)|\leq M\  \text{in} \ \mathbb{R}^{3}.$$  Thus, $$0<\int_{\mathbb{R}^{3}}\frac{v^{2}(x)}{|x|^{2m}}dx\leq C\Big(\int_{B_{R}(0)}\frac{1}{|x|^{2m}}+\int_{\mathbb{R}^{3}\setminus B_{R}(0)}|x|^{-1-2m}e^{-\lambda^{\frac{1}{2}}_{0}2|x|}\Big)<+\infty,$$
$$0<\int_{\mathbb{R}^{3}}\frac{v(x)}{|x|^{2m}}dx\leq C\Big(\int_{B_{R}(0)}\frac{1}{|x|^{2m}}+\int_{\mathbb{R}^{3}\setminus B_{R}(0)}|x|^{-\frac{1}{2}-2m}e^{-\lambda^{\frac{1}{2}}_{0}|x|}\Big)<+\infty.$$
By the Dominated Convergence Theorem, we obtain that $$\lim_{t\rightarrow -\infty}\int_{\mathbb{R}^{3}}\frac{\varphi^{2}(e^{t}(x))}{|x|^{2m}}v^{2}(x)dx=\int_{\mathbb{R}^{3}}\frac{v^{2}(x)}{|x|^{2m}}dx=C_{1}\in(0,+\infty),$$
$$\lim_{t\rightarrow -\infty}\int_{\mathbb{R}^{3}}\frac{\varphi^{2}(e^{t}(x))}{|x|^{2m}}v(x)dx=\int_{\mathbb{R}^{3}}\frac{v(x)}{|x|^{2m}}dx=C_{2}\in(0,+\infty),$$
so
\begin{align}\label{LA33}
\int_{\mathbb{R}^{3}}|t\star u|^{2}v^{2}dx=C_{0}e^{(3-2m)t}\int_{\mathbb{R}^{3}}\frac{\varphi^{2}(e^{t}(x))}{|x|^{2m}}v^{2}(x)dx&=C_{0}e^{(3-2m)t}(C_{1}+o(1))\geq \frac{C_{0}C_{1}}{2}e^{(3-2m)t},
\end{align}
as $t\rightarrow-\infty$.
\begin{align}\label{LA34}
\int_{\mathbb{R}^{3}}|t\star u|^{2}vdx&=C_{0}e^{(3-2m)t}\int_{\mathbb{R}^{3}}\frac{\varphi^{2}(e^{t}(x))}{|x|^{2m}}v(x)dx\\\nonumber
&=C_{0}e^{(3-2m)t}(C_{2}+o(1))\geq \frac{C_{0}C_{2}}{2}e^{(3-2m)t}\ \text{as}\ t\rightarrow-\infty.
\end{align}
From Lemma \ref{Lem4}, we know that $t_{\ast}=t_{(t\star u,v)}$ is a local maximum of $\Psi_{u,v}(t)$ and $t_{(t\star u,v)}\star (t\star u,v)\in \mathcal{P}^{-}_{b_{1},b_{2}}$, so
\begin{equation}\label{LA32}
\begin{aligned}
0=P(t_{\star}\star (t\star u,v))&=e^{2(t+t_{\star})}\int_{\mathbb{R}^{3}}|\nabla u|^{2}dx+e^{2t_{\star}}\int_{\mathbb{R}^{3}}|\nabla v|^{2}dx-\frac{3e^{3(t+t_{\star})}}{4}\mu_{1}\int_{\mathbb{R}^{3}}u^{4}dx\\\nonumber
&\quad-\frac{3e^{3t_{\star}}}{4}\mu_{2}\int_{\mathbb{R}^{3}}v^{4}dx-\frac{3e^{3t_{\star}}}{2}\rho\int_{\mathbb{R}^{3}}|t\star u|^{2}v^{2}dx-\frac{3\beta}{4}e^{\frac{3t_{\star}}{2}}\int_{\mathbb{R}^{3}}|t\star u|^{2}vdx.
\end{aligned}
\end{equation}
From \eqref{LA33},\ \eqref{LA34}, \eqref{LA32} and  let $t\rightarrow-\infty$, we have $$e^{2t_{\star}}\int_{\mathbb{R}^{3}}|\nabla v|^{2}dx=\frac{3e^{3t_{\star}}}{4}\mu_{2}\int_{\mathbb{R}^{3}}v^{4}dx.$$
Since $(v,\lambda_{0})\in \mathrm{T}_{b_{2}}\times \mathbb{R}^{+}$ be the unique solution of equation \eqref{intg} with parameters $\mu_{2}$ and $b_{2}$, by the Pohozaev identity for equation \eqref{intg} with parameters $\mu_{2}$ and $b_{2}$, we can get $e^{3t_{\star}}\rightarrow1$ as $t\rightarrow -\infty. $ Therefore, from Lemma 2.1 in \cite{LIZOU}, we have $$e^{2t_{\star}}\int_{\mathbb{R}^{3}}|\nabla v|^{2}dx-\frac{3e^{3t_{\star}}}{4}\mu_{2}\int_{\mathbb{R}^{3}}v^{4}dx<m_{\beta}(0,b_{2}),$$ therefore
\begin{align*}
m^{-}_{\beta}(b_{1},b_{2})&\leq J_{\beta}(t_{\star}\star(t\star u,v))\\
&=\frac{e^{2(t+t_{\star})}}{2}\int_{\mathbb{R}^{3}}|\nabla u|^{2}dx+\frac{e^{2t_{\star}}}{2}\int_{\mathbb{R}^{3}}|\nabla v|^{2}dx-\frac{e^{3(t+t_{\star})}}{4}\mu_{1}\int_{\mathbb{R}^{3}}u^{4}dx\\\nonumber
&\quad-\frac{e^{3t_{\star}}}{4}\mu_{2}\int_{\mathbb{R}^{3}}v^{4}dx-\frac{e^{3t_{\star}}}{2}\rho\int_{\mathbb{R}^{3}}|t\star u|^{2}v^{2}dx-\frac{\beta}{2}e^{\frac{3t_{\star}}{2}}\int_{\mathbb{R}^{3}}|t\star u|^{2}vdx\\
&=\frac{e^{2t_{\star}}}{2}\int_{\mathbb{R}^{3}}|\nabla v|^{2}dx-\frac{e^{3t_{\star}}}{4}\mu_{2}\int_{\mathbb{R}^{3}}v^{4}dx-\frac{C_{0}C_{1}}{2}e^{(3-2m)t}-\beta\frac{C_{0}C_{2}}{2}e^{(3-2m)t}\\
&<m^{-}_{\beta}(0,b_{2})-\frac{C_{0}C_{1}}{2}e^{(3-2m)t}-\beta\frac{C_{0}C_{2}}{2}e^{(3-2m)t},
\end{align*}
from which, we see that for sufficiently small $t\ll -1$, there holds $m^{-}_{\beta}(b_{1},b_{2})<m^{-}_{\beta}(0,b_{2}).$ Similarly, we can deduce that $$m^{-}_{\beta}(b_{1},b_{2})<m^{-}_{\beta}(b_{1},0).$$
\end{proof}
\begin{lemma}\label{Lem2}
There is a radial symmetric Palais-Smale sequence of $J|_{\mathrm{T}_{b_{1}}\times \mathrm{T}_{b_{2}}}$ at level $m^{-}_{\beta}(b_{1},b_{2})$ with the additional properties $P(u_{n},v_{n})\rightarrow0$ and $u^{-}_{n},v^{-}_{n}\rightarrow 0$ a.e. in $\mathbb{R}^{3}.$ Then up to a subsequence $(u_{n},v_{n})\rightarrow (u,v)$ in $H^{1}(\mathbb{R}^{3})\times H^{1}(\mathbb{R}^{3})$, where $(u,v)$ is a positive solution of \eqref{23} for some $\lambda_{1},\lambda_{2}>0.$
\end{lemma}
\begin{proof}
From \eqref{pp}, we define the functional $\widetilde{J}:\mathbb{R}\times H^{1}(\mathbb{R}^{3})\times H^{1}(\mathbb{R}^{3})\rightarrow \mathbb{R}$  as following
\begin{align*}
\widetilde{J}_{\beta}(t,u,v)=J_{\beta}(t\star (u,v))&=\frac{e^{2t}}{2}\int_{\mathbb{R}^{3}}(|\nabla u|^{2}+|\nabla v|^{2})dx-\frac{\beta}{2}\int_{\mathbb{R}^{3}}e^{\frac{3t}{2}}u^{2}vdx\\\nonumber
&\quad-\frac{e^{3t}}{4}\int_{\mathbb{R}^{3}}\left(\mu_{1}u^{4}+\mu_{2}v^{4}+2\rho u^{2}v^{2}\right)dx.
\end{align*}
It is easy to see that $\widetilde{J}$ is of class $C^{1}$. 
Denoting $J_{\beta}^{c}=\{(u,v)\in\mathrm{T}_{b_{1}}\times \mathrm{T}_{b_{2}}:J_{\beta}(u,v)\leq c \}$ and introduce the minimax class
\begin{align}\label{LA23}
\Gamma=\big\{\gamma=(\alpha,\varphi_{1},\varphi_{2})\in C([0,1],\mathbb{R}\times \mathrm{T}_{b_{1},r}\times \mathrm{T}_{b_{2},r}):\gamma(0)\in \{0\}\times \mathcal{P}^{+}_{b_{1},b_{2}}, \gamma(1)\in \{0\}\times J_{\beta}^{2m(b_{1},b_{2})}\big\},
\end{align}
with associated minimax level
$$\sigma_{\beta}(b_{1},b_{2})=\inf_{\gamma\in \Gamma}\max_{t\in [0,1]}\widetilde{J}_{\beta}(\gamma(t)).$$
Let $(u,v)\in \mathrm{T}_{b_{1},r}\times \mathrm{T}_{b_{2},r}$, it is easy to see
$$\int_{\mathbb{R}^{3}}\big(|\nabla (t \star  u)|^{2}+|\nabla (t \star  v)|^{2}\big)dx\rightarrow 0\ \text{as}\ t\rightarrow-\infty,$$
$$J_{\beta}(t \star  u,t \star  v)\rightarrow -\infty\ \text{as}\ t\rightarrow+\infty.$$
From Lemma \ref{Lem4}, there exists  $t_{0}\ll-1$ and $t_{1}\gg1$ such that $$t_{0}\star (u,v)\in \mathcal{P}^{+}_{b_{1},b_{2}},\ \ t_{1}\star (u,v)\in J_{\beta}^{2m_{\beta}(b_{1},b_{2})}$$
and
\begin{equation}\label{LA25}
\gamma_{u,v}(t):t\in [0,1]\mapsto(0,((1-t )t_{0})+tt_{1})\star (u,v))\in \mathbb{R}\times \mathrm{T}_{b_{1},r}\times \mathrm{T}_{b_{2},r}
\end{equation}
is a path in $\Gamma$. Thus, $\sigma(b_{1},b_{2})$ is a real number. Next, we prove that for every
\begin{equation}\label{LA24}
  \gamma=(\alpha,\varphi_{1},\varphi_{2})\in \Gamma\  \text{there exists}\ t_{\gamma}\in (0,1)\  \text{such that}\ \alpha(t_{\gamma})\star (\varphi_{1}(t_{\gamma}),\varphi_{2}(t_{\gamma}))\in \mathcal{P}^{-}_{b_{1},b_{2}}.
\end{equation}
Indeed, $\gamma(0)=(0,\varphi_{1}(0),\varphi_{2}(0))\in \{0\}\times \mathcal{P}^{+}_{b_{1},b_{2}}$, by Lemma \ref{Lem4} and the fact that  $t\star (u,v)\in \mathcal{P}_{b_{1},b_{2}} $  if and only if $(\Psi_{u,v})'(t)=0$, we have $$t_{\alpha(0)\star (\varphi_{1}(0),\varphi_{2}(0))}=t_{(\varphi_{1}(0),\varphi_{2}(0))}>s_{(\varphi_{1}(0),\varphi_{2}(0))}=0.$$  Since $J_{\beta}(\alpha(1)\star (\varphi_{1}(1),\varphi_{2}(1)))=\widetilde{J}_{\beta}(\gamma(1))\leq 2m_{\beta}(b_{1},b_{2})$, from Lemma \ref{Lem8}, we have that $$t_{(\varphi_{1}(1), \varphi_{2}(1))}<0.$$ From Lemma \ref{Lem4}, the map $t_{\alpha(t)\star(\varphi_{1}(t),\varphi_{2}(t))}$ is continuous in $t,$ so there exists \\ $t_{\gamma}\in (0,1)$ such that $t_{\alpha(t_{\gamma})\star(\varphi_{1}(t_{\gamma}),\varphi_{2}(t_{\gamma}))}=0$, so $\text{for every }\  \gamma=(\alpha,\varphi_{1},\varphi_{2})\in \Gamma\  \text{there exists}\ t_{\gamma}\in (0,1)\  \text{such that}\ \alpha(t_{\gamma})\star (\varphi_{1}(t_{\gamma}),\varphi_{2}(t_{\gamma}))\in \mathcal{P}^{-}_{b_{1},b_{2}}. $ This implies that $$\max_{\gamma([0,1])}\widetilde{J}_{\beta}\geq \widetilde{J}_{\beta}(\gamma(t_{\gamma}))=J_{\beta}(\alpha(t_{\gamma})\star (\varphi_{1}(t_{\gamma}),\varphi_{2}(t_{\gamma})\geq \inf_{\mathcal{P}^{-}_{b_{1},b_{2}}\cap \mathrm{T}_{b_{1},r}\times \mathrm{T}_{b_{2},r}}J_{\beta}.$$ So $$\sigma_{\beta}(b_{1},b_{2})\geq \inf_{\mathcal{P}^{-}_{b_{1},b_{2}}\cap \mathrm{T}_{b_{1},r}\times \mathrm{T}_{b_{2},r}}J_{\beta}.$$ On the other hand, if $(u,v)\in \mathcal{P}^{-}_{b_{1},b_{2}}\cap \mathrm{T}_{b_{1},r}\times \mathrm{T}_{b_{2},r}$, then $\gamma_{u,v}$ defined in \eqref{LA25} is a path in $\Gamma$ with $$J_{\beta}(u,v)=\widetilde{J}_{\beta}(t,u,v)=\max_{\gamma_{u,v}([0,1])}\widetilde{J}_{\beta}\geq \sigma_{\beta}(b_{1},b_{2}),$$ so $$\inf_{\mathcal{P}^{-}_{b_{1},b_{2}}\cap \mathrm{T}_{b_{1},r}\times \mathrm{T}_{b_{2},r}}J_{\beta}\geq \sigma_{\beta}(b_{1},b_{2}).$$ Thus, $$\inf_{\mathcal{P}^{-}_{b_{1},b_{2}}\cap \mathrm{T}_{b_{1},r}\times \mathrm{T}_{b_{2},r}}J_{\beta}= \sigma_{\beta}(b_{1},b_{2}).$$ From Lemma \ref{Lem9}, we know that
\begin{align}\label{LA26}
\sigma_{\beta}(b_{1},b_{2})&=\inf_{\mathcal{P}^{-}_{b_{1},b_{2}}\cap \mathrm{T}_{b_{1},r}\times \mathrm{T}_{b_{2},r}}J_{\beta}>0\geq\sup_{\mathcal{P}^{+}_{b_{1},b_{2}}\cup J_{\beta}^{2m(b_{1},b_{2})}\cap \mathrm{T}_{b_{1},r}\times \mathrm{T}_{b_{2},r}}J_{\beta}\\\nonumber
&=\sup_{\{0\}\times \mathcal{P}^{+}_{b_{1},b_{2}}\cup \{0\}\times J^{2m(b_{1},b_{2})}\cap \mathrm{T}_{b_{1},r}\times \mathrm{T}_{b_{2},r}}\widetilde{J}_{\beta}.
\end{align}
Let $\mathcal{F}=\{\gamma([0,1]):\gamma\in \Gamma\}$. By the terminology in [\cite{NGB}, section 5], which implies that $\{\gamma([0,1]):\gamma\in \Gamma\}$ is a homotopy stable family of compact subset of $\mathbb{R}\times \mathrm{T}_{b_{1},r}\times \mathrm{T}_{b_{2},r} $ with extended closed boundary $\{0\}\cup \times \mathcal{P}^{+}_{b_{1},b_{2}}\times \{0\}\times J^{2m(b_{1},b_{2})}_\beta$, and the superlevel set$\{\widetilde{J}_{\beta}\geq\sigma\}$ is a dual set for $\mathcal{F}$, which implies that the assumption in [\cite{NGB}, Theorem 5.2] are satisfied. Therefore, taking any minimizing sequence $\{\gamma_{n}([0,1]), \gamma_{n}=(\alpha_{n}, \varphi_{1,n},\varphi_{2,n})\}$ for $\sigma$ with the property that $\alpha(t)=0,\ \varphi_{1,n}\geq0,\ \varphi_{2,n}\geq 0$ for every $t\in [0,1]$, there exists a sequence $(t_{n},u_{n},v_{n})\subset \mathbb{R}\times \mathrm{T}_{b_{1},r}\times \mathrm{T}_{b_{2},r}$ such that $\widetilde{J}_{\beta}(t_{n},u_{n},v_{n})\rightarrow\sigma(b_{1},b_{2})$ and
\begin{equation}\label{LA27}
\partial_{t}\widetilde{J}_{\beta}(t_{n},u_{n},v_{n})\rightarrow0, \ \|\partial_{(u,v)}\widetilde{J}_{\beta}(t_{n},u_{n},v_{n})\|_{T_{u_{n}\mathrm{T}_{b_{1},r}}\times T_{v_{n}\mathrm{T}_{b_{2},r}}}\rightarrow0
\end{equation}
\begin{equation}\label{LA28}
|t_{n}|+dist((u_{n},v_{n}),(\varphi_{1,n}([0,1]),\varphi_{2,n}([0,1])))\rightarrow 0.
\end{equation}
Let $(\widetilde{u}_{n},\widetilde{v}_{n})=t_{n}\star(u_{n},v_{n})\in \mathrm{T}_{b_{1},r}\times \mathrm{T}_{b_{2},r}$. By \eqref{LA28}, we know that $t_{n}$ is bounded and $\widetilde{u}^{-}_{n},\widetilde{v}^{-}_{n}\rightarrow 0$ a.e. in $\mathbb{R}^{3}$. From \eqref{LA27}, we can get $$P(\widetilde{u}_{n},\widetilde{v}_{n})=\partial_{t}\widetilde{J}_{\beta}(t_{n},u_{n},v_{n})\rightarrow0,$$ and
\begin{align*}
J_{\beta}'(\widetilde{u}_{n},\widetilde{v}_{n})[\phi,\psi]=\partial_{(u,v)}\widetilde{J}_{\beta}(t_{n}, u_{n},v_{n})[(-t_{n})\star(\phi,\psi)]&=o(1)\|(-t_{n})\star(\phi,\psi)\|_{H^{1}(\mathbb{R}^{3})\times H^{1}(\mathbb{R}^{3})}\\
&=o(1)\|(\phi,\psi)\|_{H^{1}(\mathbb{R}^{3})\times H^{1}(\mathbb{R}^{3})},
\end{align*}
 therefore $(\widetilde{u}_{n},\widetilde{v}_{n})$ is a radial Palais-Smale sequence of $\widetilde{J}_{\beta}|_{\mathrm{T}_{b_{1},r}\times \mathrm{T}_{b_{2},r}}$ and hence a radial symmetric Palais-Smale sequence of $\widetilde{J}_{\beta}|_{\mathrm{T}_{b_{1}}\times \mathrm{T}_{b_{2}}}$ at level $\sigma_{\beta}(b_{1},b_{2}).$
 By the same arguments as Lemma \ref{Lem55}, we can prove prove the $H^{1}$ convergence of the Palais-Smale sequence. We should point out that there are several differences form the proof of Lemma
 \ref{Lem55}. First, to eliminate the disappearance of the solutions, we use $m^{-}(b_{1},b_{2})>0$ to get a contradiction. To eliminate the semi-trivial solution, we use  Br\'ezis-Lieb lemma and Lemma \ref{Lem1} to get a contradiction.

Hence,
\begin{align*}
m^{-}_{\beta}(b_{1},b_{2})&=\lim_{n\rightarrow +\infty}J_{\beta}(\widetilde{u}_{n},\widetilde{v}_{n})=\lim_{n\rightarrow +\infty}J_{\beta}(w_{n},\sigma_{n})+J_{\beta}(u,v)\\
&=\lim_{n\rightarrow +\infty}\left[\frac{1}{6}\int_{\mathbb{R}^{3}}(|\nabla w_{n}|^{2}+|\nabla \sigma_{n}|^{2})dx-\frac{\beta}{4}\int_{\mathbb{R}^{3}}w_{n}^{2}\sigma_{n}dx\right]+J_{\beta}(u,v)\\
& \geq\lim_{n\rightarrow +\infty}\bigg[\frac{1}{6}\int_{\mathbb{R}^{3}}(|\nabla w_{n}|^{2}+|\nabla \sigma_{n}|^{2})dx-\frac{|\beta|}{4}\left(\int_{\mathbb{R}^{3}}|w_{n}|^{3}dx\right)^{\frac{2}{3}}\left(\int_{\mathbb{R}^{3}}|\sigma_{n}|^{3}dx\right)^{\frac{1}{3}}\bigg]+J_{\beta}(u,v)\\
&\geq\lim_{n\rightarrow +\infty}\frac{1}{6}\int_{\mathbb{R}^{3}}(|\nabla w_{n}|^{2}+|\nabla \sigma_{n}|^{2})dx+J_{\beta}(u,v)\geq m^{-}_{\beta}(b_{1},b_{2}).
\end{align*}
Thus, $J_{\beta}(u,v)=m^{-}_{\beta}(b_{1},b_{2})\ \text{and}\ (\widetilde{u}_{n},\widetilde{v}_{n})\rightarrow (u,v)$ in $H^{1}(\mathbb{R}^{3})\times H^{1}(\mathbb{R}^{3}) .$
\end{proof}

\begin{lemma}\label{Lem3}
We have
$$m^{-}_{\beta}(b_{1},b_{2})=\inf_{(u,v)\in \mathrm{T}_{b_{1},r}\times \mathrm{T}_{b_{2},r} }\max_{t\in \mathbb{R}}J_{\beta}(t\star(u,v)).$$
\end{lemma}
\begin{proof}
On the one hand
$$m^{-}_{\beta}(b_{1},b_{2})=\inf_{(u,v)\in \mathcal{P}^{-}_{b_{1},b_{2}} }J_{\beta}(u,v)=J_{\beta}(\widetilde{u},\widetilde{v}).$$ Then, by Lemma \ref{Lem4}, we have $$m^{-}_{\beta}(b_{1},b_{2})=J_{\beta}(\widetilde{u},\widetilde{v})=\max_{t\in \mathbb{R}} J_{\beta}(t\star(\widetilde{u},\widetilde{v}))\geq \inf_{(u,v)\in \mathrm{T}_{b_{1},r}\times \mathrm{T}_{b_{2},r} }\max_{t\in \mathbb{R}}J_{\beta}(t\star(u,v)).$$
On the other hand, for any $(u,v)\in \mathrm{T}_{b_{1},r}\times \mathrm{T}_{b_{2},r} ,$ we have $t_{u,v}\star (u,v)\in \mathcal{P}^{-}_{b_{1},b_{2}}$ and hence $$\max_{t\in \mathbb{R}}J_{\beta}(t\star(u,v))=J_{\beta}(t_{u,v}\star (u,v))\geq m^{-}_{\beta}(b_{1},b_{2}).$$
\end{proof}
Before giving the estimate of $m^{-}_{\beta}(b_{1},b_{2})$, we would like to study the dependence of $m^{-}_{\beta}(b_{1},b_{2})$ on $\beta$.
\begin{lemma}\label{Lem10}
For any $0<\beta_{1}<\beta_{2}$, then $m^{-}_{\beta_{2}}(b_{1},b_{2})\leq m^{-}_{\beta_{1}}(b_{1},b_{2}) \leq m^{-}_{0}(b_{1},b_{2}).$
\end{lemma}
\begin{proof}
By Lemma \ref{Lem3}
\begin{align*}
m^{-}_{\beta_{2}}(b_{1},b_{2})\leq \max_{t\in \mathbb{R}}J_{\beta_{2}}(t\star(\widetilde{u}_{\beta_{1}},\widetilde{v}_{\beta_{1}}))&\leq \max_{t\in \mathbb{R}}J_{\beta_{1}}(t\star(\widetilde{u}_{\beta_{1}},\widetilde{v}_{\beta_{1}}))\\
&=J_{\beta_{1}}(\widetilde{u}_{\beta_{1}},\widetilde{v}_{\beta_{1}})=m^{-}_{\beta_{1}}(b_{1},b_{2}),
\end{align*}
$$m^{-}_{\beta_{1}}(b_{1},b_{2})\leq \max_{t\in \mathbb{R}}J_{\beta_{1}}(t\star(\widetilde{u}_{0},\widetilde{v}_{0}))\leq \max_{t\in \mathbb{R}}J_{0}(t\star(\widetilde{u}_{0},\widetilde{v}_{0}))=J_{0}(\widetilde{u}_{0},\widetilde{v}_{0})=m^{-}_{0}(b_{1},b_{2}).$$
\end{proof}

\begin{proof}[\bf Proof of  Theorem \ref{Th1}]
Let us consider a minimizing sequence $(u_{n},v_{n})$ for $J_{\beta}(u,v)\mid_{A_{R_{0}}}$ , we assume that $(u_{n},v_{n})\in \mathrm{T}_{b_{1}}\times \mathrm{T}_{b_{2}} $ is radial decreasing for every $n$. Furthermore, by Lemma \ref{Lem4}, for every $n$ we can take  $s_{u_{n},v_{n}}\ast(u_{n},v_{n})\in \mathcal{P}^{+}_{b_{1},b_{2}}$ such that $(\int_{\mathbb{R}^{3}}(|\nabla (s_{u_{n},v_{n}}\star u_{n})|^{2}+|\nabla (s_{u_{n},v_{n}}\star v_{n})|^{2})dx)^{\frac{1}{2}}\leq R_{0}$  and \begin{align*}
J_{\beta}(s_{u,v}\star (u_{n},v_{n}))&=\min\bigg\{J_{\beta}(t\star (u_{n},v_{n}):t\in \mathbb{R}\ \text{and}\left(\int_{\mathbb{R}^{3}}(|\nabla (t\star u_{n})|^{2}+|\nabla (t\star v_{n})|^{2})dx\right)^{\frac{1}{2}}< R_{0} \bigg\}\\
&<J_{\beta}(u_{n},v_{n}).
\end{align*}
Thus, we obtain a new minimizing sequence $\{w_{n},\sigma_{n}\}=\{s_{u_{n},u_{n}}\star v_{n},s_{u_{n},v_{n}}\star v_{n}\}$ with $(w_{n},\sigma_{n})\in \mathrm{T}_{b_{1,r}}\times \mathrm{T}_{b_{2,r}}\cap\mathcal{P}^{+}_{a,\mu} $ radially decreasing for every $n$. By Lemma \ref{Lem7}, we have $\left(\int_{\mathbb{R}^{3}}(|\nabla  w_{n}|^{2}+|\nabla  \sigma_{n}|^{2})dx\right)^{\frac{1}{2}}\leq R_{0}$ for every $n$ and hence by Ekeland's variational principle in a standard way, we know the existence of a new minimizing sequence for $\{u_{n},v_{n}\}\subset A_{R_{0}}$ for $m_{\beta}^{+}(b_{1},b_{2})$ with $\|(u_{n},v_{n})-(w_{n},\sigma_{n})\|\rightarrow 0$ as $n\rightarrow +\infty $, which is also a Palais-Smale sequence for $J_{\beta}(u,v)$ on $\mathrm{T}_{b_{1}}\times \mathrm{T}_{b_{2}}$. By the boundedness of $\{(w_{n},\sigma_{n})\}$, $\|(u_{n},v_{n})-(w_{n},\sigma_{n})\|\rightarrow 0$,  Br\'ezis-Lieb lemma and Sobolev embedding theorem, we have $$P_{b_{1},b_{2}}(u_{n},v_{n})=P_{b_{1},b_{2}}(w_{n},\sigma_{n})+o(1)\rightarrow 0\ \ \text{and}\ u^{-}_{n},v^{-}_{n}\rightarrow 0 \ a.e. \text{in}\ \mathbb{R}^{3}  .$$

Hence $(u_{n},v_{n})$ satisfies all the assumptions of Lemma \ref{Lem55}, as a consequence, up to a subsequence $(u_{n},v_{n})\rightarrow (u,v)$ strongly in $H^{1}(\mathbb{R}^{3})\times H^{1}(\mathbb{R}^{3})$ and $(u,v)$ is an interior local minimizer for $J_{\beta}(u,v)|_{A_{R_{0}}}$.
Finally, we prove that  $(u,v)$ is in $\mathcal{P}^{+}_{b_{1},b_{2}}(u, v)$. i.e. $\beta\int_{\mathbb{R}^{3}}u^{2}vdx>0$. Indeed, from  $(u_{n},v_{n})\in \mathcal{P}^{+}_{b_{1},b_{2}}(u, v)$, we have $\beta\int_{\mathbb{R}^{3}}u^{2}vdx\geq0$. Assume $\beta\int_{\mathbb{R}^{3}}u^{2}vdx=0$, then from Pohozaev identity, we have
\begin{align*}
\int_{\mathbb{R}^{3}}(|\nabla u|^{2}+|\nabla v|^{2})dx
=\frac{3}{4}\int_{\mathbb{R}^{3}}(\mu_{1}u^{4}+\mu_{2}v^{4}+2\rho u^{2}v^{2})dx.
\end{align*}
Thus, from \eqref{LA11}-\eqref{LA13}, we have
\begin{align*}
\int_{\mathbb{R}^{3}}(|\nabla u|^{2}+|\nabla v|^{2})dx&=
\frac{3}{4}\int_{\mathbb{R}^{3}}\left(\mu_{1}u^{4}+\mu_{2}v^{4}+2\rho u^{2}v^{2}\right)dx\\\nonumber
&\leq\frac{3\mu_{1}}{4} C^{4}_{3,4}b_{1}\|\nabla u\|^{3}_{L^{2}(\mathbb{R}^{3})}+\frac{3\mu_{2}}{4} C^{4}_{3,4}b_{2}\|\nabla v\|^{3}_{L^{2}(\mathbb{R}^{3})}\\\nonumber
&\quad+\frac{3}{4}\rho C^{4}_{3,4}b^{\frac{1}{2}}_{1}b^{\frac{1}{2}}_{2}\left[\|\nabla u\|^{2}_{L^{2}(\mathbb{R}^{3})}+\|\nabla v\|^{2}_{L^{2}(\mathbb{R}^{3})}\right]^{\frac{3}{2}}\\\nonumber
&\leq \frac{3}{4}C^{4}_{3,4}(\mu_{1} b_{1}+\mu_{2} b_{2}+\rho b^{\frac{1}{2}}_{1}b^{\frac{1}{2}}_{2})\left[\|\nabla u\|^{2}_{L^{2}(\mathbb{R}^{3})}+\|\nabla v\|^{2}_{L^{2}(\mathbb{R}^{3})}\right]^{\frac{3}{2}},
\end{align*}
so,
\begin{align}\label{U}
\left[\|\nabla u\|^{2}_{L^{2}(\mathbb{R}^{3})}+\|\nabla v\|^{2}_{L^{2}(\mathbb{R}^{3})}\right]^{\frac{1}{2}} \geq \frac{4}{3C^{4}_{3,4}(\mu_{1} b_{1}+\mu_{2} b_{2}+\rho b^{\frac{1}{2}}_{1}b^{\frac{1}{2}}_{2})}.
\end{align}
From Corollary \ref{Lem6} and the definition of $R_{0}$ in Lemma \ref{LAA}, we have
\begin{align}\label{U1}
\left[\|\nabla u\|^{2}_{L^{2}(\mathbb{R}^{3})}+\|\nabla v\|^{2}_{L^{2}(\mathbb{R}^{3})}\right]^{\frac{1}{2}}<R_{0}\leq\widetilde{t}=\frac{2}{3C^{4}_{3,4}(\mu_{1} b_{1}+\mu_{2} b_{1}+\rho b^{\frac{1}{2}}_{1}b^{\frac{1}{2}}_{2})}.
\end{align}
a contradiction. Thus $\beta\int_{\mathbb{R}^{3}}u^{2}vdx>0$ and $(u,v)\in\mathcal{P}^{+}_{b_{1},b_{2}}(u, v)$. From Lemma \ref{Lem7}, we know that $(u,v)$ is a ground normalized solution. From Lemma \ref{Lem2}, we get  a second critical point of mountain pass type for  $J_{\beta}(u,v)\mid_{\mathrm{T}_{b_{1}}\times \mathrm{T}_{b_{2}}}$.

{\bf Next, we prove the second part of Theorem \ref{Th1}}, that is  the limit behavior of the ground state solution as $\beta\rightarrow0.$

For $b_{1}, b_{2}>0$ fixed, from the proof Lemma \ref{LAA}, we can deduce that when $\beta\rightarrow0$, then $R_{0}(b_{1},b_{2},\rho,\beta)\rightarrow0$. By corollary \ref{Lem6}, when $(\widehat{u},\widehat{v})$ is the ground normalized solution obtained in Theorem\  \ref{Th1},\ then $$\big(\int_{\mathbb{R}^{3}}(|\nabla  u|^{2}+|\nabla  v|^{2})dx\big)^{\frac{1}{2}}<R_{0}(b_{1},b_{2},\rho,\beta)\rightarrow0,$$ and
\begin{align*}
0>&m^{+}_{\beta}(b_{1},b_{2})=J_{\beta}(u,v)\\
&\geq\frac{1}{2}\int_{\mathbb{R}^{3}}(|\nabla u|^{2}+|\nabla v|^{2})dx-\mathcal{D}_{1}\big(\int_{\mathbb{R}^{3}}|\nabla u|^{2}dx\big)^{\frac{3}{2}}-\mathcal{D}_{2}\big(\int_{\mathbb{R}^{3}}|\nabla v|^{2}dx\big)^{\frac{3}{2}}\\\nonumber
&\quad-\rho\mathcal{D}_{3}\big(\int_{\mathbb{R}^{3}}(|\nabla u|^{2}+|\nabla v|^{2})dx\big)^{\frac{3}{2}}-|\beta|\mathcal{D}_{4}\big(\int_{\mathbb{R}^{3}}(|\nabla u|^{2}+|\nabla v|^{2})dx\big)^{\frac{3}{4}}\\\nonumber
&\geq h\Big(\big(\int_{\mathbb{R}^{3}}(|\nabla u|^{2}+|\nabla v|^{2})dx\big)^{\frac{1}{2}}\Big)\rightarrow 0,
\end{align*}
thus $$m^{+}_{\beta}(b_{1},b_{2})\rightarrow0.$$

Denote $\{(\widetilde{u}_{\beta},\widetilde{v}_{\beta}):0<\beta<\overline{\beta}\}$ with $\overline{\beta}$ small enough. By the same arguments as Lemma \ref{Lem55}, we can obtain the $H^{1}$ convergence. First, to eliminate the disappearance of the solutions, we use Lemma \ref{Lem10} and  $0<m^{-}_{\beta}(b_{1},b_{2})<m^{-}_{0}(b_{1},b_{2})$ to get a contradiction. To eliminate the semi-trivial solution, we use  Br\'ezis-Lieb lemma and energy comparison to get a contradiction.

Hence,
\begin{align*}
m^{-}_{0}(b_{1},b_{2})&=\lim_{\beta\rightarrow 0}J_{\beta}(\widetilde{u}_{\beta},\widetilde{v}_{\beta})=\lim_{\beta\rightarrow 0}J_{\beta}(w_{\beta},\sigma_{\beta})+\lim_{\beta \rightarrow 0}J_{\beta}(\widetilde{u},\widetilde{v})\\
&=\lim_{\beta \rightarrow 0}\left[\frac{1}{6}\int_{\mathbb{R}^{3}}(|\nabla w_{\beta}|^{2}+|\nabla \sigma_{\beta}|^{2})dx-\frac{\beta}{4}\int_{\mathbb{R}^{3}}w_{\beta}^{2}\sigma_{\beta}dx\right]+\lim_{\beta \rightarrow 0}J_{\beta}(\widetilde{u},\widetilde{v})\\
&\geq\lim_{\beta\rightarrow 0}\frac{1}{6}\int_{\mathbb{R}^{3}}(|\nabla w_{\beta}|^{2}+|\nabla \sigma_{\beta}|^{2})dx+\lim_{\beta \rightarrow 0}J_{\beta}(\widetilde{u},\widetilde{v})\geq m^{-}_{0}(b_{1},b_{2}).
\end{align*}
Thus, $\lim_{\beta\rightarrow 0}J(\widetilde{u}_{\beta},\widetilde{v}_{\beta})=m^{-}_{0}(b_{1},b_{2})\ \text{and}\ (\widetilde{u}_{\beta},\widetilde{v}_{\beta})\rightarrow (\widetilde{u},\widetilde{v})$ in $H^{1}(\mathbb{R}^{3})\times H^{1}(\mathbb{R}^{3}) .$
\end{proof}
\section{Proof of Theorem \ref{Th2}}\label{sec10}
In this section, we give a refined upper of $m^{+}_{\beta}(b_{1},b_{2})$ and search for $(\lambda_{1,b_{1},b_{2}},\lambda_{2,b_{1},b_{2}}, u_{b_{1},b_{2}},v_{b_{1},b_{2}})$ solving
\begin{equation}\label{LA1}
\begin{cases}

-\Delta u+\lambda_{1}u=\beta uv  & \text{in} \ \ \mathbb{R}^{3},\\

-\Delta v+\lambda_{2}v= \frac{\beta}{2}u^{2}& \text{in} \ \ \mathbb{R}^{3},
\end{cases}
\end{equation}
satisfying the additional conditions
\begin{align}\label{LA2}
\int_{\mathbb{R}^{3}}u^{2}dx=b^{2}_{1}\ \text{and} \ \int_{\mathbb{R}^{3}}v^{2}dx=b^{2}_{2}.
\end{align}

Denote $$
J_{0}(u,v)=\frac{1}{2}\int_{\mathbb{R}^{3}}(|\nabla u|^{2}+|\nabla v|^{2})dx-\frac{\beta}{2}\int_{\mathbb{R}^{3}}u^{2}vdx,
$$
on the constraint $\mathrm{T}_{b_{1}}\times \mathrm{T}_{b_{2}} $.
 \begin{align*}
\mathcal{P}^{0}_{b_{1},b_{2}}(u,v):=\Big\{&(u,v)\in\mathrm{T}_{b_{1}}\times \mathrm{T}_{b_{2}} :P^{0}_{b_{1},b_{2}}(u,v)=0\Big\},
\end{align*}
where $$P^{0}_{b_{1},b_{2}}(u,v)=\int_{\mathbb{R}^{3}}(|\nabla u|^{2}+|\nabla v|^{2})dx-\frac{3}{4}\beta\int_{\mathbb{R}^{3}}u^{2}vdx.$$Then, the solution of \eqref{LA1}-\eqref{LA2} can be found as minimizers of $$m^{+}_{0}(b_{1},b_{2})=\inf_{\mathrm{T}_{b_{1}}\times \mathrm{T}_{b_{2}}}J_{0}(u,v)>-\infty.$$
If $(u_0,v_0)$ is the unique positive solution of \eqref{eqd1}, then $(u_{0},v_{0})=(\sqrt{2}\beta^{-1} w, \beta^{-1}w)$, and $w$ is the unique positive solution of \eqref{eqd2}.




Let$$\theta_{1}=\frac{2\beta^{2}b_{1}^{\frac{6}{5}}b_{2}^{\frac{4}{5}}}{16^{\frac{2}{5}}\|w\|^{2}_{L^{2}(\mathbb{R}^{3})}},\ \theta_{2}=\frac{\beta^{2}b_{1}^{\frac{8}{5}}b_{2}^{\frac{2}{5}}}{16^{\frac{1}{5}}\|w\|^{2}_{L^{2}(\mathbb{R}^{3})}},\ L_{1}=\frac{2\beta^{4}b_{1}^{\frac{14}{5}}b_{2}^{\frac{6}{5}}}{16^{\frac{3}{5}}\|w\|^{4}_{L^{2}(\mathbb{R}^{3})}},\ \ L_{2}=\frac{4\beta^{4}b_{1}^{\frac{12}{5}}b_{2}^{\frac{8}{5}}}{16^{\frac{4}{5}}\|w\|^{4}_{L^{2}(\mathbb{R}^{3})}}, $$
$$\lambda_{1,b_{1},b_{2}}=\frac{4\beta^{4}b_{1}^{\frac{12}{5}}b_{2}^{\frac{8}{5}}}{16^{\frac{4}{5}}\|w\|^{4}_{L^{2}(\mathbb{R}^{3})}},\ \ \ \lambda_{2,b_{1},b_{2}}=\frac{\beta^{4}b_{1}^{\frac{16}{5}}b_{2}^{\frac{4}{5}}}{16^{\frac{2}{5}}\|w\|^{4}_{L^{2}(\mathbb{R}^{3})}} ,$$
we have following lemma.
\begin{lemma}\label{LE}
\eqref{LA1}-\eqref{LA2} has a unique positive solution
$$\big(\lambda_{1,b_{1},b_{2}},\lambda_{2,b_{1},b_{2}}, L_{1}u_{0}(\theta_{1}x), L_{2}v_{0}(\theta_{2}x)\big).$$
Moreover,
\begin{align*}
&m^{+}_{0}(b_{1},b_{2})=J_{0}(L_{1}u_{0}(\theta_{1}x), L_{2}v_{0}(\theta_{2}x))=-\frac{1}{6}\Big[\frac{4\beta^{4}b_{1}^{\frac{22}{5}}b_{2}^{\frac{8}{5}}}{16^{\frac{4}{5}}\|w\|^{6}_{L^{2}(\mathbb{R}^{3})}}+\frac{\beta^{4}b_{1}^{\frac{16}{5}}b_{2}^{\frac{14}{5}}}{16^{\frac{2}{5}}\|w\|^{6}_{L^{2}(\mathbb{R}^{3})}}\Big]\|\nabla w\|^{2}_{L^{2}(\mathbb{R}^{3})}.
\end{align*}
\end{lemma}
\begin{proof}
By elementary calculation, we have  $$(\lambda_{1,b_{1},b_{2}},\lambda_{2,b_{1},b_{2}}, L_{1}u_{0}(\theta_{1}x), L_{2}v_{0}(\theta_{2}x))$$
is the unique positive solution of \eqref{LA1}-\eqref{LA2}. Furthermore, we have
\begin{align*}
m^{+}_{0}(b_{1},b_{2})&=J_{0}(L_{1}u_{0}(\theta_{1}x), L_{2}v_{0}(\theta_{2}x))\\
&=-\frac{1}{6}\Big[\frac{L^{2}_{1}}{\theta_{1}}\int_{\mathbb{R}^{3}}|\nabla u_{0}|^{2}dx+\frac{L^{2}_{2}}{\theta_{2}}\int_{\mathbb{R}^{3}}|\nabla v_{0}|^{2}dx\Big]\\
&=-\frac{1}{6}\Big[2\frac{L^{2}_{1}}{\theta_{1}}\beta^{-2}+\frac{L^{2}_{2}}{\theta_{2}}\beta^{-2}\Big]\|\nabla w\|^{2}_{L^{2}(\mathbb{R}^{3})}\\
&=-\frac{1}{6}\Big[\frac{4\beta^{4}b_{1}^{\frac{22}{5}}b_{2}^{\frac{8}{5}}}{16^{\frac{4}{5}}\|w\|^{6}_{L^{2}(\mathbb{R}^{3})}}+\frac{\beta^{4}b_{1}^{\frac{16}{5}}b_{2}^{\frac{14}{5}}}{16^{\frac{2}{5}}\|w\|^{6}_{L^{2}(\mathbb{R}^{3})}}\Big]\|\nabla w\|^{2}_{L^{2}(\mathbb{R}^{3})}.
\end{align*}
\end{proof}
\begin{lemma}\label{LEM}
We have
  $$m^{-}_{\beta}(b_{1},b_{2})< m^{+}_{0}(b_{1},b_{2})=-\frac{1}{6}\Big[\frac{4\beta^{4}b_{1}^{\frac{22}{5}}b_{2}^{\frac{8}{5}}}{16^{\frac{4}{5}}\|w\|^{6}_{L^{2}(\mathbb{R}^{3})}}+\frac{\beta^{4}b_{1}^{\frac{16}{5}}b_{2}^{\frac{14}{5}}}{16^{\frac{2}{5}}\|w\|^{6}_{L^{2}(\mathbb{R}^{3})}}\Big]\|\nabla w\|^{2}_{L^{2}(\mathbb{R}^{3})}.$$
\end{lemma}
\begin{proof}
Since
\begin{align*}
J_{0}(u,v)&=\frac{1}{2}\int_{\mathbb{R}^{3}}(|\nabla u|^{2}+|\nabla v|^{2})dx-\frac{\beta}{2}\int_{\mathbb{R}^{3}}u^{2}vdx\\
&\geq \frac{1}{2}\int_{\mathbb{R}^{3}}(|\nabla u|^{2}+|\nabla v|^{2})dx-\frac{|\beta|}{2}\Big[\big(\frac{2}{3}b^{\frac{3}{2}}_{1}+\frac{1}{3}b^{\frac{3}{2}}_{2}\big)C^{3}_{3,3}\big[\|\nabla u\|^{2}_{L^{2}(\mathbb{R}^{3})}+\|\nabla v\|^{2}_{L^{2}(\mathbb{R}^{3})}\big]^{\frac{3}{4}}\Big]\\
&=g\Big(\big(\int_{\mathbb{R}^{3}}(|\nabla u|^{2}+|\nabla v|^{2})dx\big)^{\frac{1}{2}}\Big),
\end{align*}
where $$g(t)=\frac{1}{2}t^{2}-\frac{|\beta|}{2}\Big(\frac{2}{3}b^{\frac{3}{2}}_{1}+\frac{1}{3}b^{\frac{3}{2}}_{2}\Big)C^{3}_{3,3}t^{\frac{3}{2}}.$$
It is easy to see that $g(t)<0$ if  $t\in(0,\widetilde{t})$ and $g(t)>0$ if  $t\in(\widetilde{t},+\infty)$, where $\widetilde{t}=\Big[|\beta|\big(\frac{2}{3}b^{\frac{3}{2}}_{1}+\frac{1}{3}b^{\frac{3}{2}}_{2}\big)C^{3}_{3,3}\Big]^{2}$.
Since $J_{0}(L_{1}u_{0}(\theta_{1}x), L_{2}v_{0}(\theta_{2}x))=m^{+}_{0}(b_{1},b_{2})<0$,  we get $$\Big[\int_{\mathbb{R}^{3}}(|\nabla L_{1}u_{0}(\theta_{1}x)|^{2}+|\nabla L_{2}v_{0}(\theta_{2}x)|^{2})dx\Big]^{\frac{1}{2}}<\widetilde{t}.$$
From Lemma \ref{LE}, the definition of $R_{0}$ in Lemma \ref{LAA} and the definition of $\widetilde{t}$, we have $\widetilde{t}<R_{0}$,
$$\|L_{1}u_{0}(\theta_{1}x)\|^{2}_{L^{2}(\mathbb{R}^{3})}=b^{2}_{1},\ \|L_{2}v_{0}(\theta_{2}x)\|^{2}_{L^{2}(\mathbb{R}^{3})}=b^{2}_{2},$$
and
\begin{align*}
&\Big[\int_{\mathbb{R}^{3}}(|\nabla L_{1}u_{0}(\theta_{1}x)|^{2}+|\nabla L_{2}v_{0}(\theta_{2}x)|^{2})dx\Big]^{\frac{1}{2}}<\widetilde{t}<R_{0}.
\end{align*}
Therefore, we have
\begin{align*}
m^{-}_{\beta}(b_{1},b_{2})&=\inf_{A_{R_{0}}}J_{\beta}(u,v)\leq J_{\beta}(L_{1}u_{0}(\theta_{1}x),L_{2}v_{0}(\theta_{2}x))\\
&\quad <J_{0}(L_{1}u_{0}(\theta_{1}x),L_{2}v_{0}(\theta_{2}x))
=m_{0}(b_{1},b_{2})\\
&=-\frac{1}{6}\Big[\frac{4\beta^{4}b_{1}^{\frac{22}{5}}b_{2}^{\frac{8}{5}}}{16^{\frac{4}{5}}\|w\|^{6}_{L^{2}(\mathbb{R}^{3})}}+\frac{\beta^{4}b_{1}^{\frac{16}{5}}b_{2}^{\frac{14}{5}}}{16^{\frac{2}{5}}\|w\|^{6}_{L^{2}(\mathbb{R}^{3})}}\Big]\|\nabla w\|^{2}_{L^{2}(\mathbb{R}^{3})}.
\end{align*}
\end{proof}
\begin{proof}[\bf Proof of Theorem \ref{Th2}]
Let $b_{1,k},\ b_{2,k}\rightarrow 0^{+}$ with $b_{1,k}\backsim b_{2,k}$ as $k\rightarrow +\infty$  and  $(u_{b_{1,k},b_{2,k}},v_{b_{1,k},b_{2,k}})\in A_{R_{0}}$ be a positive minimizer of $m(b_{1,k},b_{2,k}, R_{0})$ \\ for each $k\in\mathbb{N}$. From
\begin{align*}
&P_{b_{1},b_{2}}(u_{b_{1,k},b_{2,k}},v_{b_{1,k},b_{2,k}})\\
&=\int_{\mathbb{R}^{3}}(|\nabla u_{b_{1,k},b_{2,k}}|^{2}+|\nabla v_{b_{1,k},b_{2,k}}|^{2})dx
-\frac{3}{4}\int_{\mathbb{R}^{3}}\left(\mu_{1}u_{b_{1,k},b_{2,k}}^{4}+\mu_{2}v_{b_{1,k},b_{2,k}}^{4}+2\rho u_{b_{1,k},b_{2,k}}^{2}v_{b_{1,k},b_{2,k}}^{2}\right)dx\\
&\quad-\frac{3}{4}\beta\int_{\mathbb{R}^{N}}u_{b_{1,k},b_{2,k}}^{2}v_{b_{1,k},b_{2,k}}dx=0,
\end{align*}
we have
\begin{align*}
J_{\beta}(u_{b_{1,k},b_{2,k}},v_{b_{1,k},b_{2,k}})
&=\frac{1}{6}\int_{\mathbb{R}^{3}}(|\nabla u_{b_{1,k},b_{2,k}}|^{2}+|\nabla v_{b_{1,k},b_{2,k}}|^{2})dx-\frac{\beta}{4}\int_{\mathbb{R}^{N}}u_{b_{1,k},b_{2,k}}^{2}v_{b_{1,k},b_{2,k}}dx\\
&=-\frac{1}{6}\int_{\mathbb{R}^{3}}(|\nabla u_{b_{1,k},b_{2,k}}|^{2}+|\nabla v_{b_{1,k},b_{2,k}}|^{2})dx\\
&\quad+\frac{1}{4}\int_{\mathbb{R}^{3}}\big(\mu_{1}u_{b_{1,k},b_{2,k}}^{4}+\mu_{2}v_{b_{1,k},b_{2,k}}^{4}+2\rho u_{b_{1,k},b_{2,k}}^{2}v_{b_{1,k},b_{2,k}}^{2}\big)dx\\
&<-\frac{1}{6}\Big[\frac{4\beta^{4}b_{1,k}^{\frac{22}{5}}b_{2,k}^{\frac{8}{5}}}{16^{\frac{4}{5}}\|w\|^{6}_{L^{2}(\mathbb{R}^{3})}}+\frac{\beta^{4}b_{1,k}^{\frac{16}{5}}b_{2,k}^{\frac{14}{5}}}{16^{\frac{2}{5}}\|w\|^{6}_{L^{2}(\mathbb{R}^{3})}}\Big]\|\nabla w\|^{2}_{L^{2}(\mathbb{R}^{3})}.
\end{align*}
Thus
\begin{align*}
\Big[\frac{4\beta^{4}b_{1,k}^{\frac{22}{5}}b_{2,k}^{\frac{8}{5}}}{16^{\frac{4}{5}}\|w\|^{6}_{L^{2}(\mathbb{R}^{3})}}+\frac{\beta^{4}b_{1,k}^{\frac{16}{5}}b_{2,k}^{\frac{14}{5}}}{16^{\frac{2}{5}}\|w\|^{6}_{L^{2}(\mathbb{R}^{3})}}\Big]\|\nabla w\|^{2}_{L^{2}(\mathbb{R}^{3})}\leq\int_{\mathbb{R}^{3}}(|\nabla u_{b_{1,k},b_{2,k}}|^{2}+|\nabla v_{b_{1,k},b_{2,k}}|^{2})dx
\end{align*}
and
\begin{align*}
\int_{\mathbb{R}^{3}}(|\nabla u_{b_{1,k},b_{2,k}}|^{2}+|\nabla v_{b_{1,k},b_{2,k}}|^{2})dx<\left[\frac{3|\beta|}{2}\Big[\big(\frac{2}{3}b^{\frac{3}{2}}_{1,k}+\frac{1}{3}b^{\frac{3}{2}}_{2,k}\big)C^{3}_{3,3}\Big]\right]^{4}.
\end{align*}
Therefore
\begin{align}\label{LE2}
&\Big[\frac{4\beta^{4}b_{1,k}^{\frac{22}{5}}b_{2,k}^{\frac{8}{5}}}{16^{\frac{4}{5}}\|w\|^{6}_{L^{2}(\mathbb{R}^{3})}}+\frac{\beta^{4}b_{1,k}^{\frac{16}{5}}b_{2,k}^{\frac{14}{5}}}{16^{\frac{2}{5}}\|w\|^{6}_{L^{2}(\mathbb{R}^{3})}}\Big]\|\nabla w\|^{2}_{L^{2}(\mathbb{R}^{3})}\\\nonumber
&\leq\int_{\mathbb{R}^{3}}(|\nabla u_{b_{1,k},b_{2,k}}|^{2}+|\nabla v_{b_{1,k},b_{2,k}}|^{2})dx<\left[\frac{3|\beta|}{2}\Big[\big(\frac{2}{3}b^{\frac{3}{2}}_{1,k}+\frac{1}{3}b^{\frac{3}{2}}_{2,k}\big)C^{3}_{3,3}\Big]\right]^{4}.
\end{align}
When $b_{1,k}\backsim b_{2,k} $, we  can get
\begin{align}\label{LE22}
\int_{\mathbb{R}^{3}}(|\nabla u_{b_{1,k},b_{2,k}}|^{2}+|\nabla v_{b_{1,k},b_{2,k}}|^{2})dx\backsim b_{1,k}^{\frac{22}{5}}b_{2,k}^{\frac{8}{5}}+b_{1,k}^{\frac{16}{5}}b_{2,k}^{\frac{14}{5}}.
\end{align}
From \eqref{LA11}-\eqref{LA12} and \eqref{LE2}, we have
\begin{align}\label{LE3}
&\frac{1}{4}\int_{\mathbb{R}^{3}}\big(\mu_{1}u_{b_{1,k},b_{2,k}}^{4}+\mu_{2}v_{b_{1,k},b_{2,k}}^{4}+2\rho u_{b_{1,k},b_{2,k}}^{2}v_{b_{1,k},b_{2,k}}^{2}\big)dx\\\nonumber
&\leq\frac{1}{4}\int_{\mathbb{R}^{3}}\Big((\mu_{1}+\rho)u_{b_{1,k},b_{2,k}}^{4}+(\mu_{2}+\rho)v_{b_{1,k},b_{2,k}}^{4}\Big)dx\\\nonumber
&\leq \Big[\frac{\mu_{1}+\rho}{4} C^{4}_{3,4}b_{1,k}+\frac{\mu_{2}+\rho}{4} C^{4}_{3,4}b_{2,k}\Big]\left[\|\nabla u_{b_{1,k},b_{2,k}}\|^{2}_{L^{2}(\mathbb{R}^{3})}+\|\nabla v_{b_{1,k},b_{2,k}}\|^{2}_{L^{2}(\mathbb{R}^{3})}\right]^{\frac{3}{2}}\\\nonumber
&\leq\Big[\frac{\mu_{1}+\rho}{4} C^{4}_{3,4}b_{1,k}+\frac{\mu_{2}+\rho}{4} C^{4}_{3,4}b_{2,k}\Big]\left[\frac{3|\beta|}{2}\Big[\big(\frac{2}{3}b^{\frac{3}{2}}_{1,k}+\frac{1}{3}b^{\frac{3}{2}}_{2,k}\big)C^{3}_{3,3}\Big]\right]^{6}\rightarrow 0
\end{align}
as $k\rightarrow +\infty.$
Since $$m^{-}_{\beta}(b_{1,k},b_{2,k})<-\frac{1}{6}\Big[\frac{4\beta^{4}b_{1,k}^{\frac{22}{5}}b_{2,k}^{\frac{8}{5}}}{16^{\frac{4}{5}}\|w\|^{6}_{L^{2}(\mathbb{R}^{3})}}+\frac{\beta^{4}b_{1,k}^{\frac{16}{5}}b_{2,k}^{\frac{14}{5}}}{16^{\frac{2}{5}}\|w\|^{6}_{L^{2}(\mathbb{R}^{3})}}\Big]\|\nabla w\|^{2}_{L^{2}(\mathbb{R}^{3})},$$
we obtain
\begin{align}
&-\frac{1}{6}\Big[\frac{4\beta^{4}b_{1,k}^{\frac{22}{5}}b_{2,k}^{\frac{8}{5}}}{16^{\frac{4}{5}}\|w\|^{6}_{L^{2}(\mathbb{R}^{3})}}+\frac{\beta^{4}b_{1,k}^{\frac{16}{5}}b_{2,k}^{\frac{14}{5}}}{16^{\frac{2}{5}}\|w\|^{6}_{L^{2}(\mathbb{R}^{3})}}\Big]\|\nabla w\|^{2}_{L^{2}(\mathbb{R}^{3})}\\\nonumber
&>m^{-}_{\beta}(b_{1,k},b_{2,k})=J_{\beta}(u_{b_{1,k},b_{2,k}},v_{b_{1,k},b_{2,k}})\\\nonumber
&\geq \inf_{\mathrm{T}_{b_{1,k}}\times \mathrm{T}_{b_{2,k}}}\Big\{\frac{1}{2}\int_{\mathbb{R}^{3}}(|\nabla u|^{2}+|\nabla v|^{2})dx-\frac{\beta}{2}\int_{\mathbb{R}^{3}}u^{2}vdx\Big\}\\\nonumber
&\quad-\frac{1}{4}\int_{\mathbb{R}^{3}}\big(\mu_{1}u_{b_{1,k},b_{2,k}}^{4}+\mu_{2}v_{b_{1,k},b_{2,k}}^{4}+2\rho u_{b_{1,k},b_{2,k}}^{2}v_{b_{1,k},b_{2,k}}^{2}\big)dx\\\nonumber
&=-\frac{1}{6}\Big[\frac{4\beta^{4}b_{1,k}^{\frac{22}{5}}b_{2,k}^{\frac{8}{5}}}{16^{\frac{4}{5}}\|w\|^{6}_{L^{2}(\mathbb{R}^{3})}}+\frac{\beta^{4}b_{1,k}^{\frac{16}{5}}b_{2,k}^{\frac{14}{5}}}{16^{\frac{2}{5}}\|w\|^{6}_{L^{2}(\mathbb{R}^{3})}}\Big]\|\nabla w\|^{2}_{L^{2}(\mathbb{R}^{3})}\\\nonumber
&\quad-\frac{1}{4}\int_{\mathbb{R}^{3}}\big(\mu_{1}u_{b_{1,k},b_{2,k}}^{4}+\mu_{2}v_{b_{1,k},b_{2,k}}^{4}+2\rho u_{b_{1,k},b_{2,k}}^{2}v_{b_{1,k},b_{2,k}}^{2}\big)dx\\
&\geq-\frac{1}{6}\Big[\frac{4\beta^{4}b_{1,k}^{\frac{22}{5}}b_{2,k}^{\frac{8}{5}}}{16^{\frac{4}{5}}\|w\|^{6}_{L^{2}(\mathbb{R}^{3})}}+\frac{\beta^{4}b_{1,k}^{\frac{16}{5}}b_{2,k}^{\frac{14}{5}}}{16^{\frac{2}{5}}\|w\|^{6}_{L^{2}(\mathbb{R}^{3})}}\Big]\|\nabla w\|^{2}_{L^{2}(\mathbb{R}^{3})}\\\nonumber
&\quad-\Big[\frac{\mu_{1}+\rho}{4} C^{4}_{3,4}b_{1,k}+\frac{\mu_{2}+\rho}{4} C^{4}_{3,4}b_{2,k}\Big]\left[\frac{3|\beta|}{2}\Big[\big(\frac{2}{3}b^{\frac{3}{2}}_{1,k}+\frac{1}{3}b^{\frac{3}{2}}_{2,k}\big)C^{3}_{3,3}\Big]\right]^{6}.
\end{align}
When $b_{1,k}\backsim b_{2,k} $, we  can get
\begin{align}\label{LE6}
m^{-}_{\beta}(b_{1,k},b_{2,k})\backsim b_{1,k}^{\frac{22}{5}}b_{2,k}^{\frac{8}{5}}+b_{1,k}^{\frac{16}{5}}b_{2,k}^{\frac{14}{5}}
\end{align}
and
\begin{align}\label{LE7}
\frac{1}{2}\int_{\mathbb{R}^{3}}(|\nabla u_{b_{1,k},b_{2,k}}|^{2}+|\nabla v_{b_{1,k},b_{2,k}}|^{2})dx-\frac{\beta}{2}\int_{\mathbb{R}^{N}}u_{b_{1,k},b_{2,k}}^{2}v_{b_{1,k},b_{2,k}}dx\backsim b_{1,k}^{\frac{22}{5}}b_{2,k}^{\frac{8}{5}}+b_{1,k}^{\frac{16}{5}}b_{2,k}^{\frac{14}{5}}.
\end{align}
The Lagrange multipliers rule implies the existence of some $\lambda_{1,k},\lambda_{2,k}\in \mathbb{R}$ such that
\begin{align*}
&\int_{\mathbb{R}^{3}}(\nabla u_{b_{1,k},b_{2,k}}\nabla \varphi) dx +\int_{\mathbb{R}^{3}}(\lambda _{1,k}u_{b_{1,k},b_{2,k}}\varphi)dx\\\nonumber
&=\int_{\mathbb{R}^{3}}(\mu_{1}u_{b_{1,k},b_{2,k}}^{3}\varphi+\rho v_{b_{1,k},b_{2,k}}^{2}u_{b_{1,k},b_{2,k}}\varphi)dx +\beta\int_{\mathbb{R}^{3}} u_{b_{1,k},b_{2,k}} v_{b_{1,k},b_{2,k}}\varphi,
\end{align*}
\begin{align*}
&\int_{\mathbb{R}^{3}}(\nabla v_{b_{1,k},b_{2,k}}\nabla \psi) dx +\int_{\mathbb{R}^{3}}(\lambda_{2,k}v_{b_{1,k},b_{2,k}}\psi)dx\\\nonumber
&=\int_{\mathbb{R}^{3}}(\mu_{2}v_{b_{1,k},b_{2,k}}^{3}\psi+\rho u_{b_{1,k},b_{2,k}}^{2}v_{b_{1,k},b_{2,k}}\psi)dx +\frac{\beta}{2}\int_{\mathbb{R}^{3}} u_{b_{1,k},b_{2,k}}^{2}\psi,
\end{align*}
for each $\varphi,\psi \in H^{1}(\mathbb{R}^{3})$. Taking $\varphi=u_{b_{1,k},b_{2,k}}$ and $\psi=v_{b_{1,k},b_{2,k}}$, we have
\begin{align}\label{AH}
\lambda _{1,k}b_{1,k}^{2}&=-\int_{\mathbb{R}^{3}}|\nabla u_{b_{1,k},b_{2,k}}|^{2} dx +
\int_{\mathbb{R}^{3}}(\mu_{1}u_{b_{1,k},b_{2,k}}^{4}+\rho v_{b_{1,k},b_{2,k}}^{2}u^{2}_{b_{1,k},b_{2,k}})dx\\\nonumber
 &+\beta\int_{\mathbb{R}^{3}} u^{2}_{b_{1,k},b_{2,k}} v_{b_{1,k},b_{2,k}},
 \end{align}
 \begin{align}\label{AG}
\lambda _{2,k}b_{2,k}^{2}&=-\int_{\mathbb{R}^{3}}|\nabla v_{b_{1,k},b_{2,k}}|^{2} dx +
\int_{\mathbb{R}^{3}}(\mu_{1}v_{b_{1,k},b_{2,k}}^{4}+\rho v_{b_{1,k},b_{2,k}}^{2}u^{2}_{b_{1,k},b_{2,k}})dx\\\nonumber
 &+\frac{\beta}{2}\int_{\mathbb{R}^{3}} u^{2}_{b_{1,k},b_{2,k}} v_{b_{1,k},b_{2,k}}.
 \end{align}
Since $P_{b_{1},b_{2}}(u_{b_{1,k},b_{2,k}},v_{b_{1,k},b_{2,k}})=0,$
we get
\begin{align}\label{LE8}
&\int_{\mathbb{R}^{3}}(|\nabla u_{b_{1,k},b_{2,k}}|^{2}+|\nabla v_{b_{1,k},b_{2,k}}|^{2})dx-\frac{3}{4}\beta\int_{\mathbb{R}^{N}}u_{b_{1,k},b_{2,k}}^{2}v_{b_{1,k},b_{2,k}}dx\\\nonumber
&-\frac{3}{4}\int_{\mathbb{R}^{3}}\left(\mu_{1}u_{b_{1,k},b_{2,k}}^{4}+\mu_{2}v_{b_{1,k},b_{2,k}}^{4}+2\rho u_{b_{1,k},b_{2,k}}^{2}v_{b_{1,k},b_{2,k}}^{2}\right)dx= 0.
\end{align}
Hence, from   \eqref{AH},  \eqref{AG}, \eqref{LE8} and $b_{1,k}\backsim b_{2,k}$, we have
\begin{align*}
&-\int_{\mathbb{R}^{3}}(|\nabla u_{b_{1,k},b_{2,k}}|^{2}+|\nabla v_{b_{1,k},b_{2,k}}|^{2} ) dx+\frac{3\beta}{2}\int_{\mathbb{R}^{3}} u^{2}_{b_{1,k},b_{2,k}} v_{b_{1,k},b_{2,k}}\\
&\leq\lambda _{1,k}b_{1,k}^{2}+\lambda _{2,k}b_{2,k}^{2}
=\int_{\mathbb{R}^{3}}(|\nabla u_{b_{1,k},b_{2,k}}|^{2}+|\nabla v_{b_{1,k},b_{2,k}}|^{2} ) dx\\
&\quad-\frac{1}{2}\int_{\mathbb{R}^{3}}\left(\mu_{1}u_{b_{1,k},b_{2,k}}^{4}+\mu_{2}v_{b_{1,k},b_{2,k}}^{4}+2\rho u_{b_{1,k},b_{2,k}}^{2}v_{b_{1,k},b_{2,k}}^{2}\right)dx,
\end{align*}
combing with  \eqref{LE22} and \eqref{LE7}, we get $$\lambda _{1,k}b_{1,k}^{2}+\lambda _{2,k}b_{2,k}^{2}\backsim b_{1,k}^{\frac{22}{5}}b_{2,k}^{\frac{8}{5}}+b_{1,k}^{\frac{16}{5}}b_{2,k}^{\frac{14}{5}},$$
thus $\lambda _{1,k}\backsim b_{1,k}^{4}$ and $\lambda _{2,k}\backsim b_{1,k}^{4}$ when $b_{1,k}\backsim b_{2,k}$.

Denote
$$\theta_{1,k}=\frac{2\beta^{2}b_{1,k}^{\frac{6}{5}}b_{2,k}^{\frac{4}{5}}}{16^{\frac{2}{5}}\|w\|^{2}_{L^{2}(\mathbb{R}^{3})}},\ \theta_{2,k}=\frac{\beta^{2}b_{1,k}^{\frac{8}{5}}b_{2,k}^{\frac{2}{5}}}{16^{\frac{1}{5}}\|w\|^{2}_{L^{2}(\mathbb{R}^{3})}},\ L_{1,k}=\frac{2\beta^{4}b_{1,k}^{\frac{14}{5}}b_{2,k}^{\frac{6}{5}}}{16^{\frac{3}{5}}\|w\|^{4}_{L^{2}(\mathbb{R}^{3})}},\ \ L_{2,k}=\frac{4\beta^{4}b_{1,k}^{\frac{12}{5}}b_{2,k}^{\frac{8}{5}}}{16^{\frac{4}{5}}\|w\|^{4}_{L^{2}(\mathbb{R}^{3})}}. $$
Define $$\widetilde{u}_{b_{1,k},b_{2,k}}=L^{-1}_{1,k}u_{b_{1,k},b_{2,k}}(\theta^{-1}_{1,k}x),\ \ \widetilde{v}_{b_{1,k},b_{2,k}}=L^{-1}_{2,k}v_{b_{1,k},b_{2,k}}(\theta^{-1}_{2,k}x),$$
then
\begin{align}\label{LE4}
&\int_{\mathbb{R}^{3}}(|\nabla \widetilde{u}_{b_{1,k},b_{2,k}}|^{2}+|\nabla \widetilde{v}_{b_{1,k},b_{2,k}}|^{2})dx\\\nonumber
&=\frac{\theta_{1,k}}{L^{2}_{1,k}}\int_{\mathbb{R}^{3}}|\nabla u_{b_{1,k},b_{2,k}}|^{2}dx+\frac{\theta_{2,k}}{L^{2}_{2,k}}\int_{\mathbb{R}^{3}}|\nabla v_{b_{1,k},b_{2,k}}|^{2}dx\\\nonumber
&\leq \Big[\frac{16^{\frac{4}{5}}\|w\|^{6}_{L^{2}(\mathbb{R}^{3})}}{2\beta^{6}b_{1,k}^{\frac{22}{5}}b_{2,k}^{\frac{8}{5}}}+\frac{16^{\frac{2}{5}}\|w\|^{6}_{L^{2}(\mathbb{R}^{3})}}{2\beta^{6}b_{1,k}^{\frac{16}{5}}b_{2,k}^{\frac{14}{5}}}\Big]\int_{\mathbb{R}^{3}}(|\nabla u_{b_{1,k},b_{2,k}}|^{2}+|\nabla v_{b_{1,k},b_{2,k}}|^{2})dx\\\nonumber
&\leq \Big[\frac{16^{\frac{4}{5}}\|w\|^{6}_{L^{2}(\mathbb{R}^{3})}}{2\beta^{6}b_{1,k}^{\frac{22}{5}}b_{2,k}^{\frac{8}{5}}}+\frac{16^{\frac{2}{5}}\|w\|^{6}_{L^{2}(\mathbb{R}^{3})}}{2\beta^{6}b_{1,k}^{\frac{16}{5}}b_{2,k}^{\frac{14}{5}}}\Big]\left[\frac{3|\beta|}{2}\Big[\big(\frac{2}{3}b^{\frac{3}{2}}_{1,k}+\frac{1}{3}b^{\frac{3}{2}}_{2,k}\big)C^{3}_{3,3}\Big]\right]^{4},
\end{align}
\begin{align}\label{LE5}
\int_{\mathbb{R}^{3}}(| \widetilde{u}_{b_{1,k},b_{2,k}}|^{2}+| \widetilde{v}_{b_{1,k},b_{2,k}}|^{2})dx&=\frac{\theta_{1,k}^{3}}{L^{2}_{1,k}}\int_{\mathbb{R}^{3}}| u_{b_{1,k},b_{2,k}}|^{2}dx+\frac{\theta_{2,k}^{3}}{L^{2}_{2,k}}\int_{\mathbb{R}^{3}}| v_{b_{1,k},b_{2,k}}|^{2}dx\\\nonumber
&=\frac{2\|w\|^{2}_{L^{2}(\mathbb{R}^{3})}+\|w\|^{2}_{L^{2}(\mathbb{R}^{3})}}{\beta^{2}},
\end{align}
\begin{align}\label{LE9}
&\int_{\mathbb{R}^{3}}\Big(\frac{\theta_{1,k}^{3}}{L^{4}_{1,k}}\mu_{1}u_{b_{1,k},b_{2,k}}^{4}+\mu_{2}\frac{\theta_{2,k}^{3}}{L^{4}_{2,k}}v_{b_{1,k},b_{2,k}}^{4}\Big)dx\\\nonumber
&\quad<\int_{\mathbb{R}^{3}}\left(\mu_{1}\widetilde{u}_{b_{1,k},b_{2,k}}^{4}+\mu_{2}\widetilde{v}_{b_{1,k},b_{2,k}}^{4}+2\rho \widetilde{u}_{b_{1,k},b_{2,k}}^{2}\widetilde{v}_{b_{1,k},b_{2,k}}^{2}\right)dx\\\nonumber
&=\int_{\mathbb{R}^{3}}\Big(\frac{\theta_{1,k}^{3}}{L^{4}_{1,k}}\mu_{1}u_{b_{1,k},b_{2,k}}^{4}+\mu_{2}\frac{\theta_{2,k}^{3}}{L^{4}_{2,k}}v_{b_{1,k},b_{2,k}}^{4}+2\rho u_{b_{1,k},b_{2,k}}^{2}v_{b_{1,k},b_{2,k}}^{2}\Big)dx\\\nonumber
&\quad\leq\int_{\mathbb{R}^{3}}\Big(\frac{\theta_{1,k}^{3}}{L^{4}_{1,k}}(\mu_{1}+\rho)u_{b_{1,k},b_{2,k}}^{4}+(\mu_{2}+\rho)\frac{\theta_{2,k}^{3}}{L^{4}_{2,k}}v_{b_{1,k},b_{2,k}}^{4}\Big)dx.
\end{align}
By the definition of $\theta_{1,k},\  \theta_{2,k},\  L_{1,k},\ L_{2,k}$, it is easy to see that
\begin{align}\label{LE10}
\int_{\mathbb{R}^{3}}\big(\mu_{1}\widetilde{u}_{b_{1,k},b_{2,k}}^{4}+\mu_{2}\widetilde{v}_{b_{1,k},b_{2,k}}^{4}+2\rho \widetilde{u}_{b_{1,k},b_{2,k}}^{2}\widetilde{v}_{b_{1,k},b_{2,k}}^{2}\big)dx\rightarrow+\infty\ \text{as}\ k\rightarrow+\infty.
\end{align}
From  \eqref{LE4}- \eqref{LE5}, we know that $(\widetilde{u}_{b_{1,k},b_{2,k}}, \widetilde{v}_{b_{1,k},b_{2,k}})$ is bounded in $H^{1}(\mathbb{R}^{3})\times H^{1}(\mathbb{R}^{3}).$ Then, we have

\begin{equation*}
\big(\widetilde{u}_{b_{1,k},b_{2,k}}(x),\widetilde{v}_{b_{1,k},b_{2,k}}(x)\big)\rightharpoonup (\bar{u}, \bar{v})\neq(0,0),
\end{equation*}
for some  $(\bar{u}, \bar{v})\in H^{1}(\mathbb{R}^{3})\times H^{1}(\mathbb{R}^{3})$.
Thus, we see that $(\widetilde{u}_{b_{1,k},b_{2,k}}(x),\widetilde{v}_{b_{1,k},b_{2,k}}(x))$ satisfies
\begin{equation}\label{LE12}
\begin{cases}

-\Delta \widetilde{u}_{b_{1,k},b_{2,k}}+\frac{\lambda_{1,k}}{\theta^{2}_{1,k}}\widetilde{u}_{b_{1,k},b_{2,k}}=\mu_{1}\frac{L^{2}_{1,k}}{\theta^{2}_{1,k}}\widetilde{u}_{b_{1,k},b_{2,k}}^{3}+\rho\frac{L^{2}_{2,k}}{\theta^{2}_{1,k}} \widetilde{v}_{b_{1,k},b_{2,k}}^{2}\widetilde{u}_{b_{1,k},b_{2,k}}+\beta \frac{L_{2,k}}{\theta^{2}_{1,k}}\widetilde{u}_{b_{1,k},b_{2,k}}\widetilde{v}_{b_{1,k},b_{2,k}}, \\

-\Delta \widetilde{v}_{b_{1,k},b_{2,k}}+\frac{\lambda_{2,k}}{\theta^{2}_{2,k}}\widetilde{v}_{b_{1,k},b_{2,k}}= \mu_{2}\frac{L^{2}_{2,k}}{\theta^{2}_{2,k}}\widetilde{v}_{b_{1,k},b_{2,k}}^{3}+\rho \frac{L^{2}_{1,k}}{\theta^{2}_{2,k}} \widetilde{u}_{b_{1,k},b_{2,k}}^{2}\widetilde{v}_{b_{1,k},b_{2,k}}+\frac{\beta}{2}\frac{L^{2}_{1,k}}{L_{2,k}\theta^{2}_{2,k}} \widetilde{u}_{b_{1,k},b_{2,k}}^{2}.
\end{cases}
\end{equation}
By the definition of $ \theta_{1,k},\  \theta_{2,k},\  L_{1,k},\ L_{2,k}$, we have  $$ \frac{L^{2}_{1,k}}{\theta^{2}_{1,k}}=\frac{\beta^{4}b_{1,k}^{\frac{16}{5}}b_{2,k}^{\frac{4}{5}}}{16^{\frac{2}{5}}\|w\|^{4}_{L^{2}(\mathbb{R}^{3})}}\rightarrow 0,\ \frac{L^{2}_{2,k}}{\theta^{2}_{2,k}}=\frac{\beta^{4}b_{1,k}^{\frac{8}{5}}b_{2,k}^{\frac{12}{5}}}{16^{\frac{1}{5}}\|w\|^{4}_{L^{2}(\mathbb{R}^{3})}}\rightarrow 0,$$
$$\frac{L^{2}_{2,k}}{\theta^{2}_{1,k}}=\frac{4\beta^{4}b_{1,k}^{\frac{12}{5}}b_{2,k}^{\frac{8}{5}}}{16^{\frac{4}{5}}\|w\|^{4}_{L^{2}(\mathbb{R}^{3})}}\rightarrow 0,\ \ \frac{L^{2}_{1,k}}{\theta^{2}_{2,k}}=\frac{4\beta^{4}b_{1,k}^{\frac{12}{5}}b_{2,k}^{\frac{8}{5}}}{16^{\frac{4}{5}}\|w\|^{4}_{L^{2}(\mathbb{R}^{3})}}\rightarrow 0,\ \frac{L_{2,k}}{\theta^{2}_{1,k}}=1,\ \ \frac{L^{2}_{1,k}}{L_{2,k}\theta^{2}_{2,k}}=1,$$
and there exists $\lambda^{\ast}_{1}>0$ and  $\lambda^{\ast}_{2}>0$ such that
$$\frac{\lambda_{1,k}}{\theta^{2}_{1,k}}\rightarrow \lambda^{\ast}_{1},\  \frac{\lambda_{2,k}}{\theta^{2}_{2,k}}\rightarrow \lambda^{\ast}_{2} \ \text{as}\ k\rightarrow+\infty. $$
Therefore, $(\bar{u}, \bar{v})$ solves that
\begin{equation}\label{LE13}
\begin{cases}

-\Delta u+\lambda^{\ast}_{1}u= \beta uv  & \text{in} \ \ \mathbb{R}^{3},\\

-\Delta v+\lambda^{\ast}_{2}v= \frac{\beta}{2}u^{2}& \text{in} \ \ \mathbb{R}^{3}.
\end{cases}
\end{equation}
From Theorem 1.2 of \cite{ZZS15}, we know that when $ \lambda^{\ast}_{1}=\lambda^{\ast}_{2}$, then $$\big(\sqrt{2}\lambda^{\ast}_{1}\beta^{-1} w((\lambda^{\ast}_{1})^{\frac{1}{2}}x), \beta^{-1}\lambda^{\ast}_{1}w((\lambda^{\ast}_{1})^{\frac{1}{2}}x)\big)$$ is the unique positive solution of above system, where $w$ is the unique positive solution of $$-\Delta u +u =u^{2}, \ u \in H^{1}(\mathbb{R}^{3}).$$
Texting \eqref{LE12} and \eqref{LE13} with $\widetilde{u}_{b_{1,k},b_{2,k}}-\bar{u}$, $\widetilde{v}_{b_{1,k},b_{2,k}}-\bar{v}$ respectively, we obtain
\begin{align}\label{LU}
\int_{\mathbb{R}^{3}}|\nabla &(\widetilde{u}_{b_{1,k},b_{2,k}}-\bar{u})|^{2}dx+\int_{\mathbb{R}^{3}}(\frac{\lambda_{1,k}}{\theta^{2}_{1,k}}\widetilde{u}_{b_{1,k},b_{2,k}}-\lambda^{\ast}_{1}\bar{u})(\widetilde{u}_{b_{1,k},b_{2,k}}-\bar{u})\\\nonumber
&-\beta\int_{\mathbb{R}^{3}}(\widetilde{u}_{b_{1,k},b_{2,k}}\widetilde{v}_{b_{1,k},b_{2,k}}-\bar{u}\bar{v})(\widetilde{u}_{b_{1,k},b_{2,k}}-\bar{u})\\\nonumber
&=\int_{\mathbb{R}^{3}}|\nabla (\widetilde{u}_{b_{1,k},b_{2,k}}-\bar{u})|^{2}dx+\lambda^{\ast}_{1}\int_{\mathbb{R}^{3}}|\widetilde{u}_{b_{1,k},b_{2,k}}-\bar{u}|^{2}=o_{k}(1),
\end{align}
\begin{align}\label{LUU}
\int_{\mathbb{R}^{3}}|\nabla &(\widetilde{v}_{b_{1,k},b_{2,k}}-\bar{v})|^{2}dx+\int_{\mathbb{R}^{3}}(\frac{\lambda_{2,k}}{\theta^{2}_{2,k}}\widetilde{v}_{b_{1,k},b_{2,k}}-\lambda^{\ast}_{2}\bar{v})(\widetilde{v}_{b_{1,k},b_{2,k}}-\bar{v})\\\nonumber
&-\frac{\beta}{2}\int_{\mathbb{R}^{3}}(\widetilde{u}^{2}_{b_{1,k},b_{2,k}}-\bar{u}^{2})(\widetilde{v}_{b_{1,k},b_{2,k}}-\bar{v})\\\nonumber
&=\int_{\mathbb{R}^{3}}|\nabla (\widetilde{v}_{b_{1,k},b_{2,k}}-\bar{v})|^{2}dx+\lambda^{\ast}_{2}\int_{\mathbb{R}^{3}}|\widetilde{v}_{b_{1,k},b_{2,k}}-\bar{v}|^{2}=o_{k}(1).
\end{align}
Therefore
\begin{align*}
&\frac{2(\lambda^{\ast}_{1})^{\frac{1}{2}}\|w\|^{2}_{L^{2}(\mathbb{R}^{3})}}{\beta^{2}}=\|\sqrt{2}\lambda^{\ast}_{1}\beta^{-1} w((\lambda^{\ast}_{1})^{\frac{1}{2}}x)\|^{2}_{L^{2}(\mathbb{R}^{3})}=\|\bar{u}\|^{2}_{L^{2}(\mathbb{R}^{3})}\\
&=\lim_{k\rightarrow +\infty}\int_{\mathbb{R}^{3}}| \widetilde{u}_{b_{1,k},b_{2,k}}|^{2}dx=\lim_{k\rightarrow +\infty}\frac{\theta_{1,k}^{3}}{L^{2}_{1,k}}\int_{\mathbb{R}^{3}}| u_{b_{1,k},b_{2,k}}|^{2}dx=\frac{2\|w\|^{2}_{L^{2}(\mathbb{R}^{3})}}{\beta^{2}},
\end{align*}
\begin{align*}
&\frac{(\lambda^{\ast}_{1})^{\frac{1}{2}}\|w\|^{2}_{L^{2}(\mathbb{R}^{3})}}{\beta^{2}}=\|\beta^{-1}\lambda^{\ast}_{1}w((\lambda^{\ast}_{1})^{\frac{1}{2}}x))\|^{2}_{L^{2}(\mathbb{R}^{3})}=\|\bar{v}\|^{2}_{L^{2}(\mathbb{R}^{3})}\\
&= \lim_{k\rightarrow +\infty}\int_{\mathbb{R}^{3}}| \widetilde{v}_{b_{1,k},b_{2,k}}|^{2}dx=\lim_{k\rightarrow +\infty}\frac{\theta_{1,k}^{3}}{L^{2}_{1,k}}\int_{\mathbb{R}^{3}}| v_{b_{1,k},b_{2,k}}|^{2}dx=\frac{\|w\|^{2}_{L^{2}(\mathbb{R}^{3})}}{\beta^{2}},
\end{align*}
therefore $$\lambda^{\ast}_{1}=\lambda^{\ast}_{2}=1.$$
From \eqref{LE12}, \eqref{LE13} ,\eqref{LU} and \eqref{LUU}, we have that $$(\widetilde{u}_{b_{1,k},b_{2,k}},\widetilde{v}_{b_{1,k},b_{2,k}})\rightarrow (\sqrt{2}\beta^{-1} w,\beta^{-1} w)\ \ \text{in}\ H^{1}(\mathbb{R}^{3})\times H^{1}(\mathbb{R}^{3}).$$
Moreover, as the limit function $(\sqrt{2}\beta^{-1} w,\beta^{-1} w)$ is independent of the sequence that we choose, which implies that the convergence is true for the whole sequence. Therefore
\begin{equation*}
\big(\widetilde{u}_{b_{1,k},b_{2,k}},\widetilde{v}_{b_{1,k},b_{2,k}}\big)=\big(L^{-1}_{1,k}u_{b_{1,k},b_{2,k}}(\theta^{-1}_{1,k}(x)),L^{-1}_{2,k}v_{b_{1,k},b_{2,k}}(\theta^{-1}_{2,k}(x))\big)
\rightarrow\big(\sqrt{2}\beta^{-1} w,\beta^{-1} w\big),
\end{equation*}
in $\ H^{1}(\mathbb{R}^{3})\times H^{1}(\mathbb{R}^{3})$ as $b_{1,k},b_{2,k}\to 0$ and $b_{1,k}\backsim b_{2,k}$.

\end{proof}
Next, we prove Theorem \ref{Th5}.
In order to prove Theorem \ref{Th5}, we first give following lemma.
\begin{lemma}\label{LK}
Let $\overline{b}_{1},\ \overline{b}_{2},\ \tau_{1},\ \tau_{2} >0.$ There exists $\widetilde{\beta}=\widetilde{\beta}(\overline{b}_{1}+\tau_{1},\overline{b}_{2}+\tau_{2})>0$ such that  if $0<b_{1}\leq \overline{b}_{1}$, $0<b_{2}\leq \overline{b}_{2}$ and $0<\beta<\widetilde{\beta}$, then

(i) $2R^{2}_{0}(\overline{b}_{1}+\tau_{1},\overline{b}_{2}+\tau_{2},\rho,\beta)<R^{2}_{1}(\overline{b}_{1},\overline{b}_{2},\rho,\beta)$,

(ii) for any  $d_{i}> 0,\  c_{i}>0, i=1,2 $  such that $d^{2}_{i}+c^{2}_{i}=b_{1}^{i},\ i=1,2$, then $R^{2}_{0}(c_{1},c_{2},\rho,\beta)+ R^{2}_{0} (d_{1},d_{2},\rho,\beta)\leq R^{2}_{1}(b_{1}, b_{2},\rho,\beta).$
\end{lemma}
\begin{proof}
By Lemma \ref{LAA}, $0<R_{0}=R_{0}(b_{1},b_{2},\rho,\beta)<R_{1}=R_{1}(b_{1},b_{2},\rho,\beta)$ are the roots of $$g(t,b_{1},b_{2},\rho,\beta)=\frac{1}{2}t^{\frac{1}{2}}-\frac{1}{4}(\mathcal{D}_{1}+\mathcal{D}_{2}+\rho\mathcal{D}_{3})t^{\frac{3}{2}}-\frac{1}{2}|\beta|\mathcal{D}_{4}=\varphi(t,b_{1},b_{2},\rho)-\frac{1}{2}|\beta|\mathcal{D}_{4}$$
and the existence of $R_{0}$ and $R_{1}$ is guaranteed by the condition
\begin{align}\label{LG}
\beta\left(2b^{\frac{3}{2}}_{1}+b^{\frac{3}{2}}_{2}\right)C^{3}_{3,3}C^{2}_{3,4}\sqrt{\mu_{1}b_{1}+\mu_{2}b_{2}+\rho b^{\frac{1}{2}}_{1}b^{\frac{1}{2}}_{2}}<\frac{2\sqrt{6}}{3}.
\end{align}
Let $\overline{b}_{1},\ \overline{b}_{2},\ \tau_{1},\ \tau_{2} >0$ and consider the range of $\beta>0$ such that \eqref{LG} is satisfied with $b_{1}=\overline{b}_{1}+\tau_{1},\ b_{2}=\overline{b}_{2}+\tau_{2} $. Taking the limit as $\beta\rightarrow 0^{+}$, by the continuity we have that $R_{0}(\overline{b}_{1}+\tau_{1},\overline{b}_{2}+\tau_{2},\rho,\beta)\rightarrow0$ and $R_{1}(\overline{b}_{1}+\tau_{1},\overline{b}_{2}+\tau_{2},\rho,\beta)$ converge to the only positive root of $\varphi(t,\overline{b}_{1}+\tau_{1},\overline{b}_{2}+\tau_{2},\rho)$. Particularity, for every $\overline{b}_{1},\ \overline{b}_{2},\ \tau_{1},\ \tau_{2} >0$ fixed , there exists $\widetilde{\beta}=\widetilde{\beta}(\overline{b}_{1}+\tau_{1},\overline{b}_{2}+\tau_{2})>0$ such that
\begin{align}\label{LJ}
2R^{2}_{0}(\overline{b}_{1}+\tau_{1},\overline{b}_{2}+\tau_{2},\rho,\beta)<R^{2}_{1}(\overline{b}_{1}+\tau_{1},\overline{b}_{2}+\tau_{2},\rho,\beta)\ \ \text{whenever} \
0<\beta<\widetilde{\beta}.
\end{align}
Let $0<b_{1}\leq \overline{b}_{1}+\tau_{1},$ $0<b_{2}\leq \overline{b}_{2}+\tau_{2}$ and $0<\beta<\widetilde{\beta}$. Under the condition of \eqref{LG}, we have that
$$\partial_{t}g(t,b_{1},b_{2},\rho,\beta)=\partial_{t}\varphi(t,b_{1},b_{2},\rho).$$ It is easy to check that $\varphi(t,b_{1},b_{2},\rho)$ has a unique critical point on $(0,+\infty)$, which is a strict maximum point in $\overline{t}=\overline{t}(b_{1},b_{2},\rho)$ with $0<R_{0}(b_{1},b_{2},\rho)<\overline{t}<R_{1}(b_{1},b_{2},\rho)$, therefore $$\partial_{t}g(R_{0}(b_{1},b_{2},\rho),b_{1},b_{2},\rho,\beta)=\partial_{t}\varphi(R_{0}(b_{1},b_{2},\rho),b_{1},b_{2},\rho)>0.$$ By the implicit function theorem, we know that $R_{0}(b_{1},b_{2},\rho,\beta)$ is a locally unique $C^{1}$ function of $((b_{1},b_{2},\rho,\beta))$ with $$\frac{\partial_{t}R_{0}(b_{1},b_{2},\rho,\beta)}{\partial b_{1}}=-\frac{\partial_{b_{1}} g(R_{0}(b_{1},b_{2},\rho,\beta),b_{1},b_{2},\rho,\beta)}{\partial_{t}g(R_{0}(b_{1},b_{2},\rho),b_{1},b_{2},\rho,\beta)}>0,$$
$$\frac{\partial_{t}R_{0}(b_{1},b_{2},\rho,\beta)}{\partial b_{2}}=-\frac{\partial_{b_{2}} g(R_{0}(b_{1},b_{2},\rho),b_{1},b_{2},\rho,\beta)}{\partial_{t}g(R_{0}(b_{1},b_{2},\rho,\beta),b_{1},b_{2},\rho,\beta)}>0.$$
Similarly, we can proof that $R_{1}(b_{1},b_{2},\rho,\beta)$ is a locally unique $C^{1}$ function of\\ $((b_{1},b_{2},\rho,\beta))$ with $$\frac{\partial_{t}R_{1}(b_{1},b_{2},\rho,\beta)}{\partial b_{1}}=-\frac{\partial_{b_{1}} g(R_{1}(b_{1},b_{2},\rho,\beta),b_{1},b_{2},\rho,\beta)}{\partial_{t}g(R_{1}(b_{1},b_{2},\rho),b_{1},b_{2},\rho,\beta)}<0,$$
$$\frac{\partial_{t}R_{1}(b_{1},b_{2},\rho,\beta)}{\partial b_{2}}=-\frac{\partial_{b_{2}} g(R_{1}(b_{1},b_{2},\rho),b_{1},b_{2},\rho,\beta)}{\partial_{t}g(R_{1}(b_{1},b_{2},\rho,\beta),b_{1},b_{2},\rho,\beta)}<0.$$ So, $R_{0} $ is monotone increasing in $b_{1},b_{2} $ and $R_{1}$ is monotone decreasing in $b_{1},b_{2} $. By \eqref{LJ} and  the monotonicity of $R_{1}$, we have $2R^{2}_{0}(\overline{b}_{1}+\tau_{1},\overline{b}_{2}+\tau_{2},\rho,\beta)<R^{2}_{1}(\overline{b}_{1},\overline{b}_{2},\rho,\beta)$. For any  $d_{i}> 0,\  c_{i}>0, i=1,2 $  such that $d^{2}_{i}+c^{2}_{i}=b_{1}^{i},\ i=1,2$, then
\begin{align*}
&R^{2}_{0}(c_{1},c_{2},\rho,\beta)+ R^{2}_{0} (d_{1},d_{2},\rho,\beta)< 2  R^{2}_{0} (b_{1},b_{2},\rho,\beta)<2  R^{2}_{0}(\overline{b}_{1}+\tau_{1},\overline{b}_{2}+\tau_{2},\rho,\beta)\\
&<2  R^{2}_{1}(\overline{b}_{1},\overline{b}_{2},\rho,\beta)<R^{2}_{1}(b_{1}, b_{2},\rho,\beta).
\end{align*}
\end{proof}
\begin{proof}[{\bf Proof of Theorem \ref{Th5}}]
Suppose that there exists a $\epsilon_{0}>0$ , a sequence of initial date $\{u^{0}_{n},v^{0}_{n}\}\subset H^{1}(\mathbb{R}^{3})\times H^{1}(\mathbb{R}^{3}) $ and a sequence $\{t_{n}\}\subset \mathbb{R}^{+}$ such that the solution $(u_{n},v_{n})$ of \eqref{L}  with the initial date $(u_{n}(0,\cdot),v_{n}(0,\cdot))=(u_{n}(\cdot), v_{n}(\cdot))$ satisfies $$dist_{H^{1}(\mathbb{R}^{3})\times H^{1}(\mathbb{R}^{3})}((u^{0}_{n},v^{0}_{n}),\mathcal{M}_{b_{1},b_{2}})<\frac{1}{n}$$ and $$dist_{H^{1}(\mathbb{R}^{3})\times H^{1}(\mathbb{R}^{3})}((u_{n}(t_{n},\cdot),v_{n}(t_{n},\cdot)),\mathcal{M}_{b_{1},b_{2}})\geq\epsilon_{0}.$$ Without loss of generality, we may assume that $\{(u^{0}_{n},v^{0}_{n})\}\subset \mathrm{T}_{b_{1}}\times \mathrm{T}_{b_{2}} ,$ since $$dist_{H^{1}(\mathbb{R}^{3})\times H^{1}(\mathbb{R}^{3})}((u^{0}_{n},v^{0}_{n}),\mathcal{M}_{b_{1},b_{2}})\rightarrow 0\  \text{as}\ \ n\rightarrow +\infty.$$ So, $\|u^{0}_{n}\|_{L^{2}(\mathbb{R}^{3})}:=b_{1n}\rightarrow b_{1}$\ \ \ $\|v^{0}_{n}\|_{L^{2}(\mathbb{R}^{3})}:=b_{2n}\rightarrow b_{2}$ and $J_{\beta}(u^{0}_{n},v^{0}_{n})\rightarrow m_{\beta}(b_{1},b_{2})$. By (i) of Lemma \ref{LK} and the continuity,  we can deduce that  $$\int_{\mathbb{R}^{3}}(|\nabla u^{0}_{n}|^{2}+|\nabla v^{0}_{n}|^{2})dx< R_{0}(b_{1}+\tau_{1},b_{2}+\tau_{2},\rho,\beta)<R_{1}(b_{1n}+\tau_{1},b_{2n}+\tau_{2},\rho,\beta)$$
for every $n$ large enough.  Since $$\int_{\mathbb{R}^{3}}(|\nabla u^{0}_{n}|^{2}+|\nabla v^{0}_{n}|^{2})dx\in[ R_{0}(b_{1n},b_{2n},\rho,\beta),R_{1}(b_{1n},b_{2n},\rho,\beta)],$$
we have  $J_{\beta}(u^{0}_{n},v^{0}_{n})\geq 0$. Thus, we can deduce that $$\int_{\mathbb{R}^{3}}(|\nabla u^{0}_{n}|^{2}+|\nabla v^{0}_{n}|^{2})dx< R_{0}(b_{1n},b_{2n},\rho,\beta)< R_{0}(b_{1}+\tau_{1},b_{2}+\tau_{2},\rho,\beta).$$
Since $(u^{0}_{n},v^{0}_{n})\in A_{R_{0}(b_{1n},b_{2n},\rho,\beta)}$, if $(u_{n}(t_{n},\cdot),v_{n}(t_{n},\cdot)$ exist from $A_{R_{0}(b_{1n},b_{2n},\rho,\beta)}$ there exists $t_{n}\in (0,T_{\max})$ such that $$\int_{\mathbb{R}^{3}}(|\nabla u_{n}(t_{n},\cdot)|^{2}+|\nabla v_{n}(t_{n},\cdot)|^{2})dx=R_{0}(b_{1n},b_{2n},\rho,\beta);$$ however, $J_{\beta}(u_{n}(t_{n},\cdot),v_{n}(t_{n},\cdot)\geq h(R_{0})= 0$ contract with the conservation of energy. Therefore, the solutions starting in $A_{R_{0}(b_{1n},b_{2n},\rho,\beta)}$ are globally defined in time  and satisfy $$\int_{\mathbb{R}^{3}}(|\nabla u_{n}(t_{n},\cdot)|^{2}+|\nabla v_{n}(t_{n},\cdot)|^{2})dx<R_{0}(b_{1n},b_{2n},\rho,\beta)<R_{0}(b_{1}+\tau_{1},b_{2}+\tau_{2},\rho,\beta)$$ for every $t_{n}\in (0,+\infty)$.
Then by the conservation laws of energy and mass $$\{(u_{n}(t_{n}, \cdot), v_{n}(t_{n}, \cdot))\}$$ is a minimizing sequence.  Thus,$\{(u_{n}(t_{n}, \cdot), v_{n}(t_{n}, \cdot))\}$ is a minimizing sequence of $m^{r}_{\beta}(b_{1},b_{2}) $.
Then according the proof of Theorem \ref{Th1} there exists $(u_{0},v_{0})\in \mathcal{M}_{b_{1},b_{2}}$ such that $(u_{n}(t_{n},\cdot), v_{n}(t_{n},\cdot))\rightarrow (u_{0},v_{0})
$ in $H^{1}(\mathbb{R}^{3})\times H^{1}(\mathbb{R}^{3})$, which contradicts to $$dist_{H^{1}(\mathbb{R}^{3})\times H^{1}(\mathbb{R}^{3})}((u_{n}(t_{n},\cdot),v_{n}(t_{n},\cdot)),\mathcal{M}_{b_{1},b_{2}})\geq\epsilon_{0}.$$
\end{proof}
\section{Proof of Theorems \ref{th1.4} and \ref{th1.5}} \label{sec16}

In this section, we first give some preliminaries. For $N=4$, we have the Sobolev inequality
	$$\mathcal{S}\|u\|_{L^4(\mathbb{R}^{4})}^2\leq \|\nabla u\|_{L^2(\mathbb{R}^{4})}^2,\quad \forall u\ \in D^{1,2}(\mathbb{R}^{4}),$$
where $D^{1,2}(\mathbb{R}^{4})$ is the completion of $C_c^{\infty}(\R^4)$ with respect to the norm $||u||_{D^{1,2}}:=\|\nabla u\|_2$. For $2<p<4$, the Gagliardo-Nirenberg inequality (see \cite{MIW}) is
\begin{equation}\label{b2}
\|u\|_{L^p(\mathbb{R}^{4})}\leq  C_{4,p}\|\nabla u\|_{L^2(\mathbb{R}^{4})}^{\gamma_p}\|u\|_{L^2(\mathbb{R}^{4})}^{1-\gamma_p},\quad \forall\ u\in H^1(\mathbb{R}^{4}),
\end{equation}
where $C_{4,p}>0$ is a constant and $\gamma_p=\frac{2(p-2)}{p}$.
If $p=4$, $\gamma_{4}=1$, then $\mathcal{S}= |C_{4,p}|^{-2}$.

\vskip1mm
We have the following already known results.

\begin{lemma}\label{lem2.1}(\cite[Corollary B.1]{LIZOU})
Suppose $(u,v)\in H^1(\mathbb{R}^{N})\times H^1(\mathbb{R}^{N})$ ($N\ge 3$) is a nonnegative solution of \eqref{23}-\eqref{24}, then $(u,v)$ is a smooth solution.
\end{lemma}

\begin{lemma}\label{lem2.2}(\cite[Lemma 2.3]{LIZOU})
Suppose $\mu_1,\mu_2,\beta>0$ and $(u,v)\in H^1(\mathbb{R}^{4})\times H^1(\mathbb{R}^{4})$ is a nonnegative solution of \eqref{23}-\eqref{24}, then $u>0$ implies $\lambda_1>0$; $v>0$ implies $\lambda_2>0$.
\end{lemma}



\vskip1mm

From \cite{CZ,PPW}, we get the least energy solutions to \eqref{a6}.
Note that \eqref{a6} has semi-trivial solutions $(\mu^{-\frac{1}{2}}_1 U_{\varepsilon}, 0)$ and $(0, \mu^{-\frac{1}{2}}_2 U_{\varepsilon})$. Here, we are only interested in nontrivial solutions of \eqref{a6}, this is $(\sqrt{k_1}U_{\varepsilon},\sqrt{k_2}U_{\varepsilon})$, where $U_\varepsilon$ is defined in \eqref{a5}, $k_1=\frac{\rho-\mu_2}{\rho^2-\mu_1\mu_2}$ and $k_2=\frac{\rho-\mu_1}{\rho^2-\mu_1\mu_2}$. The main results in this aspect are summarized below. Next, we turn to the related limiting elliptic system \eqref{a6}.

\begin{lemma}\label{lem2.5}(\cite[Theorems 1.5 and 4.1]{CZ}) For $\mu_1,\mu_2>0$,
if $0<\rho<\min\{\mu_1,\mu_2\}$ or $\rho>\max\{\mu_1,\mu_2\}$, then any positive least energy solution $(u,v)$ of \eqref{a6} must be of the form
$$(u,v)=\big(\sqrt{k_1}U_{\varepsilon}, \sqrt{k_2}U_{\varepsilon}\big).$$

\end{lemma}
Solutions to \eqref{a6} correspond to critical points of the functional
\begin{equation*}
I(u,v)=\int_{\mathbb{R}^{4}}\frac{1}{2}\big(|\nabla u|^2+|\nabla v^2|\big)dx-\frac{1}{4}\big(\mu_1|u|^4+\mu_2|v|^4+2\rho|u|^{2}|v|^{2}\big)dx,\ \
\end{equation*}
$u,v\in D^{1,2}(\R^4).$ From \cite[Lemma 2.1]{PPW}, if $0<\rho<\min\{\mu_1,\mu_2\}$ or $\rho>\max\{\mu_1,\mu_2\}$, $(u_0,v_0)$ is a least energy solution of \eqref{a6}, we can deduce that
\begin{equation*}
I(u_0,v_0)=\frac{k_1+k_2}{4}\mathcal{S}^2 \ \  \text{and} \ \ \mathcal{S}_{\mu_1,\mu_2, \rho}=\sqrt{k_1+k_2}\mathcal{S}.
\end{equation*}

We first study the case $\beta<0$, $\mu_i>0(i=1,2)$ and $\rho>0$.
\vskip1mm

\noindent \textbf{Proof of Theorem 1.11.} If $\beta<0$, let $(u,v)$ be a constraint critical point of $J_{\beta}$ on $\mathrm{T}_{b_{1}}\times \mathrm{T}_{b_{2}} $ and $u,v>0$. Therefore, $(u,v)$ solves \eqref{23}-\eqref{24} for some $\lambda_1,\lambda_2\in \R$.
By the Pohozaev identity $P_{b_1,b_2}(u,v)=0$, we have
\begin{align*}\label{y1}
\lambda_1\|u\|^2_{L^2(\mathbb{R}^{4})}+\lambda_2\|v\|^2_{L^2(\mathbb{R}^{4})}&=\frac{4-N}{4}\big(\mu_1\|u\|^4_{L^4(\mathbb{R}^{4})}+\mu_2\|v\|^4_{L^4(\mathbb{R}^{4})}+2\rho\|uv\|^2_{L^2(\mathbb{R}^{4})}\big)\\
&\quad+\frac{6-N}{4}\beta\int_{\mathbb{R}^{N}}|u|^2vdx
\end{align*}
For $N=4,5$, then one of the $\lambda_1,\lambda_2$ is negative. With out of generality, we assume $\lambda_1<0$. It follows from Lemma \ref{lem2.1} that $(u,v)$ is smooth, and is in $L^{\infty}(\mathbb{R}^{N})\times L^{\infty}(\R^N)$; thus, $|\Delta u|, |\Delta v|\in L^{\infty}(\R^N)$ as well, and stand gradient estimates for the Poisson equation (see formula (3,15) in \cite{Tru}) imply that $|\nabla u|, |\nabla v|\in L^{\infty}(\mathbb{R}^{N})$. Combining the fact that $u,v\in L^2(\mathbb{R}^{N})$, we get $u(x), v(x)\to 0$ as $|x|\to \infty$. Thus, we have
\begin{equation*}
-\Delta u=\big(-\lambda_1+\beta v+\mu_1u^2+\rho v^2\big)u\ge \frac{-\lambda_1}{2} u>0 \ \
\end{equation*}
for $|x|>R_0$, with $R_0>0$ large enough, and then $u$ is superharmonic at infinity. From the Hadamard three spheres theorem \cite[Chapter 2]{ProW}, this implies that the function $m(r):= \min\limits_{|x|=r} u(x)$ satisfies
\begin{equation*}
m(r) \ge \frac{m(r_1)\big(r^{2-N}-r_2^{2-N}\big) + m(r_2)\big( r_1^{2-N}- r^{2-N}\big)}{r_1^{2-N}- r_2^{2-N}} \qquad \forall R_0 < r_1 < r < r_2.
\end{equation*}
Since $u$ decays at infinity, we have that $m(r_2) \to 0$ as $r_2 \to +\infty$, it is not difficult to see that $r \mapsto r^{N-2} m(r)$ is monotone non-decreasing for $r>R_0$. Moreover, $m(r) >0$ for every $r>0$ because $u>0$ in $\R^N$. Thus,
\begin{equation*}
m(r) \ge m(R_0) R_0^{N-2}\cdot r^{2-N} \qquad \forall r>R_0.
\end{equation*}
If $N=4$, it follows from $u\in H^1(\R^N)$ that $u\in L^{\frac{N}{N-2}}(\mathbb{R}^{N})$. We deduce that
\begin{equation*}
\begin{split}
\|u\|_{L^{\frac{N}{N-2}}(\mathbb{R}^{N})}^{\frac{N}{N-2}}
 \ge C \int_{R_0}^{+\infty} |m(r)|^{\frac{N}{N-2}} r^{N-1} dr  \ge C \int_{R_0}^{+\infty} \frac{1}{r} dr=+\infty,
\end{split}
\end{equation*}
with $C>0$.  This is a contradiction. If $N=5$, the fact that $u\in H^1(\R^N)$ does not imply that $u\in L^{\frac{N}{N-2}}(\R^N)$ or that $u\in L^p(\R^N)$ for some $p\in (0,\frac{N}{N-2}]$. But, imposing such condition as an assumption, we still reach a contradiction.
\vskip1mm

If $N\ge 4$, and $u\in H^1(\R^N)$ is a radial function by \cite[Radial Lemma A.II]{BHL}, there exist $C>0$ and $R_1>0$ such that
\begin{equation*}
|u(x)|\le C |x|^{-\frac{N-1}{2}}\ \ \text{for}\ \ |x|\ge R_1.
\end{equation*}
If $\beta<0$, let $(u,v)$ be a non-trivial radial solution of \eqref{23}-\eqref{24}.
Similarly, we assume $\lambda_1<0$.
Setting $p(x)=-\beta v-\mu_1u^2-\rho v^2$, then
\begin{equation}\label{m2}
-\Delta u+p(x)u= -\lambda_1 u.
\end{equation}
Therefore, for $N\ge4$,
\begin{equation*}
\lim_{n\to +\infty}|x||p(x)|\le \lim_{n\to +\infty}\big[C|x|^{\frac{3-N}{2}}+C|x|^{2-N} \big]=0.
\end{equation*}
By Kato's result \cite{KT}, i.e. Schr\"{o}dinger operator $H = -\Delta+ p(x)$ has no positive eigenvalue with an $L^2$-eigenfunction if $p(x) = o(|x|^{-1})$, then \eqref{m2} has no solution. We get \eqref{23}-\eqref{24} has no non-trivial radial solution for $\beta<0$.
\qed

\begin{lemma}\label{lem3.1} If $0<\beta b_1<\!\frac{3}{2|C_{4,3}|^{3}}$ and $0<\beta b_2<\!\frac{3}{|C_{4,3}|^{3}}$, for all $(u,v)\in \mathrm{T}_{b_{1}}\times \mathrm{T}_{b_{2}} $, there exists $t_{(u,v)}$ such that $t_{(u,v)}\star(u,v)\in\mathcal{P}_{b_1,b_2}$. $t_{(u,v)}$ is the unique critical point of the function $\Psi_{u,v}$ and is a strict maximum point at positive level.
Moreover:
\vskip1mm
\noindent $(1)$ $\Psi_{u,v}''(0)<0$ and $P_{b_1,b_2}(u,v)<0$ iff $t_{(u,v)}<0$.\\
\noindent $(2)$ $\Psi_{u,v}$ is strictly increasing in $(-\infty,t_{(u,v)})$. \\
\noindent $(3)$ The map $(u,v)\mapsto t_{(u,v)} \in \mathbb{R}$ is of class $C^1$.
\end{lemma}
\begin{proof}
We are therefore considering the defocusing nonlinearity mass-critical. We rewrite \eqref{pp} as
\begin{equation*}
\begin{aligned}
\Psi_{u,v}(s) &=\frac{e^{2s}}{2}\int_{\mathbb{R}^{4}}\big(|\nabla u|^2+|\nabla v|^2-\beta|u|^2v\big)dx-\frac{e^{4s}}{4}\int_{\mathbb{R}^{4}}\big(\mu_1|u|^4+\mu_2|v|^4+2\rho|u|^2|v|^{2}\big)dx\\
&\ge \frac{e^{2s}}{2}\Big[\|\nabla u\|^2_{L^2(\R^4)}\big(1-\frac{2\beta|C_{4,3}|^3}{3}b_1\big)
+\|\nabla v\|^2_{L^2(\R^4)}\big(1-\frac{\beta|C_{4,3}|^3}{3}b_2\big)\Big]\\
&\quad-\frac{e^{4s}}{4}\Big(\mu_1\|u\|^4_{L^4(\R^4)}+\mu_2\|v\|^4_{L^4(\R^4)}+2\rho\|uv\|^{2}_{L^2(\R^4)}\Big).\\
\end{aligned}
\end{equation*}
$s\star (u,v)\in \mathcal{P}_{b_1,b_2}$ if and only if $\Psi'_{u,v}(s)=0$, it's easy to see that if $\big[\|\nabla u\|^2_{L^2(\R^4)}\big(1-\frac{2\beta|C_{4,3}|^3}{3}b_1\big)
+\|\nabla v\|^2_{L^2(\R^4)}\big(1-\frac{\beta|C_{4,3}|^3}{3}b_2\big)\big]$ is positive, then $\Psi_{u,v}(s)$ has a unique critical point $t_{(u,v)}$, which is a strict maximum point at positive level. Therefore, under the condition of $0<b_1<\!\frac{3}{2\beta|C_{4,3}|^{3}}$ and $0<b_2<\!\frac{3}{\beta|C_{4,3}|^{3}}$, we know that $\int_{\mathbb{R}^{4}}\big(|\nabla u|^2+|\nabla v|^2-\beta|u|^2v\big)dx>0$. If $(u,v)\in \mathcal{ P}_{b_1,b_2}$, then $t_{(u,v)}$ is a maximum point, we have that $\Psi''_{u,v}(t_{(u,v)})\le 0$. We claim that $\Psi''_{u,v}(t_{(u,v)})< 0$. By contradiction, this is $\Psi'_{u,v}(0)=\Psi''_{u,v}(0)=0$, then necessarily $\int_{\mathbb{R}^{4}} \big(\mu_1|u|^4+\mu_2|v|^4+2\rho|u|^2|v|^2\big)dx=0$, which is not possible because $(u,v)\in \mathrm{T}_{b_{1}}\times \mathrm{T}_{b_{2}} $. Thus, $\Psi_{u,v}''(0)<0$.

\vskip1mm

As in the proof of Lemma \ref{Lem4} shows that the map $(u,v)\in \mathrm{T}_{b_{1}}\times \mathrm{T}_{b_{2}} \mapsto t_{(u,v)}\in \R$ is of class $C^1$. Finally, $\Psi'_{u,v}(s)<0$ if and only if $s>t_{(u,v)}$, then $P_{b_1,b_2}(u,v)=\Psi'_{u,v}(0)<0$ if and only if $t_{(u,v)}<0$.
\end{proof}

\begin{lemma}\label{lem3.2}If $0<\beta b_1<\!\frac{3}{2|C_{4,3}|^{3}}$ and $0<\beta b_2<\!\frac{3}{|C_{4,3}|^{3}}$, then the set $\mathcal{P}_{b_1,b_2}$ is a $C^1$-submanifold of codimension 1 in $\mathrm{T}_{b_{1}}\times \mathrm{T}_{b_{2}} $, and it is a $C^1$-submanifold of codimension 3 in $H^1(\mathbb{R}^{4})\times H^1(\mathbb{R}^{4})$.
\end{lemma}
\begin{proof}
The proof is similar to that of Lemma \ref{lem4.3}.
\end{proof}

\begin{lemma} \label{lem3.3}
Assume $\mu_i,b_i>0(i=1,2)$ and $\rho,\beta>0$. Let $0<\beta b_1<\!\frac{3}{2|C_{4,3}|^{3}}$ and $0<\beta b_2<\!\frac{3}{|C_{4,3}|^{3}}$, then
\begin{equation*}
m_{\beta}(b_1,b_2):=\inf_{(u,v)\in \mathcal{ P}_{b_1,b_2}}J_{\beta}(u,v)>0.\\
\end{equation*}
\end{lemma}
\begin{proof}
If $(u,v)\in \mathcal{P}_{b_1,b_2}$, then by Gagliardo-Nirenberg and the Sobolev inequalities, we have
\begin{equation*}
\begin{aligned}
&\|\nabla u\|^2_{L^2(\mathbb{R}^{4})}+\|\nabla v\|^2_{L^2(\mathbb{R}^{4})}\\
&=\mu_1\|u\|^4_{L^4(\mathbb{R}^{4})}+\mu_2\|v\|^4_{L^4(\mathbb{R}^{4})}+2\rho\|uv\|^2_{L^2(\mathbb{R}^{4})}+\beta\int_{\mathbb{R}^{4}}|u|^2vdx\\
&\le \mathcal{S}^2_{\mu_1,\mu_2,\rho}\big[\|\nabla u\|^2_{L^2(\mathbb{R}^{4})}+\|\nabla v\|^2_{L^2(\mathbb{R}^{4})}\big]^2+\frac{2\beta|C_{4,3}|^3}{3}b_1\|\nabla u\|^2_{L^2(\mathbb{R}^{4})}+\frac{\beta|C_{4,3}|^3}{3}b_2\|\nabla v\|^2_{L^2(\mathbb{R}^{4})}.\\
\end{aligned}
\end{equation*}
Moreover, $0<\beta b_1<\!\frac{3}{2|C_{4,3}|^{3}}$ and $0<\beta b_2<\!\frac{3}{|C_{4,3}|^{3}}$, and $\|\nabla u\|^2_{L^2(\R^4)}+\|\nabla v\|^2_{L^2(\R^4)}\neq 0$ (since $(u,v)\in \mathrm{T}_{b_{1}}\times \mathrm{T}_{b_{2}} $), we get
\begin{equation*}
\inf_{(u,v)\in \mathcal{P}_{b_1,b_2}}\|\nabla u\|^2_{L^2(\mathbb{R}^{4})}+\|\nabla v\|^2_{L^2(\mathbb{R}^{4})}\ge C>0.
\end{equation*}
So
\begin{equation*}
\begin{aligned}
m_{\beta}(b_1,b_2)&=\inf_{(u,v)\in \mathcal{P}_{b_1,b_2}}J_{\beta}(u,v)\\
&=\inf_{(u,v)\in \mathcal{P}_{b_1,b_2}}\frac{1}{4}\big[\|\nabla u\|^2_{L^2(\mathbb{R}^{4})}+\|\nabla v\|^2_{L^2(\mathbb{R}^{4})}-\beta\int_{\mathbb{R}^{4}}|u|^2vdx\big]\\
&\ge \inf_{(u,v)\in \mathcal{P}_{b_1,b_2}}\frac{1}{4}\big(1-\frac{2\beta|C_{4,3}|^3b_1}{3}\big)\|\nabla u\|^2_{L^2(\mathbb{R}^{4})}+\frac{1}{4}\big(1-\frac{\beta|C_{4,3}|^3b_2}{3}\big)\|\nabla v\|^2_{L^2(\mathbb{R}^{4})}\\
&\ge C>0.
\end{aligned}
\end{equation*}
\end{proof}

From the above lemmas, then $J_{\beta}\big|_{\mathrm{T}_{b_{1}}\times \mathrm{T}_{b_{2}} }$ has a mountain pass geometry. We need an estimate from above on $m_{\beta}(b_1,b_2)=\inf\limits_{(u,v)\in\mathcal{P}_{b_1,b_2}}J_{\beta}(u,v)$.

\begin{lemma}\label{lem3.4}
Let $\beta,\mu_1,\mu_2>0$, $\rho\in \big(0,\min\{\mu_1,\mu_2\}\big) \cup\big(\max\{\mu_1,\mu_2\},\infty\big)$. If $0<\beta b_1<\!\frac{3}{2|C_{4,3}|^{3}}$ and $0<\beta b_2<\!\frac{3}{|C_{4,3}|^{3}}$, then $$0<m_{\beta}(b_1,b_2)<\frac{k_1+k_2}{4}\mathcal{S}^2,$$
where $k_1=\frac{\rho-\mu_2}{\rho^2-\mu_1\mu_2}$ and $k_2=\frac{\rho-\mu_1}{\rho^2-\mu_1\mu_2}$.
\end{lemma}
\begin{proof}
From \eqref{a5}, $U_{\varepsilon}=\frac{2\sqrt{2}\varepsilon}{\varepsilon^2+|x|^2}$, taking a radially decreasing cut-off function $\xi \in C_0^{\infty}(\mathbb{R}^{4})$ such that $\xi \equiv 1$ in $B_1$, $\xi \equiv 0$ in $\R^4 \backslash B_2$, and let $W_{\varepsilon}(x) = \xi(x) U_{\varepsilon}(x)$. We have (see \cite[Lemma 7.1]{JTT}),
\begin{equation}\label{ab1}
\|\nabla W_\varepsilon\|^2_{L^2(\mathbb{R}^{4})}=\mathcal{S}^2+O(\varepsilon^2),    \quad \quad \|W_\varepsilon\|^4_{L^4(\mathbb{R}^{4})}=\mathcal{S}^2+O(\varepsilon^4),
\end{equation}
and
\begin{equation}\label{ab2}
 \| W_\varepsilon\|^3_{L^3(\mathbb{R}^{4})}=O(\varepsilon), \quad \quad \|W_\varepsilon\|^2_{L^2(\mathbb{R}^{4})}=O(\varepsilon^2|\ln\varepsilon|).
\end{equation}
Setting
\begin{equation*}
\big(\overline{W}_{\varepsilon},\overline{V}_{\varepsilon}\big)=\Big(a_1\frac{W_{\varepsilon}}{\|W_{\varepsilon}\|_{L^2(\mathbb{R}^{4})}},a_2\frac{W_{\varepsilon}}{\|W_{\varepsilon}\|_{L^2(\mathbb{R}^{4})}}\Big), \end{equation*}
then $(\overline{W}_{\varepsilon},\overline{V}_{\varepsilon})\in \mathrm{T}_{b_{1}}\times \mathrm{T}_{b_{2}} $. From Lemma \ref{lem3.1}, there exists $\tau_{\varepsilon}\in\R$ such that $\tau_{\varepsilon}\star(\overline{W}_{\varepsilon},\overline{V}_{\varepsilon})\in \mathcal{P}_{b_1,b_2}$,
this implies that,
\begin{equation*}
m_{\beta}(b_1,b_2)\le J_{\beta}\big(\tau_{\varepsilon}\star(\overline{W}_{\varepsilon},\overline{V}_{\varepsilon})\big)=\max_{t\in \R}J_{\beta}\big(t\star(\overline{W}_{\varepsilon},\overline{V}_{\varepsilon})\big),
\end{equation*}
and
\begin{equation}\label{ab3}
\begin{aligned}
\big[\|\nabla \overline{W}_{\varepsilon}\|_{L^2(\mathbb{R}^{4})}^2+\|\nabla \overline{V}_{\varepsilon}\|_{L^2(\mathbb{R}^{4})}^2\big]&=e^{2\tau_{\varepsilon}}\big[\mu_1\|\overline{W}_{\varepsilon}\|_{L^4(\mathbb{R}^{4})}^4+\mu_2\|\overline{V}_{\varepsilon}\|_{L^4(\mathbb{R}^{4})}^4\\
&\quad+2\rho\|\overline{W}_{\varepsilon}\overline{V}_{\varepsilon}\|^2_{L^2(\mathbb{R}^{4})}\big]+\int_{\mathbb{R}^{4}}\beta|\overline{W}_{\varepsilon}|^2\overline{V}_{\varepsilon}dx.\\
\end{aligned}
\end{equation}
Then,
\begin{equation*}
\begin{aligned}
e^{2\tau_{\varepsilon}}&\big[\mu_1\|\overline{W}_{\varepsilon}\|_{L^4(\mathbb{R}^{4})}^4+\mu_2\|\overline{V}_{\varepsilon}\|_{L^4(\mathbb{R}^{4})}^4+2\rho\|\overline{W}_{\varepsilon}\overline{V}_{\varepsilon}\|^2_{L^2(\mathbb{R}^{4})}\big]\\
&=\big[\|\nabla \overline{W}_{\varepsilon}\|_{L^2(\mathbb{R}^{4})}^2+\|\nabla \overline{V}_{\varepsilon}\|_{L^2(\mathbb{R}^{4})}^2\big]-\int_{\mathbb{R}^{4}}\beta|\overline{W}_{\varepsilon}|^2\overline{V}_{\varepsilon}dx\\
&\ge \big(1-\frac{2\beta|C_{4,3}|^3b_1}{3}\big)\|\nabla \overline{W}_{\varepsilon}\|^2_{L^2(\mathbb{R}^{4})}+\big(1-\frac{\beta|C_{4,3}|^3b_2}{3}\big)\|\nabla \overline{V}_{\varepsilon}\|^2_{L^2(\mathbb{R}^{4})}.\\
\end{aligned}
\end{equation*}
Therefore, for $\varepsilon$ small enough, we have
\begin{equation*}
\begin{aligned}
J_{\beta}\big(\tau_{\varepsilon}\star(\overline{W}_{\varepsilon},\overline{V}_{\varepsilon})\big)&=\frac{e^{2\tau_{\varepsilon}}}{2}\big(\|\nabla \overline{W}_{\varepsilon}\|_{L^2(\mathbb{R}^{4})}^2+\|\nabla \overline{V}_{\varepsilon}\|^2_{L^2(\mathbb{R}^{4})}\big)-\frac{e^{2\tau_{\varepsilon}}}{2}\int_{\mathbb{R}^{4}}|\overline{W}_{\varepsilon}|^{2}\overline{V}_{\varepsilon}dx\\
&\quad-\frac{e^{4\tau_{\varepsilon}}}{4}\big(\mu_1\|\overline{W}_{\varepsilon}\|_{L^4(\mathbb{R}^{4})}^4+\mu_2\|\overline{V}_{\varepsilon}\|_{L^4(\mathbb{R}^{4})}^4+2\rho\|\overline{W}_{\varepsilon}\overline{V}_{\varepsilon}\|^2_{L^2(\mathbb{R}^{4})}\big)\\
&\le \max_{s_1,s_2>0}\frac{s^2_1+s^2_2}{2}\|\nabla W_{\varepsilon}\|_{L^2(\mathbb{R}^{4})}^2-\frac{\big(\mu_1s^4_1+\mu_2s^4_2+2\rho s^2_1s^2_2\big)}{4}\|W_{\varepsilon}\|_{L^4(\mathbb{R}^{4})}^4\\
&\quad-\frac{e^{2\tau_{\varepsilon}}}{2}\int_{\mathbb{R}^{4}}|\overline{W}_{\varepsilon}|^{2}\overline{V}_{\varepsilon}dx,\\
\end{aligned}
\end{equation*}
where $s_1=\frac{e^{\tau_{\varepsilon}}b_1}{\|W_{\varepsilon}\|_{L^2(\mathbb{R}^{4})}}$ and $s_2=\frac{e^{\tau_{\varepsilon}}b_2}{\|W_{\varepsilon}\|_{L^2(\mathbb{R}^{4})}}$.
Define $$f(s_1,s_2)=\frac{s^2_1+s^2_2}{2}\|\nabla W_{\varepsilon}\|_{L^2(\mathbb{R}^{4})}^2-\frac{(\mu_1s^4_1+\mu_2s^4_2+2\rho s^2_1s^2_2)}{4}\|W_{\varepsilon}\|_{L^4(\mathbb{R}^{4})}^4.$$ Using that, for all $0<s_1,s_2$,
\begin{equation*}
\max_{s_1,s_2>0}f(s_1,s_2)\le \frac{k_1+k_2}{4}\mathcal{S}^2+O(\varepsilon^2).
\end{equation*}
Finally,
\begin{equation*}
\begin{aligned}
\frac{e^{2\tau_{\varepsilon}}}{2}\int_{\mathbb{R}^{4}}|\overline{W}_{\varepsilon}|^{2}\overline{V}_{\varepsilon}dx&=\frac{e^{2\tau_{\varepsilon}}}{2}\frac{b^2_1b_2}{\|W_\varepsilon\|^3_{L^2(\mathbb{R}^{4})}}\int_{\mathbb{R}^{4}}|W_{\varepsilon}|^{3}dx\\
&\ge \frac{C}{\|W_{\varepsilon}\|_{L^2(\mathbb{R}^{4})}} \int_{\mathbb{R}^{4}}|W_{\varepsilon}|^{3}dx\\
&\ge C|\ln\varepsilon|^{-\frac{1}{2}},
\end{aligned}
\end{equation*}
hence, we deduce that
\begin{equation*}
\max_{t\in \R}J_{\beta}\big(t\star(\overline{W}_{\varepsilon},\overline{V}_{\varepsilon})\big)<\frac{k_1+k_2}{4}\mathcal{S}^2.
\end{equation*}
\end{proof}

Now, we are ready to show that the infimum is attained by nontrivial positive radial functions.
\begin{lemma}\label{lem3.5}
Let $\mu_i,b_i>0(i=1,2)$, and $\rho\in \big(0,\min\{\mu_1,\mu_2\}\big)\cup\big(\max\{\mu_1,\mu_2\},\infty\big)$.
If $0<\beta b_1<\!\frac{3}{2|C_{4,3}|^{3}}$, $0<\beta b_2<\!\frac{3}{|C_{4,3}|^{3}}$ and
\begin{equation*}
0<m_{\beta}(b_1,b_2)<\frac{k_1+k_2}{4}\mathcal{S}^2,
\end{equation*}
where $k_1=\frac{\rho-\mu_2}{\rho^2-\mu_1\mu_2}$ and $k_2=\frac{\rho-\mu_1}{\rho^2-\mu_1\mu_2}$, then $m_{\beta}(b_1,b_2)$ can be achieved by some function $(u_{b_1,b_{2}},v_{b_{1}, b_2})\in \mathrm{T}_{b_{1}}\times \mathrm{T}_{b_{2}} $ which is real valued, positive, radially symmetric and radially decreasing.
\end{lemma}
\begin{proof}
By Lemma \ref{Lem2}, we can also find a radial Palais-Smale sequence for $J_\beta\big|_{\mathrm{T}_{b_{1}}\times \mathrm{T}_{b_{2}} }$ at level $m_{\beta}(b_1,b_2)$ such that $P(u_n,v_n)\to 0$ and $u^{-}_n,v^{-}_n\to 0$ a.e. in $\R^4$. The rest of proof is similar to that of Lemma 5.5 in \cite{LYZ}, and just needs a slight modification..
\end{proof}

\begin{lemma}\label{lem3.6}
Let $\mu_i,\beta>0(i=1,2)$, and $\rho\in \big(0,\min\{\mu_1,\mu_2\}\big)\cup\big(\max\{\mu_1,\mu_2\},\infty\big)$.
If $0<b_1<\frac{3}{2\beta|C_{4,3}|^3}$ and $0<b_2<\frac{3}{\beta|C_{4,3}|^3}$, for any ground state $(u_{b_1,b_{2}},v_{b_{1},b_2})$ of \eqref{23}-\eqref{24}, then
\begin{equation*}
J_{\beta}(u_{b_1,b_{2}},v_{b_1,b_{2}})=m_{\beta}(b_1,b_2)\to\frac{k_1+k_2}{4}\mathcal{S}^2\ \ \text{as} \ \ (b_1,b_2)\to (0,0).
\end{equation*}
Moreover, there exists $\varepsilon_{b_1,b_2}>0$ such that
\begin{equation*}
\big(\sigma_1 u_{b_1,b_{2}}(\sigma_1 x),\sigma_1 v_{b_1,b_{2}}(\sigma_1 x)\big)\to (\sqrt{k_1}U_{\varepsilon_0},\sqrt{k_2}U_{\varepsilon_0})\ \ \text{in} \ D^{1,2}(\mathbb{R}^{4})\times D^{1,2}(\mathbb{R}^{4}),
\end{equation*}
for some $\varepsilon_0>0$ as $(b_1,b_2)\to (0,0)$ up to a subsequence, where
$k_1=\frac{\rho-\mu_2}{\rho^2-\mu_1\mu_2}$ and $k_2=\frac{\rho-\mu_1}{\rho^2-\mu_1\mu_2}$.
\end{lemma}
\begin{proof}
By $P_{b_1,b_2}(u_{b_1},v_{b_2})=0$, we have
\begin{equation*}
\begin{aligned}
\frac{k_1+k_2}{4}\mathcal{S}^2&>J_{\beta}(u_{b_1,b_{2}},v_{b_{1},b_2})=\frac{1}{4}\big(\|\nabla u_{b_1,b_{2}}\|^2_{L^2(\mathbb{R}^{4})}+\|\nabla v_{b_1,b_{2}}\|^2_{L^2(\mathbb{R}^{4})}-\beta\int_{\mathbb{R}^{4}}|u_{b_1,b_{2}}|^2v_{b_1,b_{2}} dx \big)\\
&\ge \frac{1}{4}\big(1-\frac{2\beta b_1}{3}|C_{4,3}|^3\big)\|\nabla u_{b_1,b_{2}}\|^2_{L^2(\mathbb{R}^{4})}+\frac{1}{4}\big(1-\frac{\beta b_2}{3}|C_{4,3}|^3\big)\|\nabla v_{b_1,b_{2}}\|^2_{L^2(\mathbb{R}^{4})},
\end{aligned}
\end{equation*}
then $\{(u_{b_1,b_{2}},v_{b_1,b_{2}})\}$ is bounded in $H^1(\mathbb{R}^{4})\times H^1(\mathbb{R}^{4})$.  Using the H\"{o}lder and Gagliardo-Nirenberg inequality again,
\begin{equation*}
\beta\int_{\mathbb{R}^{4}}|u_{b_1,b_{2}}|^2v_{b_1,b_{2}}dx\le \beta|C_{4,3}|^{3}\|\nabla u_{b_1,b_{2}}\|^{\frac{4}{3}}_{L^2(\mathbb{R}^{4})}\|\nabla v_{b_1,b_{2}}\|^{\frac{2}{3}}_{L^2(\mathbb{R}^{4})}b^2_1b_2\to 0,
\end{equation*}
as $(b_1,b_2)\to (0,0)$. Therefore, it follows from $P(u_{b_1,b_{2}},v_{b_1,b_{2}})=0$ that
\begin{align*}
\|\nabla u_{b_1,b_{2}}\|^2_{L^2(\mathbb{R}^{4})}+\|\nabla v_{b_1,b_{2}}\|^2_{L^2(\mathbb{R}^{4})}=\mu_1\|u_{b_1,b_{2}}\|^4_{L^4(\mathbb{R}^{4})}+\mu_2\|v_{b_1,b_{2}}\|^4_{L^4(\mathbb{R}^{4})}+2\rho\|u_{b_1,b_{2}}v_{b_1,b_{2}}\|^2_{L^2(\mathbb{R}^{4})}+o(1).
\end{align*}
From \eqref{a7}, we have
\begin{align*}
\sqrt{k_1+k_2}\mathcal{S}\big(\mu_1\|u_{b_1,b_{2}}\|^{4}_{L^4(\mathbb{R}^{4})}+\mu_2\|v_{b_1,b_{2}}\|^{4}_{L^4(\mathbb{R}^{4})}+2\beta\|u_{b_1,b_{2}}v_{b_1,b_{2}}\|^2_{L^2(\mathbb{R}^{4})}\big)^{\frac{1}{2}}
\le ||\nabla u_{b_1,b_{2}}||^2_{L^2(\mathbb{R}^{4})}+\|\nabla v_{b_1,b_{2}}\|^2_{L^2(\mathbb{R}^{4})}.
\end{align*}
Thus, we distinguish the two cases
\begin{align*}
\text{either}\ \ (i)\ \ ||\nabla u_{b_1,b_{2}}||^2_{L^2(\mathbb{R}^{4})}+\|\nabla v_{b_1,b_{2}}\|^2_{L^2(\mathbb{R}^{4})}\to 0\ \ \text{or}(ii) \ \ ||\nabla u_{b_1,b_{2}}||^2_{L^2(\mathbb{R}^{4})}+\|\nabla v_{b_1,b_{2}}\|^2_{L^2(\mathbb{R}^{4})}\to l>0.
\end{align*}
We claim that $(i)$ is impossible. Indeed, if $l=0$, by $\Psi''_{u_{b_1,b_{2}},v_{b_1,b_{2}}}(0)<0$,
\begin{align*}
\|\nabla u_{b_1,b_{2}}\|^2_{L^2(\mathbb{R}^{4})}+\|\nabla v_{b_1,b_{2}}\|^2_{L^2(\mathbb{R}^{4})}
&<2\big[\mu_1\|u_{b_1,b_{2}}\|^4_{L^4(\mathbb{R}^{4})}+\mu_2\|v_{b_1,b_{2}}\|^4_{L^4(\mathbb{R}^{4})}\\
&\quad +2\rho\|u_{b_1,b_{2}}v_{b_1,b_{2}}\|^2_{L^2(\mathbb{R}^{4})} \big]+\int_{\mathbb{R}^{4}}\beta|u_{b_1,b_{2}}|^2v_{b_1,b_{2}}dx,
\end{align*}
we obtain a contradiction
\begin{align*}
&\big(1-\frac{2\beta b_1}{3}|C_{4,3}|^3\big)\|\nabla u_{b_1,b_{2}}\|^2_{L^2(\mathbb{R}^{4})}+\big(1-\frac{\beta b_2}{3}|C_{4,3}|^3\big)\|\nabla v_{b_1,b_{2}}\|^2_{L^2(\mathbb{R}^{4})}\\
&<\frac{2}{\mathcal{S}^2_{\mu_1,\mu_2,\rho}}\big[\|\nabla u_{b_1,b_{2}}\|^2_{L^2(\mathbb{R}^{4})}+\|\nabla v_{b_1,b_{2}}\|^2_{L^2(\mathbb{R}^{4})}\big]^2.
\end{align*}
The claim is proved. Therefore, $l>0$, we obtain that $J_{\beta}(u_{b_1,b_{2}},v_{b_1,b_{2}})\to \frac{k_1+k_2}{4}\mathcal{S}^2$ as $(b_1,b_2)\to (0,0)$.

Thus, we know that
\begin{align*}
&||\nabla u_{b_1,b_{2}}||^2_{L^2(\mathbb{R}^{4})}+\|\nabla v_{b_1,b_{2}}\|^2_{L^2(\mathbb{R}^{4})}\to (k_1+k_2)\mathcal{S}^2,\\ &\mu_1\|u_{b_1,b_{2}}\|^{4}_{L^4(\mathbb{R}^{4})}+\mu_2\|v_{b_1,b_{2}}\|^{4}_{L^4(\mathbb{R}^{4})}+2\beta\|u_{b_1,b_{2}}v_{b_1,b_{2}}\|^2_{L^2(\mathbb{R}^{4})}\to (k_1+k_2)\mathcal{S}^2,
\end{align*}
as $(b_1,b_2)\to (0,0)$. It follows that, up to a subsequence, $\{(u_{b_1,b_{2}},v_{b_1,b_{2}})\}$ is a minimizing sequence of the minimizing problem \eqref{a7}. From Lemma \ref{lem3.5}, $(u_{b_1,b_{2}},v_{b_1,b_{2}})$ is radially symmetric. By \cite[Theorem 1.41]{WM} or \cite[Lemma 3.5]{LL}, up to a subsequence, there exists $\sigma_1:=\sigma_1(a_1,a_2)$ such that for some $\varepsilon_0>0$,
\begin{equation*}
\big(\sigma_1 u_{b_1,b_{2}}(\sigma_1 x),\sigma_1 v_{b_1,b_{2}}(\sigma_1 x)\big)\to \big(\sqrt{k_1}U_{\varepsilon_0},\sqrt{k_2}U_{\varepsilon_0}\big)
\end{equation*}
in $D^{1,2}(\mathbb{R}^{4})\times D^{1,2}(\mathbb{R}^{4})$, as $(b_1,b_2)\to (0,0)$.
\end{proof}

\noindent \textbf{Proof of Theorem 1.5.}
The proof is finished when we combine Lemma \ref{lem3.5} and Lemma \ref{lem3.6}.
\qed

{\bf Acknowledgements: }
 The research of J. Wei is partially supported by NSERC of Canada. X. Luo was supported by the
National Natural Science Foundation of China (Grant No. 11901147) and the Fundamental Research Funds for the Central University of China(Grant No. JZ2020HGTB0030). Maoding Zhen was supported by the Fundamental Research Funds for the Central Universities (Grant No. JZ2021HGTA0177).

\end{document}